\documentclass[11pt,leqno,a4paper]{article}
\usepackage{mathrsfs,amssymb,amsmath,amsthm, color}
\usepackage[dvips]{graphicx}
\makeatletter
	
	\@addtoreset{equation}{section}
\makeatother
\allowdisplaybreaks[4]
\begin{document}

\renewcommand{\theenumi}{\rm (\roman{enumi})}
\renewcommand{\labelenumi}{\rm \theenumi}

\newtheorem{thm}{Theorem}[section]
\newtheorem{hypo}{Hypothesis}
\newtheorem{defi}[thm]{Definition}
\newtheorem{lem}[thm]{Lemma}
\newtheorem{prop}[thm]{Proposition}
\newtheorem{cor}[thm]{Corollary}
\newtheorem{exam}[thm]{Example}
\newtheorem{asmp}[thm]{Assumption}
\newtheorem{rem}[thm]{Remark}

\newcommand{\id}{\mathop{\text{\rm id}}}
\newcommand{\supp}{\mathop{\text{\rm supp}}}
\newcommand{\esssup}{\mathop{\rm ess\,sup}}
\newcommand{\dd}{\rm{d}}
\newcommand{\bsC}{\boldsymbol{C}}
\newcommand{\bsG}{\boldsymbol{G}}

\newcommand{\GFF}{\mathbb{X}}
\newcommand{\GMC}{\mathbb{M}}

\title{Stochastic quantization associated with the $\exp(\Phi)_2$-quantum 
field model driven by space-time white noise on the torus
\\
in the full $L^1$-regime
}

\author{Masato Hoshino\hspace{0.9mm}$^{a)}$, Hiroshi Kawabi\hspace{0.9mm}$^{b)}$ and 
Seiichiro Kusuoka\hspace{0.9mm}$^{c)}$
\vspace{5mm}\\
\normalsize $^{a)}$ Graduate School of Engineering Science, Osaka University,\\
\normalsize 1-3 Machikaneyama, Toyonaka, Osaka 560-8531, Japan\\
\normalsize e-mail address: {\tt{hoshino@sigmath.es.osaka-u.ac.jp}} \vspace{5mm}\\
\normalsize $^{b)}$ Department of Mathematics, Hiyoshi Campus, Keio University,\\
\normalsize  4-1-1 Hiyoshi, Kohoku-ku, Yokohama 223-8521, Japan \\
\normalsize e-mail address: {\tt{kawabi@keio.jp}} \vspace{5mm}\\
\normalsize $^{c)}$ Department of Mathematics, Graduate School of Science, Kyoto University,\\
\normalsize Kitashirakawa Oiwakecho, Sakyo-ku, Kyoto 606-8502, Japan\\
\normalsize e-mail address: {\tt{kusuoka@math.kyoto-u.ac.jp}}}
\maketitle
\begin{center}
{\it{Dedicated to Professor Shigeki Aida on the occasion of his 60th birthday}}
\end{center}
\begin{abstract}
The present paper is a continuation of our previous work \cite{HKK19}
on the stochastic quantization of the $\exp(\Phi)_2$-quantum 
field model
on the two-dimensional torus. 
Making use of key properties of
Gaussian multiplicative chaos and 
refining the method 
for singular SPDEs introduced
in the previous work,
 we construct a unique time-global solution to the corresponding parabolic 
stochastic quantization equation
in the full ``$L^{1}$-regime" $\vert\alpha\vert<\sqrt{8\pi}$ of 
the charge parameter $\alpha$.
We also identify the solution with an infinite-dimensional diffusion 
process constructed by the Dirichlet form approach.
\end{abstract}


\vspace{2mm}

\phantom{\indent{\bf{Keywords:}}}~Singular SPDE, Dirichlet form.
%
\section{Introduction}
\subsection{Background}
In the present paper, we study
stochastic quantization  
associated with space-time quantum fields with interactions of exponential type, called 
the {\it{$\exp(\Phi)_2$-quantum field model}} 
in the Euclidean quantum field theory, in finite volume. 
The $\exp(\Phi)_2$-quantum field (or the $\exp(\Phi)_2$-measure) $\mu^{(\alpha)}$
is a probability measure on ${\mathcal D}'(\Lambda)$, the space of distributions on
the two-dimensional torus
$\Lambda={\mathbb T}^{2}=
(\mathbb R/ 2\pi \mathbb Z)^{2}$, which is given by 
\begin{equation*}
\mu^{(\alpha)}
(d\phi)=\frac{1}{Z^{(\alpha)}}
\exp \Big (
- \int_{\Lambda} \exp^{\diamond} ( \alpha \phi)(x) 
\hspace{0.5mm} {\dd}x \Big )
\mu_{0}({\dd}\phi),
\label{exp-Gibbs}
\end{equation*}
where the massive Gaussian free field
$\mu_{0}$ is 
the Gaussian measure on ${\mathcal D}'(\Lambda)$
with zero mean and the covariance operator $(1-\triangle)^{-1}$,
$\triangle$ being the Laplacian in $L^{2}(\Lambda)$ 
with the periodic boundary conditions, 
$\alpha (\in \mathbb R)$ is called the {\it{charge parameter}},  
the {\it{Wick exponential}} 
$\exp^{\diamond} ( \alpha \phi)(x)$
is formally introduced by the expression
$$ \exp ^{\diamond} (\alpha \phi)(x)
\hspace{0.5mm}
=\exp \Big( \alpha \phi(x)
-\frac{\alpha^{2}}{2} {\mathbb E}^{\mu_{0}} [ \phi(x)^{2}] \Big), \qquad x\in \Lambda,
$$
and
\begin{equation}
Z^{(\alpha)}
=\int_{ {\mathcal D}'(\Lambda)} \exp 
\Big(-\int_{\Lambda} 
\exp ^{\diamond} (\alpha \phi)(x) 
\hspace{0.5mm} {\dd}x \Big) \mu_{0}({\dd}\phi) >0
\nonumber
\end{equation}
is the normalizing constant.
We remark that the diverging term ${\mathbb E}^{\mu_{0}} [ \phi(x)^{2}]$
plays a role of the Wick renormalization.
Since this quantum field model
was first introduced by H\o egh-Krohn \cite{Hoe71} in the ``$L^{2}$-regime"
$$\vert \alpha \vert<{\sqrt{4\pi}},$$
it is also called the {\it{H\o egh-Krohn model}}. 
For a physical background and related early works
of this model,
see e.g., \cite{AHk73, AHk74, Sim74} and references therein.
Kahane \cite{Kah85} constructed 
a random measure
$$\nu^{(\alpha)}_{\phi}({\dd}x):= 
\exp ^{\diamond} (\alpha \phi)(x) \hspace{0.5mm} {\dd}x, \qquad x\in \Lambda,$$
called the {\it{Gaussian mulptiplicative chaos}},
in the ``$L^{1}$-regime" 
$$\vert \alpha \vert<{\sqrt{8\pi}}.$$
It implies the existence of
the $\exp(\Phi)_2$-measure $\mu^{(\alpha)}$
under $\vert \alpha \vert<{\sqrt{8\pi}}$, which gives a generalization of
the early works mentioned above. 
After that, the relevance of both the Gaussian multiplicative chaos and 
the $\exp(\Phi)_2$-quantum field model has been received much attention 
by many people in connection with topics like the Liouville conformal field theory and the stochastic Ricci flow. 
See e.g., \cite{Kah85, DS11, RV14, Ber17, JM17, DS19, Bis20, Ber16} and references therein.
We should also mention that Kusuoka \cite{Kus92} independently
studied the $\exp(\Phi)_2$-quantum field model 
under $\vert \alpha \vert<{\sqrt{8\pi}}$.
\vspace{2mm}

By heuristic calculations, we observe that 
the $\exp(\Phi)_2$-measure $\mu^{(\alpha)}$ is 
an invariant measure of
the following two-dimensional parabolic stochastic
partial differential equation (SPDE in short) involving exponential nonlinearity:
\begin{equation}
\partial_{t}\Phi_{t}(x)=
\frac{1}{2}(\triangle-1) \Phi_{t}(x)
-\frac{\alpha}{2} 
\exp^{\diamond} \big( \alpha \Phi_{t}\big) (x)
+
{\dot W}_{t}(x), 
\qquad t>0,~x\in \Lambda,
\label{exp-SQE}
\end{equation}
where 
$( \dot W_{t} )_{t\geq 0}$ is an $\mathbb R$-valued Gaussian
space-time white noise, that is, the time derivative of 
a standard $L^{2}(\Lambda)$-cylindrical Brownian motion
$\{ W_{t}=(W_{t}(x) )_{x\in {\Lambda}}\}_{t\geq 0}$. 
We call (\ref{exp-SQE}) 
the $\exp(\Phi)_{2}$-{\it{stochastic quantization equation}} associated with $\mu^{(\alpha)}$.
Due to the singularity of the nonlinear drift term, 
the interpretation and 
construction of a solution to 
this {\it{singular}}-SPDE 
have been a challenging problem over the past years.
For a concise overview on stochastic quantization equations,
we refer to \cite{AMR15, AK20, ADG19, AKMR20} and references therein.
Albeverio and R\"ockner \cite{AR91} first solved 
(\ref{exp-SQE}) 
(in the case where $\Lambda$ is replaced by $\mathbb R^{2}$) 
weakly under $\vert \alpha \vert <{\sqrt{4\pi}}$ by using the Dirichlet form theory.
Inspired by recent quick developments of singular SPDEs
based on Hairer's groundbreaking work on {\it{regularity structures}} \cite{Hai14}
and the related work, called {\it{paracontrolled calculus}}, due to Gubinelli, Imkeller and Perkowski \cite{GIP15},
Garban \cite{Gar20} constructed a unique strong solution to 
(\ref{exp-SQE}) (for the case where $(\triangle-1)$ is replaced by $\triangle$, i.e., massless case) 
in a more restrictive condition
than $\vert \alpha \vert <{\sqrt{4\pi}}$.
In our previous paper \cite{HKK19}, we constructed the time-global and pathwise-unique solution to the 
SPDE (\ref{exp-SQE}) 
under $\vert \alpha \vert <{\sqrt{4\pi}}$
by 
splitting the original equation (\ref{exp-SQE}) into 
the Ornstein--Uhlenbeck process 
\begin{equation*}
\partial_{t}X_{t}(x)=
\frac{1}{2}(\triangle-1) X_{t}(x)+
{\dot W}_{t}(x), 
\label{OU-intro}
\end{equation*}
and the shifted equation
\begin{equation}
\partial_{t}Y_{t}(x)=
\frac{1}{2}(\triangle-1) Y_{t}(x)-\frac{\alpha}{2} \exp(\alpha Y_{t}(x))
\exp^{\diamond}(\alpha X_{t})(x) .
\label{shift-intro}
\end{equation}
This split is based on the idea of
Da Prato and Debussche \cite{DPD03}, which is
now called the {\it{Da Prato--Debussche trick}}.
By the 
uniqueness of the solution, we also obtained the identification with 
the limit of the solutions to the stochastic quantization equations generated by the approximating 
measures to the $\exp (\Phi )_2$-measure $\mu^{(\alpha)}$, and with the process obtained by 
the Dirichlet form approach. 
Our construction of the solution to the shifted equation (\ref{shift-intro}) is different from
the standard fixed-point argument applied in \cite{DPD03, Gar20}. To be more precise,
we proved 
convergence of solutions to approximating equations of 
(\ref{exp-SQE}) by using compact embedding, and then identified the limit as the solution. 
We should mention that,
after \cite{HKK19}, Oh, Robert and Wang \cite{ORW19} independently constructed the 
time-global unique solution to (\ref{exp-SQE}) 
in the same regime in \cite{HKK19}.
Later in \cite{ORTW20}, together with Tzvetkov, they
studied the massless case on two-dimensional compact Riemannian manifolds
in the $L^2$-regime.
Besides, elliptic SPDEs, which also
realize the $\exp (\Phi )_2$-quantum field model have been studied in e.g., \cite{ADG19}.
\vspace{2mm}

The main purpose of the present paper is 
to construct the time-global and pathwise-unique solution to 
the parabolic SPDE (\ref{exp-SQE}) in the full ``$L^1$-regime"
$\vert \alpha \vert <{\sqrt{8\pi}}$.
Although the present paper builds on our previous work \cite{HKK19}, we 
significantly improve the arguments of \cite{HKK19} in several ways.
To apply the Da Prato--Debussche trick, 
we need to
construct the Wick exponential of the Ornstein--Uhlenbeck process 
$\{ \exp^{\diamond}(\alpha X_{t}) \}_{t\geq 0}$ 
as a driving noise of the  
shifted equation (\ref{shift-intro}). Since the Gaussian free field $\mu_{0}$ is the stationary measure 
of the Ornstein--Uhlenbeck process $\{X_{t} \}_{t\geq 0}$, this problem is reduced to the construction
of the Wick exponential $\exp^{\diamond}(\alpha \phi)$. In \cite[Theorem 2.2]{HKK19}, 
we constructed it 
under $\vert \alpha \vert<{\sqrt{4\pi}}$
by combining the Wick calculus of the Gaussian free field $\mu_{0}$
with the standard Fourier expansion on a negative order
$L^{2}$-Sobolev space $H^{s}(\Lambda)$ ($s<0$). 
However, this kind of argument does not work beyond the $L^{2}$-regime. 
Refining
existing results on 
the convergence
of the Gaussian multiplicative chaos 
$\nu^{(\alpha)}_{\phi}({\dd}x)$ in 
\cite{DS11, RV14, Ber17},
we construct
the Wick exponential $\exp^{\diamond}(\alpha \phi)$ on 
a suitable Besov space 
under $\vert \alpha \vert<{\sqrt{8\pi}}$
(see Theorem \ref{thm:expwick almostsure}).
This is one of the important contributions of the present paper.
On the other hand, in this case, since the Wick exponential $\exp^{\diamond}(\alpha \phi)$ 
does not have $L^{2}$-integrability with respect to $\mu_{0}$ unlike 
the case of $\vert \alpha \vert <{\sqrt{4\pi}}$,
we need to modify our arguments mentioned above into 
$L^{p}$-setting for the construction of the time-global and
pathwise-unique solution to (\ref{exp-SQE}).
Besides, due to the lack of the $L^{2}$-integrability, we cannot follow
the argument as in
\cite{AR90, AR91, HKK19, AKMR20}
to show the closability of
the associated Dirichlet form. To overcome this difficulty, in Corollary \ref{log-derivative}, 
we prove that  the Wick exponential $\exp^{\diamond}(\alpha \phi)$
has 
the $L^{2}$-integrability 
with respect not to $\mu _{0}$, but to
$\mu^{(\alpha)}$.
This key property plays a significant role 
not only for the closability of the Dirichlet form, but
also for the identication of the diffusion process
obtained by the Dirichlet form approach with the solution to the SPDE (\ref{exp-SQE}).
\vspace{2mm}

We should mention here that our model
is closely connected with the {\it{sine-Gordon model}}
(or $\cos (\Phi)_{2}$-quantum field model), which has 
been studied for a long period by many authors. 
See e.g., \cite{Fro76, FS76, FP77, AHk79} for the early works.
Since the sine-Gordon model is formally obtained by replacing the nonlinearity $e^{\alpha \phi}$ by
$e^{{\sqrt{-1}}\alpha \phi}$, it has some similarities with 
the $\exp(\Phi)_{2}$-model. Indeed, it can be constructed rigorously in the same way as 
the $\exp(\Phi)_{2}$-model in the case
$\vert \alpha \vert< {\sqrt{4\pi}}$.
On the other hand, for large values of
$\vert \alpha \vert$ up to ${\sqrt{8\pi}}$, further renormalization by counter-terms 
is required (see \cite{BGN82, DH93} for details).
To make a rigorous meaning to stochastic quantization equations
associated with both the $\Phi^{4}_{3}$-model and the sine-Gordon model,
we require further renormalization procedures beyond the Wick renormalization, and 
recent developments of regularity structure and paracontrolled calculus
enable us to study such singular SPDEs rigorously. 
In \cite{HS16, CHS18}, Hairer, Shen and Chandra proved local well-posedness of 
(the massless version of) the sine-Gordon stochastic quantization equation 
by applying regularity structure.
Hence, at first sight, one might guess that regularity structure or paracontrolled calculus 
is applicable to the $\exp(\Phi)_{2}$-stochastic quantization equation (\ref{exp-SQE}) beyond the $L^{2}$-regime.
To apply such general theories, we usually assume that the inputs of the solution map
to the shifted equation of a given singular SPDE 
take values in a Besov space
$B^{s}_{\infty, \infty}(\Lambda)$ with some $s<0$.
 (We mention here that
the reconstruction theorem in $B_{p,q}^s ({\mathbb R}^d)$ was also
studied by Hairer and Labb\'{e} \cite{HL17}, but they considered only 
the models with $B_{\infty,\infty}^{s}$-type bounds.)
In contrast,
the Wick exponential of the Ornstein--Uhlenbeck process 
$\{ \exp^{\diamond}(\alpha X_{t}) \}_{t\geq 0}$, which 
plays a role of an input in our case,
belongs to $B^{s}_{p,p}(\Lambda)$ 
for some $p \in [1,2)$, but does not to $B^{s}_{\infty, \infty}(\Lambda)$ (see Theorem \ref{thm:expwick OU}).
Moreover, since the nonlinear term of 
the SPDE (\ref{exp-SQE}) has exponential growth, it is out of results by these general theories. Alternatively, by making use of
the nonnegativity of $\exp^{\diamond}(\alpha X_{t})$, we may define a product between two rough objects
$\exp(\alpha Y_{t})$ and $\exp^{\diamond}(\alpha X_{t})$ on the right-hand side of the shifted equation 
(\ref{shift-intro})
(see Theorem
\ref{replaced keythm}).
This is the most crucial point in our argument.
We remark that the nonnegativity of the Wick exponential is a remarkable and useful property, and is also applied in proofs of previous results (see e.g \cite{AHk74, Gar20, ORW19, ORTW20, HKK19}).

\vspace{2mm}

The organization of the rest of the present paper is as follows: In Section \ref{sec:mainthm}, we present
the framework and state the main theorems (Theorems \ref{mainthm1},
\ref{mainthm2} and \ref{mainthm3}). In Section \ref{subsec:Besov}, we fix some
notations and summarize several basic properties on Besov spaces.
In Section \ref{sec:Wick}, we introduce an approximation
of the Wick exponentials of the Gaussian free fields and 
show its almost-sure convergence in an appropriate Besov space (see Theorem
\ref{thm:expwick almostsure}).
For later use, we modify the argument of \cite{Ber17} to 
obtain a stronger estimate than existing results.
Moreover, we also prove that the $\exp(\Phi)_2$-measure $\mu^{(\alpha)}$ is well-defined and 
the Wick exponential $\exp^{\diamond}(\alpha \phi)$ has the $L^{2}$-integrability with respect to $\mu^{(\alpha)}$ (see Corollaries
\ref{cor:expmeas} and \ref{log-derivative}).
In Section \ref{sec:OU}, we prove the almost-sure convergence of the Wick exponential of the infinite-dimensional Ornstein--Uhlenbeck process (see Theorem \ref{thm:expwick OU}).
In Section \ref{sec:wellposed}, we prove Theorem \ref{mainthm1} using the result of Section \ref{sec:OU}.
In Sections \ref{sec:stationary} and \ref{sec:DF}, we prove Theorems \ref{mainthm2} and \ref{mainthm3}, respectively.
Since some parts of Sections \ref{sec:wellposed}, \ref{sec:stationary} and \ref{sec:DF} go in very similar ways 
to the arguments of the previous paper  \cite{HKK19}, we sometimes omit the details in the present paper.
Finally, in Appendix, we give several estimates on the approximation of the Green 
function of $(1-\triangle)$, which is used in Section \ref{sec:Wick}.
%
\subsection{Statement of the main theorems}\label{sec:mainthm}
We begin with introducing some notations and objects.
Let $\Lambda={\mathbb T}^{2}=(\mathbb R/2\pi {\mathbb Z})^{2}$ be the two-dimensional
torus equipped with the Lebesgue measure ${\dd}x$. Let 
$L^{p}(\Lambda)$ ($p\in[1,\infty]$) be
the usual real-valued Lebesgue space. 
In particular, $L^{2}(\Lambda)$ is a Hilbert space equipped with the usual inner product
$$ \langle f, g \rangle=\int_{\Lambda} f(x)g(x)\hspace{0.5mm}{\dd}x, \qquad 
f,g\in L^{2}(\Lambda).$$
Let $C^{\infty}(\Lambda)$ be the space of real-valued smooth functions on 
$\Lambda$
equipped with the topology given by the convergence $f_n\rightarrow f$ in $C^\infty (\Lambda )$:
$$
\sup _{(x_1, x_2) \in \Lambda} \Big| \frac{\partial ^{i+j} f_n}{\partial x_1^i \partial x_2^j}(x_1, x_2) 
- \frac{\partial ^{i+j} f}{\partial x_1^i \partial x_2^j}(x_1, x_2) \Big| \rightarrow 0
\quad \mbox{as } n\to \infty
$$
for all $i,j \in {\mathbb N}\cup \{ 0\}$.
We denote by ${\mathcal D}'(\Lambda)$ the topological dual space of 
$C^{\infty}(\Lambda)$.
We have $L^{p}(\Lambda) \subset {\mathcal D}'(\Lambda)$ for all $p\in[1,\infty]$ by identification of $f \in L^{p}(\Lambda)$ with the map
$C^{\infty}(\Lambda)\ni\varphi \mapsto
\int_{\Lambda} f(x) \varphi(x) \hspace{0.5mm} {\dd}x \in {\mathbb R}$.
Since 
$C^{\infty}(\Lambda) \subset L^{2}(\Lambda) \subset
{\mathcal D}'(\Lambda)$, the $L^{2}$-inner product $\langle \cdot, \cdot \rangle$ is naturally
extended to the pairing of $C^{\infty}(\Lambda)$ and its dual space 
${\mathcal D}'(\Lambda)$.

For ${k}=(k_{1}, k_{2}) \in {\mathbb Z}^{2}$ and $x=(x_{1}, x_{2}) \in \Lambda$,
we write $\vert k \vert=(k_1^2+k_2^2)^{1/2}$ and
$k\cdot x=k_{1}x_{1}+k_{2}x_{2}$.
Although we work in the
framework of real-valued functions, 
it is sometimes easier to do computations by using 
a system of complex-valued functions $\{ {\bf e}_{k}\}_{k \in \mathbb Z^{2}}$
defined by
$$
{\bf e}_{k}(x)=\frac{1}{2\pi}e^{{\sqrt{-1}}k\cdot x}, \qquad  k\in {\mathbb Z^{2}},~x\in \Lambda. 
$$
For $f\in {\mathcal D}'(\Lambda)$, we define the $k$-th Fourier coefficient 
$\hat f(k)$
($k\in \mathbb Z^{2}$) by
$$
\hat f(k):=
\langle f, {\bf e}_{-k} \rangle=
\int_\Lambda f(x)\overline{{\bf e}_k(x)}
\hspace{0.5mm}{\dd}x.
$$
Note that, since $f$ is real-valued, 
$\hat{f}(-k)=\overline{\hat{f}(k)}$ for $k\in\mathbb{Z}^2$.

For $s\in \mathbb R$, 
we define the real $L^{2}$-Sobolev space of order $s$ with periodic boundary condition
by
$$ 
H^{s}(\Lambda)=\left\{ u\in {\mathcal D}'(\Lambda) \, ; 
\|u\|_{H^s}^2:=\sum_{k \in {\mathbb Z}^{2}} (1+\vert k \vert^{2})^{s} 
\vert\hat{u}(k)\vert^2
<\infty
\right\}.
$$
This space is a Hilbert space equipped with the inner product
$$ { (}u,v{ )}_{H^{s}} :=
\sum_{k \in {\mathbb Z}^{2}} (1+\vert k \vert^{2})^{s}   
\hat{u}(k)\overline{\hat{v}(k)},
\qquad u,v \in H^{s}(\Lambda).$$
Note that $H^{0}(\Lambda)$ coincides with $L^{2}(\Lambda)$ and
we regard $H^{-s}(\Lambda)$ as the dual space of $H^{s}(\Lambda)$
through the standard chain $ H^{s}(\Lambda) \subset 
L^{2}(\Lambda) \subset 
H^{-s}(\Lambda)$ for $s \geq 0$.

We now define the massive Gaussian free field measure $\mu_{0}$ by the centered Gaussian measure 
on $\mathcal{D}'(\Lambda)$ with covariance $(1-\triangle)^{-1}$, that is, determined by the formula
\begin{equation}
\int_{\mathcal{D}'(\Lambda)}
\langle \phi,{\bf e}_k\rangle \overline{\langle \phi,{\bf e}_{\ell}\rangle}
\
\mu_0({\dd}\phi)
=(1+|k|^2)^{-1}\mathbf{1}_{k=\ell},
\qquad
k,\ell\in\mathbb{Z}^2,
\label{eq:def of mu_0}
\end{equation}
where $\triangle$ is the Laplacian acting on $L^{2}(\Lambda)$ with periodic boundary condition.
This formula implies
$$ \int_{\mathcal{D}'(\Lambda)}\|\phi\|_{H^{-\varepsilon}}^2\mu_0({\dd}\phi)<\infty
$$
for any $\varepsilon>0$, and thus the Gaussian free field measure $\mu_0$ has a full support on 
$H^{-\varepsilon}(\Lambda)$.
For a charge parameter $\alpha \in (-{\sqrt{8\pi}}, {\sqrt{8\pi}})$,
we then define the \emph{$\exp (\Phi) _2$-quantum field} (or the \emph{$\exp (\Phi) _2$-measure}) 
$\mu^{(\alpha)}$
on $\mathcal{D}'(\Lambda)$ by
\begin{equation*}\label{expphimeas}
\mu^{(\alpha)}({\dd}\phi)
:=\frac{1}{Z^{(\alpha)}} \exp \left ( - \int_\Lambda\exp^\diamond(\alpha \phi )(x)
\hspace{0.5mm} {\dd} x\right ) \mu _0 ( {\dd}\phi ),
\end{equation*}
where
$Z^{(\alpha)}>0$ is the normalizing constant and
$\exp^\diamond(\alpha\phi)$ is the \emph{Wick exponential} which will be rigorously
constructed in Section \ref{sec:Wick}.
We prove in 
Theorem \ref{thm:expwick almostsure}
that the function $\phi\mapsto\int_{\Lambda} 
\exp^\diamond(\alpha \phi )(x)\hspace{0.5mm} {\dd}x$ 
is a positive measurable function for all $\vert \alpha \vert
<{\sqrt{8\pi}}$.
Hence, we may also regard 
$\mu^{(\alpha)}$ as a probability measure on $\mathcal{D}'(\Lambda)$ (see Corollary \ref{cor:expmeas}).

In the present paper, we consider the stochastic quantization equation (\ref{exp-SQE})
associated with $\exp(\Phi)_2$-measure $\mu^{(\alpha)}$, that is a parabolic SPDE 
\begin{equation}
\partial _t \Phi_t (x) = \frac 12 (\triangle-1) \Phi_t (x) - \frac\alpha2 \exp^\diamond(\alpha \Phi_t)(x) + \dot W_t (x) , \qquad t>0,~x\in \Lambda,
\nonumber
\end{equation}
where $W=\{ W_t(x); t\geq 0, x\in \Lambda \}$ is an $L^{2}(\Lambda)$-cylindrical Brownian motion 
defined on a filtered probability space $(\Omega,\mathcal{F},({\mathcal F}_{t})_{t\geq 0},
\mathbb{P})$ and $(\dot W_t)_{t\geq 0}$ is its time derivative in weak sense. This driving noise has 
a convenient Fourier series representation
$$ 
W_{t}(x)=\sum_{k\in {\mathbb Z}^{2}} w^{(k)}_{t} e_{k}(x), \qquad t\geq 0,~x\in \Lambda,
$$
where $\{e_k\}_{k\in\mathbb{Z}^2}$ is a real-valued
complete orthonormal system (CONS) of $L^2(\Lambda)$ defined by
$e_{(0,0)}(x)=(2\pi)^{-1}$ and
\begin{equation}\label{real valued CONS}
e_{k}(x)
=
\frac{1}{{\sqrt 2}\pi}
\left\{
\begin{aligned}
&\cos(k\cdot x),& k\in\mathbb{Z}^2_+,\\
&\sin(k\cdot x),& k\in\mathbb{Z}^2_-,
\end{aligned}
\right.
\end{equation}
with ${\mathbb Z}^{2}_{+}=\{ (k_{1}, k_{2}) \in {\mathbb Z}^{2}\, ; k_{1}>0 \} 
\cup \{ (0, k_{2})\, ; k_{2}>0 \}$ and ${\mathbb Z}^{2}_{-}=-{\mathbb Z}^{2}_{+}$,
and $\{w^{(k)} \}_{k\in {\mathbb Z}^{2}}$ is a family of independent 
one-dimensional $({\mathcal F}_{t})_{t\geq 0}$-Brownian motions starting at origin.
See \cite[Chapter 4]{DZ92} for the precise definition of cylindrical Brownian motions.
For later use, we note here that
\begin{equation*}
e_{k}(x)=
\left\{
\begin{aligned}
&
\frac{\sqrt{2}}{2}\big( {\bf e}_{k}(x)+{\bf e}_{-k}(x) \big),&
k\in\mathbb{Z}^2_{+},
\\
&
\frac{\sqrt{2}}{2{\sqrt{-1}}}\big( {\bf e}_{k}(x)-{\bf e}_{-k}(x) \big),&
k\in\mathbb{Z}^2_{-}.
\end{aligned}
\right.
\end{equation*}

The exponential term of the SPDE \eqref{exp-SQE} 
is difficult to treat as it is, because
the solution $\Phi_t$ is expected to take values in 
$\mathcal{D}'(\Lambda) \setminus C(\Lambda)$.
For this reason, we need to 
give a rigorous meaning of this SPDE by the renormalization.
We assume some properties for the multiplier function.

\begin{hypo}\label{hypo on psi}
$\psi:\mathbb{R}^2\to[0,1]$ is a function satisfying the following properties:
\begin{enumerate}
\item $\psi(0)=1$ and $\psi(x)=\psi(-x)$ for any $x\in\mathbb{R}^2$.
\item $\sup_{x\in\mathbb{R}^2}|x|^{2+\kappa}|\psi(x)|<\infty$ for some $\kappa>0$.
\item $\sup_{x\in\mathbb{R}^2}|x|^{-\zeta}|\psi(x)-1|<\infty$ for some $\zeta>0$.
\end{enumerate}
\end{hypo}

Note that $\psi$ does not need to be continuous except the origin.
For a function $\psi$ satisfying Hypothesis \ref{hypo on psi}, we define the Fourier cut-off operator $P_N$ on ${\mathcal D}'(\Lambda )$ by
\begin{align}\label{def of P_N}
P_N f (x) = \sum _{k\in {\mathbb Z}^2} \psi (2^{-N} k) 
\hat{f}(k)
{\bf{e}}_k (x), \qquad N\in \mathbb N,~x\in \Lambda.
\end{align}
From Hypothesis \ref{hypo on psi}, we have the following.
\begin{itemize}
\item $P_N$ maps $H^{-1-\varepsilon}(\Lambda)$ to $H^{1+\varepsilon}(\Lambda)$ for small $\varepsilon>0$.
Since $H^{1+\varepsilon}(\Lambda) \subset C(\Lambda)$, the regularized cylindrical 
Brownian motion $(P_NW_t)_{t\ge0}$ is a continuous function almost surely.
\item $\lim_{N\to\infty}\|P_Nf-f\|_{H^s}=0$ for any $s\in\mathbb{R}$ and $f\in H^s(\Lambda)$.
\end{itemize}

By introducing approximating equations driven by the regularized white noise $(P_{N} \dot W_{t})_{t\geq 0}$, we obtain the following theorem in the 
full $L^{1}$-regime of the charge parameter $\alpha$. 
See Section \ref{subsec:Besov} below for the definition of the Besov space $B_{p,p}^{-\varepsilon}(\Lambda)$.

\begin{thm}\label{mainthm1}
Assume that $\psi$ satisfies Hypothesis {\rm{\ref{hypo on psi}}}.
Let $|\alpha|<\sqrt{8\pi}$, $p\in(1,\frac{8\pi}{\alpha^2}\wedge2)$, and $\varepsilon>0$.
For any $N\in\mathbb{N}$, consider the initial value problem
\begin{equation}\label{expsqe1}
\left\{
\begin{aligned}
\partial _t \Phi_t^N &= \frac 12 (\triangle-1) \Phi_t^N 
- \frac\alpha2 \exp\left(\alpha \Phi_t^N-\frac{\alpha^2}2 C_N\right) + P_N\dot W_t ,\qquad t>0,\\
\Phi_0^N&=P_N\phi,
\end{aligned}
\right.
\end{equation}
where $\phi\in\mathcal{D}'(\Lambda)$ and
$$
C_N:=\frac1{4\pi ^2}\sum_{k\in\mathbb{Z}^2}\frac{\psi(2^{-N}k)^2}{1+|k|^2}.
$$
Then for $\mu_0$-almost every $\phi\in\mathcal{D}'(\Lambda)$, the unique time-global classical solution $\Phi^N$ converges as $N\to\infty$ to a $B_{p,p}^{-\varepsilon}(\Lambda)$-valued stochastic process $\Phi$ in the space $C([0,T];B_{p,p}^{-\varepsilon}(\Lambda))$ 
for any $T>0$, $\mathbb P$-almost surely. Moreover, the limit $\Phi$ is independent of the choice of $\psi$.
\end{thm}
In this paper we call this ${\Phi}$ 
the \emph{strong solution} of the SPDE \eqref{exp-SQE}, because in view of Theorem \ref{mainthm1} we have the mapping from the initial value $\phi$ and the driving noise $\dot W_t$ to the process ${\Phi}$.

\begin{rem}
The key ingredient of the proof is Theorem {\rm{\ref{thm:expwick almostsure}}} below is the almost-sure convergence of Gaussian multiplicative chaos (GMC in short).
The law of GMC was first constructed by Kahane {\rm{\cite{Kah85}}}, and Robert and Vargas {\rm{\cite{RV10}}} extended it for general convolution approximations of the covariance kernel.
Although these results give only convergence in law, some stronger convergence results were also obtained: almost-sure convergence for the circle average and standard Fourier projection {\rm{\cite{DS11}}} and the convergence in probability for general convolution approximations {\rm{\cite{Ber17}}}.
See {\rm{\cite{RV14, Ber16}}} for the reviews of these theories.
Our proof of Theorem {\rm{\ref{thm:expwick almostsure}}} is a modification of {\rm{\cite{Ber17}}}.
We remark that Hypothesis {\rm{\ref{hypo on psi}}} is prepared for the main theorems on singular SPDEs (for e.g. Theorem {\rm{\ref{mainthm1}}}), and the circle average approximation contained in {\rm{\cite{Ber17}}} does not satisfy Hypothesis {\rm{\ref{hypo on psi}}}.
However, our construction of Wick exponentials of the Gaussian free field in Section {\rm{\ref{sec:Wick}}} includes the case of the approximations by averaging treated in {\rm{\cite{Ber17}}}, in particular the circle average approximation, because the estimates \eqref{estimate:GN1} and \eqref{estimate:GN2} below hold also for the approximations by averaging (see Section {\rm{\ref{sec:average}}}).
See Section {\rm{\ref{sec:Wick}}} for our construction of Wick exponentials (GMC).
\end{rem}

\begin{rem}
Since the $\exp(\Phi)_2$-measure $\mu^{(\alpha)}$ is absolutely continuous with respect to $\mu_0$ 
(see Corollary {\rm{\ref{cor:expmeas}}}), 
``$\mu_0$-almost every $\phi$" can be replaced by ``$\mu^{(\alpha)}$-almost every $\phi$".
\end{rem}

\begin{rem}
We can refine the state space of the strong solution obtained in Theorem {\rm{\ref{mainthm1}}}.
Precisely, the strong solution is in $C([0,T];H^{-\varepsilon}(\Lambda))$ almost surely (see 
Corollary {\rm{\ref{maincor}}} for detail).
\end{rem}

To introduce another approach to the SPDE \eqref{exp-SQE}, we define the regularized ${\rm exp}(\Phi)_{2}$-measure 
\begin{align}\label{expNmeas}
\mu_N^{(\alpha)}({\dd} \phi):=\frac{1}{Z^{(\alpha)}_N} \exp \left\{ - \int_\Lambda\exp\left(\alpha P_N\phi (x)-\frac{\alpha^2}{2}C_N\right){\dd}x\right\} \mu _0 ( {\dd}\phi ),
\qquad N\in\mathbb{N},
\end{align}
where $Z^{(\alpha)}_N>0$ is the normalizing constant, and consider the SPDE associated with this measure.
The sequence $\{\mu_N^{(\alpha)}\}_{N\in\mathbb{N}}$ of probability measures weakly converges to 
$\mu^{(\alpha)}$
(see Corollary \ref{cor:expmeas}).

\begin{hypo}\label{hypo on P}
The operators $P_N$ defined by \eqref{def of P_N} satisfy the following properties.
\begin{enumerate}
\item $P_N$ is nonnegative, that is, $P_Nf\ge0$ if $f\ge0$.
\item For any $p\in(1,2)$, $s\in\mathbb{R}$, 
there exists a constant $C>0$ such that
$$
\sup_{N\in\mathbb{N}}\|P_Nf\|_{B_{p,p}^s}\le C\|f\|_{B_{p,p}^s},\qquad
\lim_{N\to\infty}\|P_Nf-f\|_{B_{p,p}^s}=0
$$
for any $f\in B_{p,p}^s(\Lambda)$.
\end{enumerate}
\end{hypo}

If $\psi$ is a Schwartz function and the inverse Fourier transform of $\psi$ is a nonnegative function, then Hypothesis \ref{hypo on P} holds. See e.g., \cite[Proposition 2.78]{BCD11}.

\begin{thm}\label{mainthm2}
Assume that $\psi$ satisfies Hypotheses {\rm{\ref{hypo on psi}}} and {\rm{\ref{hypo on P}}}.
Let $|\alpha|<\sqrt{8\pi}$ and $\varepsilon>0$.
For any $N\in\mathbb{N}$, consider the solution $\widetilde{\Phi}^N=\widetilde{\Phi}^N(\phi)$ of the SPDE
\begin{equation}\label{expsqe2}
\left\{
\begin{aligned}
\partial_t\widetilde{\Phi}_t^N&=\frac12(\triangle-1)\widetilde{\Phi}_t^N
-\frac\alpha2 P_N\exp\left(\alpha P_N\widetilde{\Phi}_t^N-\frac{\alpha^2}2C_N\right)+\dot{W}_t,
\qquad t>0,\\
\widetilde{\Phi}_0^N&=\phi\in\mathcal{D}'(\Lambda).
\end{aligned}
\right.
\end{equation}
Let $\xi_N$ be a random variable with the law $\mu_N^{(\alpha)}$ independent of $W$.
Then $\widetilde{\Phi}^{N,{\rm stat}}=\widetilde{\Phi}^N(\xi_N)$ is a stationary process and 
converges in law as $N\to\infty$ to the strong solution ${\Phi}^{\rm stat}$ of \eqref{exp-SQE} with an initial law $\mu^{(\alpha)}$, on the space $C([0,T];H^{-\varepsilon}(\Lambda))$ for any $T>0$.
Moreover, the law of the random variable ${\Phi}^{\rm stat}_t$ is $\mu^{(\alpha)}$ for any $t\ge0$.
\end{thm}

\begin{cor}\label{maincor}
The strong solution $\Phi$ of the SPDE \eqref{exp-SQE} belongs to the space 
$C([0,T];H^{-\varepsilon}(\Lambda))$, $\mathbb{P}$-almost surely, for 
$\mu_0$-almost every (or $\mu^{(\alpha)}$-almost every) initial value $\phi\in\mathcal{D}'(\Lambda)$.
\end{cor}

Finally, we discuss a connection between the SPDE (\ref{exp-SQE})
and the Dirichlet form theory.
Let $s \in(0,1)$ be an exponent fixed later (see Corollary \ref{log-derivative}) and set 
$H=L^{2}(\Lambda)$ and $E=H^{-s}(\Lambda)$.
Recall that $\{e_k\}_{k\in\mathbb{Z}^2}$ is a real-valued CONS of $H$ defined by \eqref{real valued CONS}.
We then denote by ${\mathfrak F}C_{b}^{\infty}$ the space of all smooth cylinder functions $F:E\to\mathbb{R}$ having the form
$$
F(\phi)=f(\langle \phi, l_{1} \rangle , \dots , \langle \phi, l_{n} \rangle), \qquad \phi \in E,
$$
with $n\in {\mathbb N}$, $f
\in C^{\infty}_{b}({\mathbb R}^{n}; {\mathbb R})$ 
and $l_{1}, \dots , l_{n} \in {\rm Span}\{ e_{k}; k\in {\mathbb Z}^{2} \}$.
Since supp$(\mu^{(\alpha)})=E$, two different functions in
${\mathfrak F}C_{b}^{\infty}$ are also different in $L^{p}(\mu^{(\alpha)})$-sense.
Moreover, ${\mathfrak F}C_{b}^{\infty}$ is dense in $L^{p}(\mu^{(\alpha)})$ for all $p\geq 1$.
For $F \in {\mathfrak F}C_{b}^{\infty}$, we define the $H$-derivative 
$D_{H}F:E\to H$ by
$$
D_{H}F(\phi):=\sum_{j=1}^{n}{\partial_{j}} {f}
\big(
\langle \phi,l_{1} \rangle , \ldots,
\langle \phi,l_{n} \rangle
\big)l_{j}, \qquad \phi\in E.
$$
We then consider a pre-Dirichlet form $({\cal E},{\mathfrak F}C_{b}^{\infty})$
which is given by 
\begin{equation}
{\cal E}(F,G)=
\frac{1}{2} 
\int_{E} \big( D_{H}F(\phi), D_{H}G(\phi) \big)_{H}
 \mu ^{(\alpha )}({\dd}\phi),\qquad F,G\in {\mathfrak F}C_{b}^{\infty},
 \label{DF-intro}
\end{equation}
where $(\cdot,\cdot)_H$ is the inner product of $H$.
Applying the integration by parts formula for $\mu^{(\alpha)}$,
we obtain that
$({\cal E},{\mathfrak F}C_{b}^{\infty})$ is closable on $L^{2}(\mu^{(\alpha)})$
(see Proposition \ref{Prop-IbP} below for detail),
so we can define ${\cal D(E)}$ as the completion of ${\mathfrak F}C_{b}^{\infty}$
with respect to ${\cal E}_{1}^{1/2}$-norm. Thus, by directly applying the general methods in the theory of 
Dirichlet forms (cf. \cite{MR92, CF12}), we can prove quasi-regularity of $({\cal E}, {\cal D(E)})$ and 
the existence of a diffusion process 
${\mathbb M}=(\Theta, {\mathcal G}, ( {\mathcal G}_{t})_{t\geq 0},  (\Psi_{t})_{t\geq 0}, ( {\mathbb Q}_{\phi} )_{\phi \in E})$
properly associated with $({\cal E}, {\cal D(E)})$.
\vspace{2mm}

The following theorem says that the diffusion process
$\Psi=(\Psi_{t})_{t\geq 0}$
coincides with the strong solution $\Phi$ obtained in Theorem \ref{mainthm1}.

\begin{thm}\label{mainthm3}
Let $|\alpha|<\sqrt{8\pi}$. Then for $\mu^{(\alpha)}$-almost every $\phi$, the diffusion process
$\Psi$ coincides ${\mathbb Q}_\phi$-almost surely with the strong solution $\Phi$ of the SPDE \eqref{exp-SQE} 
with the initial value $\phi$, driven by some 
$L^{2}(\Lambda)$-cylindrical 
$({\mathcal G}_{t})_{t\ge0}$-Brownian motion ${\mathcal W}=({\mathcal W}_{t})_{t\geq 0}$.
\end{thm}

\subsection{Notations and preliminaries}\label{subsec:Besov}
Throughout this paper, we use the notation $A\lesssim B$ for two functions $A = A(\lambda)$ and $B = B(\lambda)$ of a variable $\lambda$, if there exists a constant $c>0$ independent of $\lambda$ such that $A(\lambda)\le cB(\lambda)$ for any $\lambda$. We write $A\asymp B$ if $A\lesssim B$ and $B\lesssim A$. We write $A\lesssim_\mu B$ if we want to emphasize that the constant $c$ depends on another variable $\mu$.

For a measure space $(\mathfrak{M},m)$ and a Banach space $B$, denote by $L^p(\mathfrak{M},m;B)$ the usual $L^p$-space, where $\mathfrak{M}$ or $m$ may be omitted if they are obvious in the context.
If $B=\mathbb{R}$, then we write it by $L^p(\mathfrak{M},m)$ simply.
If $\mathfrak{M}$ is a compact topological space, denote by $C(\mathfrak{M};B)$ the space of continuous functions with the supremum norm.
\vspace{2mm}

We collect several basic facts on function spaces used through this paper.
Below we usually denote $L^{p}(\Lambda)$, $H^{s}(\Lambda)$ and $B^{s}_{p,q}(\Lambda)$ by
$L^{p}$, $H^{s}$ and $B^{s}_{p,q}$, respectively, for the sake of brevity.
Denote by $\mathcal{S}(\mathbb{R}^2)$ for the space of real-valued Schwartz functions on $\mathbb{R}^2$ 
and denote its dual by $\mathcal{S}'(\mathbb R^{2})$, which is the space of tempered distributions.
The Fourier transform $\mathcal F$ is defined by
$$
(\mathcal{F}f)(\xi):=\frac1{2\pi}\int_{\mathbb{R}^2} f(x)e^{-\sqrt{-1}x\cdot\xi}
\hspace{0.5mm}{\dd}x,
\qquad f\in {\mathcal S}(\mathbb R^{2}),~\xi\in\mathbb{R}^2,
$$
and so the inverse Fourier transform is given by ${\mathcal F}^{-1}f(z)={\mathcal F}f(-z)$ ($z\in \mathbb R^{2}$).
Also for the distribution $f\in {\mathcal S}'(\mathbb R^{2})$, the usual generalization of the Fourier transform is considered.

Let $(\chi,\rho)$ be a dyadic partition of unity, that is, they satisfy the following:
\begin{itemize}
\item $\chi, \rho: \mathbb R^{2} \to [0,1]$ are smooth radial functions,
\item ${\rm{supp}}(\chi) \subset B(0, 4/3)$,~${\rm{supp}}(\rho) \subset B(0,8/3) \setminus B(0,3/4)$,
\item ${\displaystyle{\chi(\xi)+\sum_{j=0}^\infty\rho(2^{-j}\xi)=1}}$ for any $\xi\in\mathbb{R}^2$,
\end{itemize}
where $B(x,r)$ stands for the open ball in $\mathbb R^{2}$ centered at $x$ and with radius $r$.
We then set $\rho_{-1}:=\chi$ and $\rho_j:=\rho(2^{-j}\cdot)$ for $j\ge0$.
We define the {\it{Littlewood-Paley blocks}} (or the Littlewood-Paley operator) $\{ \Delta_{j} \}_{j=-1}^{\infty}$ by
$$
(\Delta_jf)(x):=\sum_{k\in\mathbb{Z}^2}\rho_j(k)
\hat{f}(k){\bf e}_k(x), \qquad
f\in {\mathcal D}'(\Lambda),~x\in\Lambda.
$$
We then define the inhomogeneous Besov norm $\Vert \cdot \Vert_{B^{s}_{p,q}}$ 
and the Besov space 
$B^{s}_{p,q}(\Lambda)$ 
($s\in \mathbb{R}, p,q\in[1,\infty]$) by
\begin{equation*}
\Vert f\Vert_{B_{p,q}^s}:=
\left\{
\begin{aligned}
&\Big( \sum_{j=-1}^{\infty} 2^{jsq} \Vert \Delta_{j} f \Vert_{L^{p}}^{q} \Big)^{1/q},
&\qquad&q\in[1,\infty),\\
&\sup_{j\geq -1} \big( 2^{js} \Vert \Delta_{j} f \Vert_{L^{p}} \big),
&\qquad&q=\infty,
\end{aligned}
\right.
\end{equation*}
and
$$ B_{p,q}^s=B^{s}_{p,q}(\Lambda):=\{ f \in {\mathcal D}'(\Lambda)\,; \Vert f \Vert_{B^{s}_{p,q}}<\infty \},$$
respectively.
The Besov space $B^{s}_{p,q}$ is a Banach space.
Moreover, $B_{p,q}^s$ is separable if $q<\infty$ (see \cite[Lemma 2.73]{BCD11}).

We recall mainly from \cite{BCD11} some basic properties of Besov spaces.
We remark that the setting in \cite{BCD11} is not on a torus but on the Euclidean spaces.
However, it is known that most results in \cite{BCD11} also follow in the case of function spaces on a torus, and are proved by a parallel argument or by extending functions on a torus to those on the Euclidean spaces periodically (see e.g. \cite[Appendix A]{GIP15}).
In view of this fact we refer associate results in \cite{BCD11} below, though there is a difference between a torus and the Euclidean spaces.
The following embeddings are immediate consequences of the definition.
\begin{itemize}
\item If $s_1\le s_2$, then $B_{p,q}^{s_2}\subset B_{p,q}^{s_1}$.
\item If $p_1\le p_2$, then $B_{p_2,q}^s\subset B_{p_1,q}^s$.
\item If $q_1\le q_2$, then $B_{p,q_2}^s\supset B_{p,q_1}^s$. However, $B_{p,q_2}^s\subset B_{p,q_1}^{s-\varepsilon}$ for any $\varepsilon>0$.
\end{itemize}
It is important to note that $B_{2,2}^s$ coincides with the Sobolev space $H^s$ for any $s\in\mathbb{R}$, and $B_{\infty,\infty}^s$ coincides with the H\"older space $C^s(\Lambda)$ for any $s\in\mathbb{R}\setminus\mathbb{N}$ with the equivalent norms (\cite[Page 99]{BCD11}). The second and third properties above implies that $H^s\subset B_{p,p}^{s-\varepsilon}$ for any $p\in[1,2]$ and $\varepsilon>0$.

The following is an immediate consequence of the interpolations of $L^p$-spaces and of $\ell^p$-spaces.

\begin{prop}\label{prop:interpolation besov}
Let $s_1,s_2\in\mathbb{R}$ and $p_1,p_2,q_1,q_2\in[1,\infty]$. Let $\theta\in[0,1]$ and set $s=(1-\theta)s_1+\theta s_2$, $\frac1p=\frac{1-\theta}{p_1}+\frac{\theta}{p_2}$, and $\frac1q=\frac{1-\theta}{q_1}+\frac{\theta}{q_2}$. Then one has
$$
\|f\|_{B_{p,q}^s}\le\|f\|_{B_{p_1,q_1}^{s_1}}^{1-\theta}\|f\|_{B_{p_2,q_2}^{s_2}}^{\theta}.
$$
\end{prop}

\begin{prop}[{\cite[Proposition 2.71]{BCD11}}]\label{prop:besov embedding}
For any $s\in\mathbb{R}$, $p_1,p_2,q_1,q_2\in[1,\infty]$ such that $p_1\le p_2$ and $q_1\le q_2$, one has the embedding
$$
B_{p_1,q_1}^s \hookrightarrow B_{p_2,q_2}^{s-2(1/p_1-1/p_2)}.
$$
In particular, the space $B_{p,p}^s$ is embedded into $C(\Lambda)$ if $s>\frac2p$.
\end{prop}

The following equivalence of norms plays an important role in Corollary \ref{log-derivative}.
\begin{prop}[{\cite[Theorem 9 and Remark 26]{Tri88}}]
\label{prop:heat-besov}
For any $s>0$ and $p,q\in[1,\infty]$,
$$
\|\xi\|_{B_{p,q}^{-s}}\asymp\|e^{\triangle}\xi\|_{L^p}+\left\|t^{s/2}\|e^{t\triangle}\xi\|_{L^p}\right\|_{L^q((0,1],t^{-1}
{\dd}t)},
$$
where $e^{t\triangle}$ denotes the heat semigroup of the Laplacian $\triangle$ on $\Lambda$.
\end{prop}
A distribution $\xi\in\mathcal{D}'(\Lambda)$ is said to be \emph{nonnegative} if $\xi(\varphi)=\langle \xi, \varphi
\rangle \ge0$ for any nonnegative $\varphi\in\mathcal{D}(\Lambda)$.
Let $B^{s,+}_{p,q}$ be the set of all nonnegative elements in $B^{s}_{p,q}$.
Thanks to the following theorem, a nonnegative distribution is regarded as a nonnegative Borel measure. This fact plays a crucial role in Section \ref{sec:wellposed}.

\begin{thm}[{\cite[Theorem 6.22]{LL01}}]\label{thm:LL}
For any nonnegative $\xi\in\mathcal{D}'(\Lambda)$, there exists a unique nonnegative Borel measure $\mu_\xi$ such that
$$
\xi(\varphi)=\int_{\Lambda}\varphi(x)\mu_\xi({\dd}x),\qquad
\varphi\in C^\infty (\Lambda).
$$
Consequently, the domain of $\xi$ can be extended to whole $C(\Lambda)$.
\end{thm}

\section{Wick exponentials of GFFs}\label{sec:Wick}

In this section, we construct
the Wick exponentials of Gaussian free fields (GFFs in short) on $\Lambda$, that is, the so-called 
Gaussian multiplicative chaos (see \cite{Kah85, DS11, RV14, Ber17, JM17, DS19, Bis20, Ber16}). 
For some specific approximations for Gaussian multiplicative chaos 
(e.g., usual Fourier cut-off and circle average), the almost-sure convergence 
were obtained in \cite{DS11, Ber16}. In the present paper, we consider the approximation 
with general Fourier multiplier operators as in 
(\ref{def of P_N}).
Since we need a stronger convergence 
for our purpose, we give a self-contained proof of the construction in this section.

As mentioned in Remark \ref{rem:approx} below, our arguments work more general approximations than previous results.

\subsection{GFFs and Wick exponentials}\label{sec:Wickmain}
Recall that $\mu_0$ is the centered Gaussian measure on $\mathcal{D}'(\Lambda)$ with covariance $(1-\triangle)^{-1}$.
On the probability space $(\Omega,\mathcal{F},\mathbb{P})$, a $\mathcal{D}'(\Lambda)$-valued random variable $\GFF$ with the law $\mu_0$ is called a (massive) Gaussian free field.
Recalling \eqref{eq:def of mu_0}, we have the covariance formula of the random field $\GFF$:
\begin{align}\label{eq:def of GFF}
\mathbb{E}\left[\GFF(x)\GFF(y)\right]
=\frac1{2\pi}\sum_{k\in\mathbb{Z}^2}\frac1{1+|k|^2}{\bf e}_k(x-y)
=G_\Lambda(x,y),\qquad x,y\in\Lambda,
\end{align}
where $G_\Lambda$ stands for the Green function of $1-\triangle$ on $\Lambda$.
Since $G_\Lambda$ depends on only the difference $x-y$, the law of $\GFF$ is shift invariant, that is, $\GFF\overset{d}{=}\GFF(\cdot+h)$ for any fixed $h\in\Lambda$.

The aim of this section is to define the formal exponential
$$
\exp(\alpha \GFF)
$$
for any GFF $\GFF$ and any $\alpha$ with $|\alpha|<\sqrt{8\pi}$. Since $\GFF$ is $\mathcal{D}'(\Lambda)$-valued, we need a renormalization procedure to give a rigorous meaning to it.
Recall that $\psi$ satisfies Hypothesis \ref{hypo on psi}, and the Fourier cut-off operator $P_N$ on $\mathcal{D}'(\Lambda)$ is defined by \eqref{def of P_N}:
$$
P_N f(x) = \sum _{k\in {\mathbb Z}^2} \psi _N(k) \hat{f}(k)
{\bf e}_k (x),
$$
where $\psi _N:= \psi (2^{-N}\cdot )$.
Since $P_N$ maps $H^{-1-\varepsilon}$ to $C(\Lambda)$ for small $\varepsilon>0$ as mentioned before (after Hypothesis \ref{hypo on psi}), the approximation $\GFF_N:=P_N\GFF$ is a continuous function, so the exponential $\exp(\alpha \GFF_N)$ is well-defined. 
However, to take a limit as $N\to\infty$, we need an approximation with renormalization
\begin{align}\label{eq:def of M_N}
\exp_{N}^{\diamond}(\alpha \GFF)(x):=\exp\left(\alpha \GFF_N(x)-\frac{\alpha^2}2C_N\right),\qquad N\in\mathbb{N},
\end{align}
where
$$
C_N:=\mathbb{E}[\GFF_N(x)^2]=\frac1{4\pi^2}\sum_{k\in\mathbb{Z}^2}\frac{\psi_N(k)^2}{1+|k|^2}.
$$
The following is the main theorem of this section.

\begin{thm}\label{thm:expwick almostsure}
Assume that $\psi$ satisfies Hypothesis {\rm{\ref{hypo on psi}}}.
Let $|\alpha|<\sqrt{8\pi}$ and choose parameters $p,\beta$ such that
\begin{align}\label{def p beta}
p\in \left(1,\frac{8\pi}{\alpha^2}\wedge2\right),
\qquad
\beta\in\left(\frac{\alpha^2}{4\pi}(p-1),\frac2{p}(p-1)\right).
\end{align}
Then the sequence $\{ \exp_{N}^{\diamond}(\alpha \GFF) \} _{N\in {\mathbb N}}$
converges in the space $B_{p,p}^{-\beta}$, $\mathbb{P}$-almost surely and in $L^p(\mathbb P)$.
Moreover, by regarding $\exp_{N}^{\diamond}(\alpha \GFF)$ as the random nonnegative Borel measure 
$\exp_{N}^{\diamond}(\alpha \GFF)(x)\hspace{0.5mm}{\dd}x$ 
on $\Lambda$ for $N\in {\mathbb N}$, one has the weak convergence of 
$\{ \exp_{N}^{\diamond}(\alpha \GFF)\} _{N\in {\mathbb N}}$ almost surely.
The limits obtained by different $\psi$'s coincide with each other almost surely.
\end{thm}
\begin{rem}

The conclusion of Theorem {\rm{\ref{thm:expwick almostsure}}} holds under the estimates \eqref{estimate:GN1} and \eqref{estimate:GN2} in Proposition {\rm{\ref{hypo on G}}} below, even without Hypothesis {\rm{\ref{hypo on psi}}}.
See Remark {\rm{\ref{rem:approx}}} below for details.
In most references, approximations with continuous parameter are used for the convergence 
in probability and in $L^p(\mathbb P)$.
It is associated to adopt $\psi _\varepsilon := \psi (\varepsilon \cdot )$ instead of $\psi _N$ 
for the approximation.
For almost-sure convergence we need discretization of the approximation parameter 
and sufficiently high speed of the approximation with respect to the parameter in order to 
control the $\mathbb P$-null sets.
This is the reason why we choose approximation with discrete parameter as appeared in 
the definition of $\psi _N$ in Theorem {\rm{\ref{thm:expwick almostsure}}}.
Here, we remark that for the convergence in $L^p(\mathbb P)$ (in particular the convergence in probability), we do not need to discretize the approximation parameter.
Furthermore we remark that we choose the exponential speed $2^{-N}$ for the definition $\psi _N$ because of the simplicity of the proof, and $N^{-r}$ with sufficiently large $r>0$ instead of $2^{-N}$ is also sufficient for the almost-sure convergence.
See the proof of Theorem 
{\rm{\ref{thm:expwick almostsure}}} in the last part of Section {\rm{\ref{sec:asconv}}}.
\end{rem}
We denote the ($\mathbb{P}$-almost-sure) unique limit by
$$
\exp^\diamond(\alpha \GFF).
$$
When the probability space $(\Omega,\mathbb{P})$ is $(\mathcal{D}'(\Lambda),\mu_0)$, the 
canonical map $\phi\mapsto\phi$ is obviously a GFF. We denote by $\exp^\diamond(\alpha\phi)$
the associated Wick exponential.
Since the approximation \eqref{eq:def of M_N} is nonnegative, we can define the $\exp(\Phi)_2$-measure $\mu^{(\alpha)}$ as follows.
\begin{cor}\label{cor:expmeas}
On the Borel probability space $(\mathcal{D}'(\Lambda),\mu_0)$, the probability measure
\begin{align*}
\mu^{(\alpha)}({\dd}\phi)
=\frac1{Z^{(\alpha)}}\exp\left(-\int_\Lambda\exp^\diamond(\alpha\phi)(x)
\hspace{0.5mm}{\dd}x\right)\mu_0({\dd}\phi)
\end{align*}
is defined as the limit of the approximating measures $\{\mu_N^{(\alpha)}\}_{N\in\mathbb{N}}$ given by \eqref{expNmeas} in weak topology.
Moreover, the following holds.
\begin{enumerate}
\item The Radon-Nikodym derivatives 
$\Big \{\frac{{\dd}\mu_N^{(\alpha)}}{{\dd}\mu_0}\Big\}_{N\in \mathbb N}$ are uniformly bounded.
\item $\frac{{\dd}\mu^{(\alpha)}}{{\dd}\mu_0}$ is bounded and strictly positive 
$\mu_0$-almost everywhere. Hence $\mu^{(\alpha)}$ and $\mu_0$ are 
absolutely continuous with respect to each other.
\end{enumerate}
\end{cor}
\begin{proof}
Denote $M_{\phi,N}^{(\alpha)}=\exp_N^\diamond(\alpha\phi)$ and $M_\phi^{(\alpha)}=\exp^\diamond(\alpha\phi)$ in short, and regard them as the corresponding random nonnegative Borel measures on $\Lambda$, according to Theorem \ref{thm:LL}.

Although the proof of (i) is completely the same as \cite[Corollary 2.3]{HKK19}, we note the fact on the uniform positivity of the normalizing constants
$$ 
Z^{(\alpha)}_{N}:=
\int_{\mathcal{D}'(\Lambda)}\exp \big( -M_{\phi,N}^{(\alpha)}(\Lambda) \big) \mu_0({\dd}\phi),
\qquad N\in \mathbb N,
$$
which is used in the next corollary.
By Jensen's inequality,
\begin{align*}
Z^{(\alpha)}_{N}
&\ge
\exp \Big( - \int_{\mathcal{D}'(\Lambda)}
M_{\phi,N}^{(\alpha)}(\Lambda)\mu_{0}({\dd}\phi) \Big )
=\exp
\Big( -\int_\Lambda {\dd}x \Big)>0.
\end{align*}
Here we used the fact that $\int_{
\mathcal{D}'(\Lambda)}
M_{\phi,N}^{(\alpha)}(x)\mu_0({\dd}\phi)=1$ for any $x\in \Lambda$, 
which follows from the definition.

Next we show (ii). Let $p$ and $\beta$ be as in Theorem \ref{thm:expwick almostsure}. For any $n\in\mathbb{N}$, we have
\begin{align*}
\mu_0\left(M_{\phi}^{(\alpha)}(\Lambda)\ge n\right)
&\le\frac1{n^p}\int_{\mathcal{D}'(\Lambda)}\big(M_{\phi}^{(\alpha)}(\Lambda)\big)^p\mu_0({\dd}\phi)\\
&\lesssim 
\frac1{n^p}\int_{\mathcal{D}'(\Lambda)}\big\|M_{\phi}^{(\alpha)}\big\|_{B_{p,p}^{-\beta}}^p\mu_0({\dd}\phi),
\end{align*}
since ${\bf1}_\Lambda\in C^\infty(\Lambda)\subset B_{p',p'}^\beta$ ($1/p+1/p'=1$) and $B_{p',p'}^\beta$ is a dual space of $B_{p,p}^{-\beta}$ (see e.g., \cite[Proposition 2.76]{BCD11}).
Letting $n\to\infty$, we have $\mu_0(M_{\phi}^{(\alpha)}(\Lambda)=\infty)=0$. Since 
$Z^{(\alpha)}:=\int_{\mathcal{D}'(\Lambda)}
\exp\big(-M_\phi^{(\alpha)}(\Lambda)\big)\mu_0({\dd}\phi)$ is positive by the above estimate of 
$Z_N^{(\alpha)}$ and the dominated convergence theorem, this implies 
$\frac{{\dd}\mu^{(\alpha)}}{{\dd}\mu_0}$ is bounded and strictly positive $\mu_0$-almost everywhere.
\end{proof}
Even though Theorem \ref{thm:expwick almostsure} and Corollary \ref{cor:expmeas} imply that the random variable $\phi\mapsto\exp^\diamond(\alpha\phi)$ belongs to $L^p(\mu^{(\alpha)};B_{p,p}^{-\beta})$,
the state space can be chosen smaller. The following fact plays a crucial role in 
Sections \ref{sec:stationary} and \ref{sec:DF}.
\begin{cor} \label{log-derivative}
If $|\alpha|<\sqrt{8\pi}$, then there exists an exponent $s\in (0,1)$ such that
\begin{align}\label{log-derivative conclusion}
\sup_{N\in\mathbb{N}}\int_{\mathcal{D}'(\Lambda)}\|\exp_N^\diamond(\alpha\phi)\|_{H^{-s}}^2
\mu_N^{(\alpha)}({\dd}\phi)<\infty.
\end{align}
Moreover, the random variable $\phi\mapsto\exp^\diamond(\alpha\phi)$ belongs to $L^2(\mu^{(\alpha)};H^{-s})$.
\end{cor}
\begin{proof}
Recall that $H^{-s} = B_{2,2}^{-s}$ for $s\in {\mathbb R}$.
By the interpolation between Besov spaces (Proposition \ref{prop:interpolation besov}),
$$
\|M_{\phi,N}^{(\alpha)}\|_{H^{-s}}\le\|M_{\phi,N}^{(\alpha)}\|_{B_{p,p}^{-\beta}}^{p/2}\|M_{\phi,N}^{(\alpha)}\|_{B_{\infty,\infty}^{-2}}^{1-p/2}
$$
for $p,\beta$ in \eqref{def p beta}, and $-s:=-\beta\frac{p}2-2(1-\frac{p}2)>-1$.
Since $M_{\phi,N}^{(\alpha)}$ is nonnegative, we have 
$$\|e^{t\triangle}M_{\phi,N}^{(\alpha)}\|_{L^\infty}\lesssim t^{-1}M_{\phi,N}^{(\alpha)}(\Lambda),
\quad t\in (0,1]
$$
by the bound of the heat kernel in spacial component.
By Proposition \ref{prop:heat-besov} we have
$$
\|M_{\phi,N}^{(\alpha)}\|_{B_{\infty,\infty}^{-2}}\lesssim M_{\phi,N}^{(\alpha)}(\Lambda).
$$
Since the function $x^{2-p}e^{-x}$ is bounded on $x\in(0,\infty)$,
\begin{align*}
\int_{\mathcal{D}'(\Lambda)}\|M_{\phi,N}^{(\alpha)}\|_{H^{-s}}^2e^{-M_{\phi,N}^{(\alpha)}(\Lambda)}
\mu_0({\dd}\phi)
&\lesssim\int_{\mathcal{D}'(\Lambda)}\|M_{\phi,N}^{(\alpha)}\|_{B_{p,p}^{-\beta}}^p
\big( M_{\phi,N}^{(\alpha)}(\Lambda) \big)^{2-p}
e^{-M_{\phi,N}^{(\alpha)}(\Lambda)}\mu_0({\dd}\phi)\\
&\lesssim\int_{\mathcal{D}'(\Lambda)}\|M_{\phi,N}^{(\alpha)}\|_{B_{p,p}^{-\beta}}^p
\mu_0({\dd}\phi).
\end{align*}
Since $\{M_{\phi,N}^{(\alpha)}\}_{N\in\mathbb{N}}$ are bounded in the space $L^p(\mu_0;B_{p,p}^{-\beta})$ as in Theorem \ref{thm:expwick almostsure}, 
and $\{Z_N^{(\alpha)}\}_{N\in\mathbb{N}}$ are uniformly positive as stated in the proof of Corollary \ref{cor:expmeas}, we have the uniform bound \eqref{log-derivative conclusion}.
Since $\langle M_{\phi,N}^{(\alpha)},{\bf e}_k\rangle\to\langle M_\phi^{(\alpha)},{\bf e}_k\rangle$ for any $k\in\mathbb{Z}^2$ almost everywhere,
by using Fatou's lemma we have
$$
\int_{\mathcal{D}'(\Lambda)}\big\|M_{\phi}^{(\alpha)}\big\|_{H^{-s}}^2\mu^{(\alpha)}({\dd}\phi)
\le
\liminf_{N\to\infty}\int_{\mathcal{D}'(\Lambda)}\big\|M_{\phi,N}^{(\alpha)}\big\|_{H^{-s}}^2
\mu_N^{(\alpha)}({\dd}\phi)<\infty.
$$
Thus we complete the proof.
\end{proof}

Below, we give a self-contained proof of Theorem \ref{thm:expwick almostsure}.
For the proof we prepare a lot of technical results, and in the end of Section \ref{sec:Wick},
Theorem \ref{thm:expwick almostsure} is proved.

\subsection{Approximation of the Green function}\label{sec:estimate of G}
By definition, the random field $\GFF_N=P_N\GFF$ has the covariance function
\begin{align*}
G_{M,N}(x,y):=\mathbb{E}[\GFF_M(x)\GFF_N(y)]
&=\frac1{2\pi}\sum_{k\in\mathbb{Z}^2}\frac{\psi_M(k)\psi_N(k)}{1+|k|^2}{\bf e}_k(x-y)
\end{align*}
for any $M,N\in\mathbb{N}$. Recall that $\psi_N=\psi(2^{-N}\cdot)$. By definition, $C_N=\mathbb{E}[\mathbb{X}_N^2(x)]=G_{N,N}(x,x)$.
The function $G_{M,N}$ approximates the Green function $G_\Lambda$ defined by \eqref{eq:def of GFF}.
In the following proposition, we summarize the properties of the function $G_{M,N}$ used in the proof of Theorem \ref{thm:expwick almostsure}.
We regard $G_{M,N}$ as a periodic function on $\mathbb{R}^2\times\mathbb{R}^2$, rather than a function on $\Lambda\times\Lambda$.

\begin{prop}\label{hypo on G}
Assume that $\psi$ satisfies Hypothesis {\rm{\ref{hypo on psi}}}.
Then for any $x,y\in\mathbb{R}^2$ with $|x-y|<1$ and any $M,N\in\mathbb{N}$,
\begin{align}\label{estimate:GN1}
G_{M,N}(x,y)&=-\frac1{2\pi}\log\left(|x-y|\vee2^{-M}\vee2^{-N}\right)+R_{M,N}(x,y),
\end{align}
where the remainder term $R_{M,N}(x,y)$ is uniformly bounded over $x,y,M,N$.
Moreover, there exist constants $C>0$ and $\theta>0$ such that, for any $M,N\in\mathbb{N}$,
\begin{align}\label{estimate:GN2}
\iint_{\Lambda\times\Lambda} \big|G_{M,N+1}(x,y)-G_{M,N}(x,y) \big|
\hspace{0.5mm} {\dd}x{\hspace{0.2mm}}{\dd}y\le C2^{-\theta N}.
\end{align}
\end{prop}

Since the proof of Proposition \ref{hypo on G} is long and technical, we provide it in Appendix \ref{hypo on G app}.
We remark that \eqref{estimate:GN2} can be improved by $L^p$-estimate for all $p\in [1,\infty )$ (see Proposition \ref{propA.G}).

\begin{rem}\label{rem:approx}
Theorem {\rm{\ref{thm:expwick almostsure}}} holds true for any multiplier $\psi$ 
such that the function $G_{M,N}$ defined from $\psi$ satisfies the 
estimates \eqref{estimate:GN1} and \eqref{estimate:GN2}.
Indeed, in the proof of Theorem {\rm{\ref{thm:expwick almostsure}}} after Proposition {\rm{\ref{hypo on G}}}, we use only \eqref{estimate:GN1} and \eqref{estimate:GN2}, but do not use Hypothesis {\rm{\ref{hypo on psi}}}.
The class of approximations satisfying \eqref{estimate:GN1} and \eqref{estimate:GN2} is quite large, and includes the approximations by averaging, treated in {\rm{\cite{Ber17}}}, in particular the circle average approximation (see Section {\rm{\ref{sec:average}}}).
Moreover, our proofs would go similarly even if we replace the torus $\Lambda$ with the 
Lebesgue measure ${\dd}x$ and the Gaussian field $\GFF$ generated by free field measure, 
by a two-dimensional compact Riemannian manifold ${\mathcal M}$ with its volume measure $\sigma$ 
and a Gaussian random field $\GFF ^{\mathcal M}$ on ${\mathcal M}$ with covariance function $G_{\mathcal M}$ satisfying \eqref{estimate:GN1} and \eqref{estimate:GN2} with replacement of $|x-y|$ by the metric $d(x,y)$ in ${\mathcal M}$, respectively.
However, in the case of ${\mathcal M}$ and $\GFF_N^{\mathcal M}$, $C_N (x) := \mathbb{E}[\GFF_N^{\mathcal M}(x)^2]$ appeared in \eqref{eq:def of M_N} for renormalization, which is a constant in the case of the torus with the Lebesgue measure ${\dd}x$, will depend on $x\in {\mathcal M}$ generally.
We are also able to extend it to compact Riemannian manifold with other dimensions.
In the case the range of the charge constant $\alpha$ should be changed according to the dimension.
Even though we have such extensions, for simplicity, we discuss our problem only on the torus 
$\Lambda$ with the Lebesgue measure ${\dd}x$ in the present paper.
\end{rem}

\subsection{Uniform integrability}\label{sec:constWick}
Using the first property \eqref{estimate:GN1} of Proposition \ref{hypo on G}, we first prove the uniform bound of $\{\exp_N^\diamond(\alpha \GFF)\}_{N\in\mathbb{N}}$ in 
$L^p(\mathbb{P};B_{p,p}^{-\beta})$.  
Below, we usually denote
$$
\GMC^{(\alpha)}_{N}=\exp_N^\diamond(\alpha \GFF)
$$
in short.
At the beginning, we present {\it{Kahane's convexity inequality}} (cf. \cite{Kah85}), which plays a significant role in the proof. 
\begin{lem}
[see {\cite[Proposition 5.6]{Bis20}}]
\label{lemma Kahane}
Let $D$ be an open and bounded subset of $\mathbb{R}^2$.
Let $\varphi_1,\varphi_2$ be continuous Gaussian random fields on $D$ with mean zero and with covariance functions $C_1,C_2:D\times D\to\mathbb{R}$, respectively. If $C_1(x,y)\le C_2(x,y)$ for any $x,y\in D$, then one has
$$
\mathbb{E}\left[\left\{\int_D \exp\left({\varphi_1(x)-\frac12C_1(x,x)}\right){\dd}x\right\}^p\right]
\le \mathbb{E}\left[\left\{\int_D \exp\left({\varphi_2(x)-\frac12C_2(x,x)}\right){\dd}x\right\}^p\right]
$$
for any $p\in[1,\infty)$.
\end{lem}
The following estimate is useful to determine the regularity of $\GMC_N^{(\alpha)}$.
The estimate is called a multifractal property and is proved also in previous results (see e.g. \cite[Theorem 3.23]{Ber16}, \cite[Proposition 3.9]{Gar20} and \cite[Theorem 2.14]{RV14}).
As mentioned in Remark \ref{rem:approx}, our arguments work in the case of more general approximations than those treated in the previous results.
\begin{prop}\label{thm:Mscale}
For any $\alpha\in\mathbb{R}$ and $p\in[1,\infty)$ there exists a constant $C>0$ such that, for any $N\in\mathbb{N}$ and $\lambda \in (0,1]$,
\begin{align*}
& \mathbb{E}\left[ \left( \int _{B(0, \lambda/2)} \GMC_N^{(\alpha)} (x) \hspace{0.5mm} {\dd}x \right) ^p\right] 
\leq C\lambda^{2p-\alpha^2p(p-1)/4\pi}
\mathbb{E}\left[ \left( \int _{\Lambda} \GMC_N^{(\alpha)} (x) \hspace{0.5mm} {\dd}x \right) ^p\right] .
\end{align*}
\end{prop}
\begin{proof}
Consider the random field $x\mapsto\GFF_N(\lambda x)$.
The inequality
$$
\log\left(|\lambda x|\vee2^{-N}\right)
\ge\log\left(|x|\vee2^{-N}\right)+\log\lambda
$$
is easily checked by considering the three cases separately; $\lambda|x|<|x|\le2^{-N}$, $\lambda|x|\le2^{-N}<|x|$, and $2^{-N}<\lambda|x|<|x|$.
By the estimate \eqref{estimate:GN1}, for $x,y\in\mathbb{R}^2$ with $|x|\vee|y|<1/2$,
\begin{align*}
\mathbb{E}[\GFF_N(\lambda x)\GFF_N(\lambda y)]
&=-\frac1{2\pi}\log\left(|\lambda (x-y)|\vee2^{-N}\right)+O(1)\\
&\le-\frac1{2\pi}\log\left(|x-y|\vee2^{-N}\right)-\frac1{2\pi}\log\lambda+O(1)\\
&\le \mathbb{E}[\GFF_N(x)\GFF_N(y)]-\frac1{2\pi}\log\lambda+c
\end{align*}
for some constant $c>0$ independent of $\lambda$, $x$, $y$ and $N$. Hence by introducing a centered Gaussian random variable $Y_\lambda$ with variance $- (1/2\pi )\log \lambda + c$, independent of $\GFF$, we have
\begin{align*}
\mathbb{E}[\GFF_N(\lambda x)\GFF_N(\lambda y)]
\le \mathbb{E}[(\GFF_N(x)+Y_\lambda)(\GFF_N(y)+Y_\lambda)].
\end{align*}
Then Lemma \ref{lemma Kahane} yields
\begin{align*}
\mathbb{E}\left[ \left( \int _{|x|<1/2} \GMC_N^{(\alpha)} (\lambda x) \hspace{0.5mm} {\dd}x \right) ^p\right] 
&\leq \mathbb{E}\left[\exp \left( \alpha p Y_\lambda - \frac{\alpha ^2p}{2} \mathbb{E}[Y_\lambda ^2] \right) \right] \mathbb{E}\left[ \left( \int _{|x|<1/2} \GMC_N^{(\alpha)} (x) \hspace{0.5mm} {\dd}x \right) ^p\right]\\
&=C\exp \left( -\frac{\alpha ^2 p(p-1)}{4\pi} \log \lambda  \right)
\mathbb{E}\left[ \left( \int _{|x|<1/2} \GMC_N^{(\alpha)} (x) \hspace{0.5mm} {\dd}x \right) ^p\right]
\end{align*}
for some constant $C>0$.
By changing the variable $y=\lambda x$ we obtain the assertion.
\end{proof}

The following lemmas are useful to show the uniform integrability of 
$\int_{\Lambda} \GMC_N^{(\alpha)}(x)\hspace{0.5mm} {\dd}x$.

\begin{lem}\label{lem:expWick12}
For $\alpha\in\mathbb{R}$ and $p\in[1,2]$ there exists a constant $C>0$ such that, for any $N\in {\mathbb N}$ and $\delta \in (0,1/4]$,
\begin{align*}
& \mathbb{E}\left[ \left( \iint _{|x|\vee|y|<1/2,\, |x-y|<\delta} 
\GMC_N^{(\alpha)} (x) \GMC_N^{(\alpha)} (y) \hspace{0.5mm} {\dd}x 
\hspace{0.2mm} {\dd}y\right) ^{p/2} \right] \\
& \leq C\delta ^{(2-\alpha ^2p/4\pi)(p-1)} 
\mathbb{E}\left[ \left( \int _{\Lambda} \GMC_N^{(\alpha)} (x) \hspace{0.5mm} {\dd}x \right) ^p\right] .
\end{align*}
\end{lem}

\begin{proof}
For any $\delta \in (0,1/4]$ we can choose $\{ x_i ; i=1,2,\dots , n_\delta\}$ such that
\[
B(0,1/2) \subset \bigcup _{i=1}^{n_\delta } B(x_i, \delta) , \qquad n_\delta \leq c \delta ^{-2},
\]
where $c$ is an absolute constant.
Since
\begin{align*}
&\iint _{|x|\vee|y|<1/2,\, |x-y|<\delta} \GMC_N^{(\alpha)}(x) \GMC_N^{(\alpha)}(y) \hspace{0.5mm} {\dd}x 
\hspace{0.2mm} {\dd}y \\
&\le \int _{|x|<1/2} \GMC_N^{(\alpha)}(x) \left( \int _{B(x,\delta )} \GMC_N^{(\alpha)}(y) \hspace{0.5mm}
{\dd}y\right) {\dd}x \\
&\le \sum _{i=1}^{n_\delta} \int _{B(x_i, \delta)} \GMC_N^{(\alpha)}(x) \left( \int _{B(x,\delta )} 
\GMC_N^{(\alpha)}(y) \hspace{0.5mm} {\dd}y\right) {\dd}x \\
&\leq \sum _{i=1}^{n_\delta} \left( \int _{B(x_i, \delta)} \GMC_N^{(\alpha)}(x) \hspace{0.5mm} {\dd}x \right) 
\left( \int _{B(x_i, 2\delta )} \GMC_N^{(\alpha)}(y) \hspace{0.5mm} {\dd}y\right) \\
&\leq \sum _{i=1}^{n_\delta} \left( \int _{B(x_i, 2\delta)} \GMC_N^{(\alpha)}(x) \hspace{0.5mm} {\dd}x \right) ^2,
\end{align*}
we have by the elementary inequality $(a+b)^{p/2}\le a^{p/2}+b^{p/2}$ for $a,b\ge0$ and the shift invariance of the law of $\GMC_N^{(\alpha)}$,
\begin{equation*}
\mathbb{E}\left[ \left( \iint _{|x|\vee|y|<1/2,\,|x-y|<\delta} \GMC_N^{(\alpha)}(x) 
\GMC_N^{(\alpha)}(y) \hspace{0.5mm} {\dd}x \hspace{0.2mm} {\dd}y\right) ^{p/2} \right]  
\leq c\delta ^{-2} \mathbb{E}\left[ \left( \int _{B(0, 2\delta)} \GMC_N^{(\alpha)}(x) 
\hspace{0.5mm} {\dd}x \right) ^p\right] .
\end{equation*}
Hence Proposition \ref{thm:Mscale} yields the conclusion.
\end{proof}
\begin{lem}\label{lem:expWick13}
For any $\alpha\in\mathbb{R}$ there exists a constant $C>0$ such that, 
for any $N\in\mathbb{N}$ and $\delta \in (0,1/4]$,
$$
\mathbb{E}\left[ \iint _{|x|\vee|y|<1/2,\,|x-y|\ge\delta} \GMC_N^{(\alpha)}(x) \GMC_N^{(\alpha)}(y) 
\hspace{0.5mm} {\dd}x \hspace{0.2mm} {\dd}y \right] 
\leq C(1+ \delta^{2-\alpha^2/2\pi}).
$$
\end{lem}
\begin{proof}
By the estimate \eqref{estimate:GN1},
\begin{align*}
&\mathbb{E}\left[ \iint _{|x|\vee|y|<1/2,\,|x-y|\ge\delta} \GMC_N^{(\alpha)}(x) \GMC_N^{(\alpha)}(y) 
\hspace{0.5mm} {\dd}x \hspace{0.2mm} {\dd}y \right] \\
&=e^{- \alpha^2 C_N} \iint _{|x|\vee|y|<1/2,\,|x-y|\ge\delta} \mathbb{E}
\big[ \exp\left(\alpha (\GFF_N(x) + \GFF_N(y)) \right) \big] \hspace{0.5mm} {\dd}x 
\hspace{0.2mm} {\dd}y \\
&=\iint _{|x|\vee|y|<1/2,\,|x-y|\ge\delta} e^{\alpha^2G_{N,N}(x,y)} 
\hspace{0.5mm} {\dd}x \hspace{0.2mm} {\dd}y
\\
&\lesssim \iint _{|x|\vee|y|<1/2,\,|x-y|\ge\delta} |x-y|^{-\alpha^2/2\pi}
\hspace{0.5mm} {\dd}x \hspace{0.2mm} {\dd}y
\lesssim 1+ \delta^{2-\alpha^2/2\pi}.
\end{align*}
\end{proof}

By using above estimates, we prove $L^p$-boundedness, in particular the uniform integrability, of $\int_\Lambda\GMC_N^{(\alpha)}(x){\dd}x$.
It has also proved in previous results (see e.g. \cite[Theorem 3.26]{Ber16} and \cite[Proposition 3.5]{RV10}).
As mentioned in Remark \ref{rem:approx}, our arguments  work in the case of more general approximations than those treated in the previous results.

\begin{prop}\label{prop:Mintegrable}
For any $|\alpha|<\sqrt{8\pi}$ and $p \in (1, 8\pi /\alpha ^2 )\cap(1,2]$,
\begin{equation*}\label{eq:propMintegrable1}
\sup _{N\in {\mathbb N}} \mathbb{E}\left[ \left( \int _\Lambda \GMC_N^{(\alpha)}(x) 
\hspace{0.5mm} {\dd}x \right) ^p\right] < \infty .
\end{equation*}
\end{prop}
\begin{proof}
Choosing finite points $\{x_i\}$ such that $\Lambda=[-\pi,\pi)^2\subset \bigcup_iB(x_i,1/2)$ and using the shift invariance of the law of $\GMC_N^{(\alpha)}$,
\begin{equation*}
\mathbb{E}\left[ \left( \int _{\Lambda} \GMC_N^{(\alpha)}(x) 
\hspace{0.5mm}
{\dd}x \right) ^p \right] 
\leq C^p \mathbb{E}\left[ \left( \int _{B(0,1/2)} \GMC_N^{(\alpha)}(x) 
\hspace{0.5mm}
{\dd}x \right) ^p \right] 
\end{equation*}
for an absolute constant $C>0$. Let $\delta \in (0,1/4]$ and we decompose
\begin{align*}
\mathbb{E}\left[ \left( \int _{B(0,1/2)} \GMC_N^{(\alpha)}(x) 
\hspace{0.5mm} {\dd}x \right) ^p \right] 
&\leq \mathbb{E}\left[ \left( \iint _{|x|\vee|y|<1/2,\, |x-y|<\delta} 
\GMC_N^{(\alpha)}(x) \GMC_N^{(\alpha)}(y) 
\hspace{0.5mm} {\dd}x \hspace{0.2mm} {\dd}y
\right) ^{p/2} \right] \\
&\quad+ \mathbb{E}\left[ \left( \iint _{|x|\vee|y|<1/2,\, |x-y|\ge\delta} 
\GMC_N^{(\alpha)}(x) \GMC_N^{(\alpha)}(y)
\hspace{0.5mm} {\dd}x \hspace{0.2mm} {\dd}y
\right) ^{p/2} \right]\\
&\leq \mathbb{E}\left[ \left( \iint _{|x|\vee|y|<1/2,\, |x-y|<\delta} 
\GMC_N^{(\alpha)}(x) \GMC_N^{(\alpha)}(y) 
\hspace{0.5mm} {\dd}x \hspace{0.2mm} {\dd}y
\right) ^{p/2} \right] \\
&\quad+ \mathbb{E}\left[ \iint _{|x|\vee|y|<1/2,\, |x-y|\ge\delta} 
\GMC_N^{(\alpha)}(x) \GMC_N^{(\alpha)}(y)
\hspace{0.5mm} {\dd}x \hspace{0.2mm} {\dd}y
\right]^{p/2}.
\end{align*}
In the second inequality, we use $p\le2$ and the nonnegativity of $\GMC_N^{(\alpha)}$.
Applying Lemmas \ref{lem:expWick12} and \ref{lem:expWick13}, we have
\[
\mathbb{E}\left[ \left( \int _{\Lambda} \GMC_N^{(\alpha)}(x) \hspace{0.5mm} {\dd}x \right) ^p \right] 
\leq C'\delta ^{(2-\alpha ^2 p/4\pi )(p-1)}  
\mathbb{E}\left[ \left( \int _{\Lambda} \GMC_N^{(\alpha)}(x) \hspace{0.5mm} {\dd}x \right) ^p\right] 
+C'\delta^{p(1-\alpha^2/4\pi)},
\]
where the constant $C'$ is independent of $N$ and $\delta$.
Since $\alpha^2p<8\pi$, by choosing sufficiently small $\delta$, we complete the proof.
\end{proof}
\begin{cor}\label{cor:unif int M}
For any parameters $p$ and $\beta$ as in \eqref{def p beta}, one has
$$
\sup_{N\in\mathbb{N}}\mathbb{E} 
\left[ \big \Vert \GMC_N^{(\alpha)} \big \Vert_{B_{p,p}^{-\beta}}^p \right]<\infty.
$$
\end{cor}
\begin{proof}
By definition of the Besov norm,
\begin{align*}
\mathbb{E} \left [ \big \Vert \GMC_N^{(\alpha)} \big \Vert_{B_{p,p}^{-\beta}}^p \right ]
=\sum_{j=-1}^\infty2^{-j\beta p}\mathbb{E} \left [ \big \| \Delta_j \GMC_N^{(\alpha)} \big \|_{L^p}^p \right ]
=\sum_{j=-1}^\infty2^{-j\beta p}\int_\Lambda\mathbb{E} \left [ \big | \Delta_j \GMC_N^{(\alpha)}(x) \big |^p \right ]  {\dd}x.
\end{align*}
By the shift invariance of the law of $\GMC_N^{(\alpha)}$, it is sufficient to consider $\mathbb{E}\big [|\Delta_j \GMC_N^{(\alpha)} (0)|^p \big]$. 
The bounds for $j=-1,0$ are obvious in view of Proposition \ref{prop:Mintegrable}.
For $j\ge1$, by using Mikowski's inequality, rapid decay of the Schwartz function $\mathcal{F}^{-1}\rho$, and the shift invariance of the law of $\GMC_N^{(\alpha)}$,
\begin{align*}
\mathbb{E}\left [ \big | \Delta_j \GMC_N^{(\alpha)}(0) \big |^p \right ]^{1/p}
&=
\left\|\int_{\mathbb{R}^2}(\mathcal{F}^{-1}\rho)(x)\GMC_N^{(\alpha)}(2^{-j}x)
{\hspace{0.5mm}}{\dd}x \right\|_{L^p(\mathbb{P})}
\\
&\lesssim\sum_{k \in\mathbb{Z}^2} (1+|k|)^{-3} \left\|
 \int_{B(k,1)} \GMC_N^{(\alpha)}(2^{-j}x)
 \hspace{0.5mm} {\dd}x \right\|_{L^p(\mathbb{P})}\\
&\lesssim \left\|
 \int_{B(0,1)} \GMC_N^{(\alpha)}(2^{-j}x)
 \hspace{0.5mm} {\dd}x \right\|_{L^p(\mathbb{P})}.
\end{align*}
Hence by Proposition \ref{thm:Mscale},
$$
\mathbb{E}\left [ \big | \Delta_j \GMC_N^{(\alpha)}(0) \big |^p \right ]^{1/p}
\lesssim(2^{-j})^{-\alpha^2(p-1)/4\pi}.
$$
Therefore, we obtain $\mathbb{E}\big [ \big \| \GMC_N^{(\alpha)} \big \|_{B_{p,p}^{-\beta}}^p
\big] \lesssim1$ for $\beta>\alpha^2(p-1)/4\pi$.
\end{proof}
\subsection{Almost-sure convergence}\label{sec:asconv}
In this subsection, we show the almost-sure weak convergence of 
$\GMC_N^{(\alpha)}$ as $N\to\infty$ in the space of positive Borel measures on $\Lambda$,
and we finally complete the proof of Theorem \ref{thm:expwick almostsure}.
We apply the following proposition several times, which follows from direct computation.

\begin{prop}\label{prop:Gauss1}
Let $X$ be an $n$-dimensional centered Gaussian random vector with a covariance matrix $V$.
Then, for $a\in {\mathbb R}^n$ and a Borel function $f$ on ${\mathbb R}^n$,
\[
\mathbb{E}\left[ e^{a\cdot X} f(X) \right] = e^{a\cdot (Va)/2} 
\mathbb{E}\left[f(X+Va)\right].
\]
\end{prop}

The following theorem plays a crucial role to prove Theorem \ref{thm:expwick almostsure}.

\begin{thm}\label{thm:asconv}
Let $|\alpha|<\sqrt{8\pi}$.
Then, there exist positive constants $c$ and $C$ such that
\begin{equation}\label{eq:thmasconv-00}
\mathbb{E}
\big [ \big | \big \langle f,\GMC_{N+1}^{(\alpha)} \big \rangle 
- \big \langle f,\GMC_N^{(\alpha)} \big \rangle \big | 
\big]
\le C \|f\|_{C(\Lambda)}2^{-c N}
\end{equation}
for any $N\in {\mathbb N}$ and $f\in C(\Lambda )$.
\end{thm}

\begin{proof}
Our proof is based on the same spirit as \cite[Sections 3 and 4]{Ber17}.
It is well-known that the limiting measure $\GMC^{(\alpha)}$ must be supported on the points $x$ such that 
$$
\lim_{N\to\infty}\frac{\GFF_N(x)}{C_N}=\alpha,
$$
called {\it{$\alpha$-thick points}}.
An essential point of \cite{Ber17} is to decompose $\GMC_N^{(\alpha)}$ into two parts:
\begin{align*}
\GMC_N^{<}(x):=\GMC_N^{(\alpha)}(x)\prod_{n_0\le n\le N}{\bf 1}_{\{\GFF_n(x)\le\alpha(1+\delta)C_n\}},\qquad
\GMC_N^{>}(x):=\GMC_N^{(\alpha)}(x)-\GMC_N^{<}(x),
\end{align*}
for some fixed $n_0\in\mathbb{N}$ and $\delta>0$.
Then, $L^1$ contribution of $\GMC_N^{>}$ can be eliminated, while $\GMC_N^{<}$ has a good control in $L^2$ depending on the choice of $n_0$ and $\delta$.
However, we need the following modifications to obtain the stronger estimate \eqref{eq:thmasconv-00}.
\begin{itemize}
\item Let $n_0=\delta^3N$ be a \emph{variable} depending on $\delta$ and $N$.
\item Replace the indicator function ${\bf1}$ with some Lipschitz function.
\end{itemize}

Now we start the proof of this theorem.
Denote by $\tilde{B}(x,r)$ the open ball in $\Lambda$ centered at $x$ and with radius $r$ under the canonical metric of $\Lambda$.
It is sufficient to show \eqref{eq:thmasconv-00} for $f\in C(\Lambda)$ with ${\rm supp}f \subset \tilde{B}(0, 1/2)$.
Indeed, we obtain the assertion for general $f\in C(\Lambda)$, once we apply the finite decomposition $f=\sum_kf_k$ with $f_k$ supported in some ball $\tilde{B}(x_k,1/2)$ and the shift invariance of the law of $\GMC_N^{(\alpha)}$.
Hence we assume $|x|\vee|y|<1/2$ throughout this proof.

As introduced in Section \ref{sec:estimate of G}, we set $G_{M,N}(x,y)=\mathbb{E}[\GFF_M(x)\GFF_N(y)]$ and set $C_{M,N}=G_{M,N}(x,x)$
for $M,N \in \mathbb N$.
By the estimate \eqref{estimate:GN1}, for any $x,y\in\mathbb{R}^2$ with $|x|\vee|y|<1/2$ and any $M,N\in\mathbb{N}$ with $M\le N$, we have 
\begin{align*}
G_{M,N}(x,y)=-\frac1{2\pi}\log\left(|x-y|\vee2^{-M}\right)+O(1),\qquad
C_{M,N}=\frac{M}{2\pi}\log2+O(1).
\end{align*}
These yield the following: for any sufficiently small $\delta>0$, there exists an integer $N_\delta'$ depending on $\delta$ such that, for any $N_\delta'\le M\le N$ and $|x|\vee|y|<1/2$
\begin{align}\label{eq:aymp c and g 00}
\frac1{\tilde C_M}\le\delta^3,\qquad
\left|\frac{C_{M,N}-\tilde C_M}{\tilde C_M}\right|\le \delta^3,\qquad
\left|\frac{G_{M,N}(x,y)-\tilde G_M(x,y)}{\tilde C_M}\right|\le \delta^3,
\end{align}
where
\begin{align*}
\tilde C_M:=\frac{M}{2\pi}\log2,\qquad
\tilde G_M(x,y):=-\frac1{2\pi}\log\left(|x-y|\vee2^{-M}\right) .
\end{align*}
The parameter $\delta$ is to be chosen later, as a sufficiently small number compared with $1-\alpha^2/8\pi$ and the exponent $\theta$ in the estimate \eqref{estimate:GN2}.

Furthermore, let $\chi _\delta$ be a function on $\mathbb{R}$ such that
\begin{align*}
\chi _\delta (\tau)=
\left\{
\begin{aligned}
&1,&\qquad&\tau \le\delta,\\
&-\tau/\delta+2,&\qquad&\delta\le\tau\le2\delta,\\
&0,&\qquad&\tau \ge2\delta.
\end{aligned}
\right.
\end{align*}
Then we define for each $N,i\in\mathbb{N}$ such that $N\le i$ (actually we will let $i=N$ or $N+1$),
\begin{align*}
\GMC_{N,i}^<(x)&:=\GMC_{i}^{(\alpha)}(x)
\prod_{\delta^3N \le n\le i}\chi _\delta \left(\frac{\GFF_n(x)-\alpha C_{n,i}}{\alpha \tilde C_n}\right),
\qquad
\GMC_{N,i}^>(x):=\GMC_i^{(\alpha)} (x)
-\GMC_{N,i}^<(x).
\end{align*}
Let $N_\delta$ be an integer such that $N_\delta\ge N_\delta'/\delta^3$.
From \eqref{eq:aymp c and g 00} we have that, if $N\ge N_\delta$, then for any integers $m,n$ with $\delta^3N\le m\le n$ and $|x|\vee|y|<1/2$,
\begin{align}\label{eq:aymp c and g}
\frac1{\tilde C_m}\le\delta^3,\qquad
\left|\frac{C_{m,n}-\tilde C_m}{\tilde C_m}\right|\le \delta^3,\qquad
\left|\frac{G_{m,n}(x,y)-\tilde G_m(x,y)}{\tilde C_m}\right|\le \delta^3.
\end{align}
We assume $N\ge N_\delta$ throughout this proof, and decompose
$$
\GMC_{N+1}^{(\alpha)}-\GMC_N^{(\alpha)}
=(\GMC_{N,N+1}^<-\GMC_{N,N}^<)+(\GMC_{N,N+1}^>-\GMC_{N,N}^>).
$$

\medskip

\noindent
{\bf(1) The terms $\GMC_{N,N+1}^>$ and $\GMC_{N,N}^>$.}
For any fixed $i\in\{N,N+1\}$ and $x\in\Lambda$, we apply Proposition \ref{prop:Gauss1} to the $(i-[\delta^3N]+1)$-dimensional random vector $X=(\GFF_n(x))_{\delta^3N\le n\le i}$ and a fixed vector $a=(0,\dots,0,\alpha)$. Then, since $Va=(\alpha C_{n,i})_{\delta^3N\le n\le i}$, we have
\begin{align*}
\mathbb{E}[\GMC_{N,i}^>(x)]
&=\mathbb{E}\Bigg[e^{a\cdot X-a\cdot Va/2}
\Bigg\{1-\prod_{\delta^3N \le n\le i}\chi _\delta \left(\frac{\GFF_n(x)-\alpha C_{n,i}}{\alpha \tilde C_n}\right)\Bigg\}\Bigg]\\
&=\mathbb{E}\Bigg[1-\prod_{\delta^3N \le n\le i}\chi _\delta \left(\frac{\GFF_n(x)}{\alpha \tilde C_n}\right)\Bigg]\\
&\leq \sum_{\delta^3N\le n\le i}\mathbb{E}\left[1-\chi _\delta \left(\frac{\GFF_n(x)}{\alpha \tilde C_n}\right)\right]\\
&\le \sum_{\delta^3N\le n\le i}\mathbb{P}\left(\GFF_n(x)\ge\delta\alpha \tilde C_n\right),
\end{align*}
where we used the elementary inequality 
$$
1-\prod_{n=1}^Ka_n\le\sum_{n=1}^K(1-a_n), \qquad a_1,\dots,a_K\in[0,1].
$$
Since $\GFF_n(x)$ has a variance $C_{n,n}$ and \eqref{eq:aymp c and g} implies $C_{n,n}=(1+o(\delta))\tilde C_n$,
we have by the tail estimate of the normal distribution,
\begin{align*}
\mathbb{E}[\GMC_{N,i}^>(x)]
&\le \sum_{\delta^3N\le n\le i}
\mathbb{P}\left(\frac{\GFF_n(x)}{\sqrt{C_{n,n}}}\ge \alpha(\delta+o(\delta))\sqrt{{\tilde C}_n} \right)\\
&\le C_\delta\sum_{\delta^3N\le n\le i} e^{-\alpha^2(\delta+o(\delta))^2\tilde C_n/2}
\le C_\delta'2^{-\alpha^2(\delta+o(\delta))^2 \delta^3N/4\pi}
\end{align*}
for some positive constants $C_\delta$ and $C_\delta'$ depending on $\delta$.
Therefore, we obtain the exponential decay \eqref{eq:thmasconv-00} for $\mathbb{E}\left[\vert\langle f,\GMC_{N,N+1}^>\rangle-\langle f,\GMC_{N,N}^>\rangle\vert\right]$.
\medskip

\noindent
{\bf(2) The difference $\GMC_{N,N+1}^<-\GMC_{N,N}^<$.}
We actually show the stronger estimate
\begin{equation*}\label{eq:thmasconv-00L2}
\mathbb{E} \left [
\left| \left\langle f, \GMC_{N,N+1}^< \right\rangle 
- \left\langle f, \GMC_{N,N}^<  \right\rangle \right|^2 \right ]
\le C_\delta\|f\|_{C(\Lambda)}^22^{-c_\delta N}
\end{equation*}
than \eqref{eq:thmasconv-00} with replacement of $\GMC_{N+1}^{(\alpha )}$ and $\GMC _N^{(\alpha )}$ by $\GMC_{N,N+1}^<$ and $\GMC_{N,N}^<$, respectively.
We write the expectation as the form $\iint_{\Lambda^2}f(x)f(y)\mathfrak{M}_N(x,y)
\hspace{0.5mm} {\dd}x \hspace{0.2mm} {\dd}y$, where
$$
\mathfrak{M}_N(x,y)
=\mathbb{E}\left[ 
\left( \GMC_{N,N+1}^<(x)-\GMC_{N,N}^<(x) \right) 
\left( \GMC_{N,N+1}^<(y)-\GMC_{N,N}^<(y)  \right) \right],
$$
and consider the integral
\begin{equation}\label{eq:thmasconv-00M}
\iint_{|x|\vee|y|<1/2}\left|\mathfrak{M}_N(x,y)\right|
\hspace{0.5mm} {\dd}x \hspace{0.2mm} {\dd}y.
\end{equation}
Moreover, we decompose the integrand by
\begin{align*}
\mathfrak{M}_N(x,y)=I_{N+1,N+1}(x,y) - I_{N+1,N}(x,y) -I_{N,N+1}(x,y) +I_{N,N}(x,y),
\end{align*}
where $I_{i,j}(x,y):=\mathbb{E}[\GMC_{N,i}^<(x)\GMC_{N,j}^<(y)]$ ($i,j=N, N+1$).
For any fixed $x,y\in\Lambda$, we apply Proposition \ref{prop:Gauss1} to the multidimensional Gaussian random variable
\begin{align*}
\mathfrak{X}=\Big((\GFF_n(x))_{\delta^3N\le n\le i},(\GFF_m(y))_{\delta^3N\le m\le j}\Big)
\end{align*}
and a fixed vector $a\in\mathbb{R}^{(i-[\delta^3N] +1)+(j-[\delta^3N] +1)}$ such that $a\cdot\mathfrak{X}=\alpha(\GFF_i(x)+\GFF_j(y))$. Since the covariance matrix $V$ of $\mathfrak{X}$ is given by
\begin{align*}
Va&=\alpha\Big((C_{n,i}+G_{n,j}(x,y))_{\delta^3N\le n\le i},(G_{m,i}(x,y)+C_{m,j})_{\delta^3N\le m\le j}\Big),\\
a\cdot Va&=\alpha^2 (C_{i,i}+C_{j,j}+2G_{i,j}(x,y)),
\end{align*}
Proposition \ref{prop:Gauss1} yields
\begin{align*}
&I_{i,j}(x,y)\\
&=e^{\alpha^2G_{i,j}(x,y)}\mathbb{E}\bigg[
e^{a\cdot\mathfrak{X}-a\cdot Va/2}
\prod_{\delta^3N \le n\le i}\chi _\delta \left(\frac{\GFF_n(x)-\alpha C_{n,i}}{\alpha \tilde C_n}\right)
\prod_{\delta^3N \le m\le j}\chi _\delta \left(\frac{\GFF_m(x)-\alpha C_{m,j}}{\alpha \tilde C_m}\right)
\bigg]\\
&=e^{\alpha^2G_{i,j}(x,y)}
\mathbb{E} \bigg [ 
\prod_{\delta^3N\le n\le i}\chi _\delta \left(\frac{\GFF_n(x)+\alpha G_{n,j}(x,y)}{\alpha \tilde C_n}\right)
\prod_{\delta^3N\le m\le j}\chi _\delta \left(\frac{X_m(y)+\alpha G_{m,i}(x,y)}{\alpha \tilde C_m}\right)
\bigg ].
\end{align*}
We decompose the integral \eqref{eq:thmasconv-00M} into the two regions
$$
|x-y|<2^{-\delta^3N},\qquad
2^{-\delta^3N}\le |x-y|<1.
$$

\medskip

\noindent
{\bf (2-1) The integral over $|x-y|<2^{-\delta^3N}$.}
We estimate each $I_{i,j}$ ($i,j=N,N+1$) separately. Assume $i\le j$ without loss of generality. 
We further decompose the integral into two regions
$$
|x-y|<2^{-i},\qquad
2^{-i}\le |x-y|<2^{-\delta^3N}.
$$

\noindent
{\bf (2-1-1) The integral over $|x-y|<2^{-i}$.}
Since $\chi_\delta\le1$,
\begin{align*}
I_{i,j}(x,y)\le e^{\alpha^2G_{i,j}(x,y)}\mathbb{E}\left[\chi _\delta \left(\frac{\GFF_i(x)+\alpha G_{i,j}(x,y)}{\alpha \tilde C_i}\right)\right].
\end{align*}
Since $|x-y|<2^{-i}$, \eqref{eq:aymp c and g} implies that 
$$
G_{i,j}(x,y)=\tilde G_i(x,y)+o(\delta)\tilde C_i=(1+o(\delta))\tilde C_i.
$$
Hence we have
\begin{align*}
\mathbb{E}\left[\chi _\delta \left(\frac{\GFF_i(x)+\alpha G_{i,j}(x,y)}{\alpha \tilde C_i}\right)\right]
&\le\mathbb{P}\left(\GFF_i(x)\le(-1+2\delta+o(\delta))\alpha \tilde C_i\right)\\
&\lesssim e^{-(1+O(\delta))\alpha^2\tilde C_i/2}
\lesssim2^{-(1+O(\delta))\alpha^2N/4\pi}.
\end{align*}
Since $e^{\alpha^2G_{i,j}(x,y)}\lesssim |x-y|^{-\alpha^2/2\pi} \lesssim 2^{\alpha^2N/2\pi}$ by the estimate \eqref{estimate:GN1},
we obtain
\begin{align*}
\iint_{|x|\vee|y|<1/2,\, |x-y|<2^{-i}}I_{i,j}(x,y)
\hspace{0.5mm} {\dd}x \hspace{0.2mm} {\dd}y
&\lesssim \iint_{|x|\vee|y|<1/2,\, |x-y|<2^{-i}}2^{(1+O(\delta))\alpha^2N/4\pi}
\hspace{0.5mm} {\dd}x \hspace{0.2mm} {\dd}y
\\
&\lesssim 2^{N\left({\alpha^2}/{4\pi}-2+O(\delta)\right)}.
\end{align*}
This decays exponentially if $\alpha^2<8\pi$ and $\delta$ is chosen sufficiently small.
\medskip

\noindent
{\bf (2-1-2) The integral over $2^{-i}\le|x-y|<2^{-\delta^3N}$.}
The argument is similar to {\bf (2-1-1)}.
For any $x,y$ in this region, 
there exists an integer $n_{x,y}\in[\delta^3N,i]$ satisfying $2^{-n_{x,y}}\le|x-y|<2^{-n_{x,y}+1}$.
For such $n_{x,y}$, we have
\begin{align*}
I_{i,j}(x,y)\le e^{\alpha^2G_{i,j}(x,y)}\mathbb{E}\left[\chi _\delta \left(\frac{\GFF_{n_{x,y}}(x)+\alpha G_{n_{x,y},j}(x,y)}{\alpha \tilde C_{n_{x,y}}}\right)\right].
\end{align*}
Since \eqref{eq:aymp c and g} implies
\begin{align*}
G_{n_{x,y},j}(x,y)&=\tilde G_{n_{x,y}}(x,y)+o(\delta)\tilde C_{n_{x,y}}
=(1+o(\delta))\tilde C_{n_{x,y}},
\end{align*}
similarly to the argument in {\bf (2-1-1)} we have
\begin{align*}
\mathbb{E}\bigg[\chi _\delta \bigg(\frac{\GFF_{n_{x,y}}(x)+\alpha G_{{n_{x,y}},j}(x,y)}{\alpha \tilde C_{n_{x,y}}}\bigg)
\bigg]
&\le\mathbb{P}\left(\GFF_{n_{x,y}}(x)\le(-1+2\delta+o(\delta)) \alpha\tilde C_{n_{x,y}}\right)\\
&\lesssim e^{-(1+O(\delta))\alpha^2\tilde C_{n_{x,y}}/2}\\
&\lesssim2^{-(1+O(\delta))\alpha^2n_{x,y}/4\pi}
\asymp |x-y|^{(1+O(\delta))\alpha^2/4\pi}.
\end{align*}
On the other hand, by the estimate \eqref{estimate:GN1}, $e^{\alpha^2G_{i,j}(x,y)}\lesssim |x-y|^{-\alpha^2/2\pi}$.
Hence we have
\begin{align*}
&
\iint_{|x|\vee|y|<1/2,\,2^{-i}\le|x-y|<2^{-\delta^3N}}
I_{i,j}(x,y)
\hspace{0.5mm} {\dd}x \hspace{0.2mm} {\dd}y
\\
&\lesssim \iint_{|x|\vee|y|<1/2,\,2^{-i}\le|x-y|<2^{-\delta^3N}}
|x-y|^{-{\alpha^2}/{4\pi}+O(\delta)}
\hspace{0.5mm} {\dd}x \hspace{0.2mm} {\dd}y
\\
&\lesssim \int_{2^{-i}\le|x|<2^{-\delta^3N}}|x|^{-\alpha^2/4\pi+O(\delta)}
\hspace{0.5mm} {\dd}x \\
&\lesssim \int _{2^{-i}}^{2^{-\delta^3N}} r^{-{\alpha^2}/{4\pi}+1+O(\delta)} 
\hspace{0.5mm} {\dd}r
\\
&\lesssim 2^{\delta^3N\left({\alpha^2}/{4\pi}-2+O(\delta)\right)}
\end{align*}
if $\alpha^2<8\pi$.
This decays exponentially if $\delta$ is chosen sufficiently small.
\medskip

\noindent
{\bf (2-2) The integral over $|x-y|\ge2^{-\delta^3N}$.}
We have to consider combinations of $I$ terms.
We consider only $I_{N+1,N}-I_{N,N}$, since the other difference $I_{N+1,N+1}-I_{N,N+1}$ is estimated by a similar way.
For simplicity, we write
$$
\chi_n^j(x)=\chi _\delta \left(\frac{\GFF_n(x)+\alpha G_{n,j}(x,y)}{\alpha \tilde C_n}\right),\quad
\chi_m^i(y)=\chi _\delta \left(\frac{\GFF_m(y)+\alpha G_{m,i}(x,y)}{\alpha \tilde C_n}\right) .
$$
Now we decompose
\begin{align*}
&I_{N+1,N}(x,y)-I_{N,N}(x,y)\\
&=e^{\alpha^2G_{N+1,N}(x,y)}\mathbb{E}\bigg[\prod_{\delta^3N\le n\le N+1}\chi_n^N(x)
\prod_{\delta^3N\le m\le N}\chi_m^{N+1}(y)\bigg]\\
&\quad-e^{\alpha^2G_{N,N}(x,y)}\mathbb{E}\bigg[\prod_{\delta^3N\le n\le N}\chi_n^N(x)
\prod_{\delta^3N\le m\le N}\chi_m^{N}(y)\bigg]\\
&=\left(e^{\alpha^2G_{N+1,N}(x,y)}-e^{\alpha^2G_{N,N}(x,y)}\right)
\mathbb{E}\bigg[\prod_{\delta^3N\le n\le N+1}\chi_n^{N}(x)\prod_{\delta^3N\le m\le N}\chi_m^{N+1}(y)\bigg]\\
&\quad+e^{\alpha^2G_{N,N}(x,y)}
\mathbb{E}\bigg[\left(\chi_{N+1}^N(x)-1\right)\prod_{\delta^3N\le n\le N}\chi_n^{N}(x)\prod_{\delta^3N\le m\le N}\chi_m^{N+1}(y)\bigg]\\
&\quad+e^{\alpha^2G_{N,N}(x,y)}
\mathbb{E}\bigg [ \sum_{\delta^3N\le m_0\le N}
\Big \{ 
\Big(\prod_{\delta^3N\le n\le N}\chi_n^{N}(x)\Big)
\left(\chi_{m_0}^{N+1}(y)-\chi_{m_0}^N(y)\right)
\\
&\hspace{50mm}
\times
\prod_{\delta^3N\le m<m_0}\chi_m^{N}(y)
\prod_{m_0<m\le N}\chi_m^{N+1}(y)
\Big \}  \bigg ]\\
&=:J_1(x,y)+J_2(x,y)+J_3(x,y).
\end{align*}
In the region $|x-y|\ge2^{-\delta^3N}$, we have no choice but to do
\begin{align*}
\prod_{\delta^3N\le n\le N}\chi_n^N(z)\le1.
\end{align*}
However, we can use the estimate \eqref{estimate:GN2}. Indeed,
\begin{align*}
&\big|e^{\alpha^2G_{N+1,N}(x,y)}-e^{\alpha^2G_{N,N}(x,y)}\big|\\
&\lesssim |G_{N+1,N}(x,y)-G_{N,N}(x,y)|
\left(e^{\alpha^2G_{N+1,N}(x,y)}\vee e^{\alpha^2G_{N,N}(x,y)}\right)\\
&\lesssim |G_{N+1,N}(x,y)-G_{N,N}(x,y)|\, \cdot |x-y|^{-\alpha^2/2\pi}
\end{align*}
and
\begin{align*}
&\left|\chi^{N+1}_{m_0}(y)-\chi^N_{m_0}(y)\right|\\
&=\left|\chi _\delta \left(\frac{\GFF_{m_0}(y)+\alpha G_{m_0,N+1}(x,y)}{\alpha \tilde C_{m_0}}\right)
-\chi _\delta \left(\frac{\GFF_{m_0}(y)+\alpha G_{m_0,N}(x,y)}{\alpha \tilde C_{m_0}}\right)\right|
\\ &
\lesssim_\delta \left|G_{m_0,N+1}(x,y)-G_{m_0,N}(x,y)\right|.
\end{align*}
Hence by the estimate \eqref{estimate:GN2} we have
\begin{align*}
&\iint_{|x|\vee|y|<1/2,\,|x-y|\ge2^{-\delta^3N}}\left(|J_1(x,y)|+|J_3(x,y)|\right)
 {\dd}x \hspace{0.2mm} {\dd}y
\\
&\lesssim \sum_{\delta^3N\le m_0\le N}
\iint_{|x|\vee|y|<1/2,\,|x-y|\ge2^{-\delta^3N}}
|G_{m_0,N+1}(x,y)-G_{m_0,N}(x,y)|\,|x-y|^{-{\alpha^2}/{2\pi}}
\hspace{0.5mm} {\dd}x \hspace{0.2mm} {\dd}y
\\
&\lesssim 2^{\delta^3N\alpha^2/2\pi}\sum_{\delta^3N\le m_0\le N}
\iint_{|x|\vee|y|<1/2}|G_{m_0,N+1}(x,y)-G_{m_0,N}(x,y)|
\hspace{0.5mm} {\dd}x \hspace{0.2mm} {\dd}y
\\
&\lesssim N2^{N(\delta^3\alpha^2/2\pi-\theta)}.
\end{align*}
Since $\theta>0$, this decays exponentially if $\delta$ is chosen sufficiently small.

Finally we consider $J_2$.
The estimate \eqref{estimate:GN1} implies that for $2^{-\delta^3N}\le|x-y|\le1$,
\begin{align*}
\left|\frac{G_{N+1,N}(x,y)}{\tilde{C}_{N+1}}\right|
&\le\frac1{{\tilde C}_{N+1}}\left(-\frac1{2\pi}\log|x-y|+O(1)\right)\\
&\le\frac1{{\tilde C}_{N+1}}\left(\frac{\delta^3N}{2\pi}\log2 + O(1)\right)
\\
&
=o(\delta).
\end{align*}
Hence we have
\begin{align*}
\mathbb{E}\Big [ \left|\chi_{N+1}^{N}(x)-1\right| \Big ]
&=\mathbb{E}\bigg [ \left |
\chi _\delta \left (\frac{\GFF_{N+1}(x)+\alpha G_{N+1,N}(x,y)}{\alpha \tilde C_{N+1}}\right)
-1\right |
\bigg ]
\\
&\le \mathbb{P}\left(\GFF_{N+1}(x)+\alpha G_{N+1,N}(x,y)\ge \delta\alpha\tilde{C}_{N+1}\right)\\
&\le \mathbb{P}\left(\GFF_{N+1}(x)\ge(\delta+o(\delta))\alpha \tilde C_{N+1}\right) \\
&\lesssim_\delta e^{-(\delta+o(\delta))^2\alpha^2\tilde C_{N+1}/2} \\
&
\lesssim 2^{-(\delta+o(\delta))^2N\alpha^2/4\pi}.
\end{align*}
Therefore,
\begin{align*}
&\iint_{|x|\vee|y|<1/2,\, |x-y|\ge2^{-\delta^3N}} |J_2(x,y)|
\hspace{0.5mm} {\dd}x \hspace{0.2mm} {\dd}y
\\
&\lesssim \iint_{|x|\vee|y|<1/2,\, |x-y|\ge2^{-\delta^3N}}
|x-y|^{-\alpha^2/2\pi}2^{-(\delta+o(\delta))^2N\alpha^2/4\pi}
\hspace{0.5mm} {\dd}x \hspace{0.2mm} {\dd}y
\\
&
\lesssim 2^{-c_\delta N}
\end{align*}
with $c_\delta=(\delta+o(\delta))^2\alpha^2/4\pi-\delta^3\alpha^2/2\pi$. Since $c_\delta$ is positive for sufficiently small $\delta$, this completes the proof.
\end{proof}

\begin{cor}\label{cor:weak limit M}
For any $f\in C(\Lambda)$, the sequence $\{ \langle f,\GMC_N^{(\alpha)}\rangle \}_{N\in \mathbb N}$ 
converges almost surely and in $L^1({\mathbb P})$.
This limit is independent to the choice of $\psi$.
\end{cor}

\begin{proof}
Almost-sure convergence follows from Theorem \ref{thm:asconv}. Denote by $\langle f,\GMC_\infty^{(\alpha)} \rangle$ the limit. The uniqueness follows completely in the same way as the argument in \cite[Section 5]{Ber17}, but we provide a sketch of the proof for readers' convenience.
Let $\bar\psi={\bf 1}_{B(0,1)}$, the indicator function of the ball $B(0,1)$, and define $\bar P_N$ and $\bar{\GMC}_N^{(\alpha)}$ in a similarly way to $P_N$ and ${\GMC}_N^{(\alpha )}$, respectively, by $\bar\psi$ instead of $\psi$.
Since $\bar\psi$ satisfies Hypothesis \ref{hypo on psi}, there exists an almost-sure and $L^1$-limit $\langle f,\bar{\GMC}_\infty^{(\alpha)} \rangle$. Denote by $\mathcal{F}_n$ the filtration generated by $\{\hat{\GFF}(k)\}_{|k|<2^n}$.
Since $(1-\bar{P}_n)\GFF_N$ is independent of $\mathcal{F}_n$, we have
$$
\mathbb{E}[\langle f,\GMC_N^{(\alpha)}\rangle|\mathcal{F}_n]=\langle f,\bar{\GMC}_{N,n}^{(\alpha)}\rangle,
$$
where
\[
\bar{\GMC}_{N,n}^{(\alpha)}:= \exp \left( \alpha \bar{P}_n\GFF_N - \frac{\alpha ^2}{2}\bar{C}_{N,n}\right) , \quad \bar{C}_{N,n}=\mathbb{E}[(\bar{P}_n\GFF_N(x))^2] .
\]
Since $\bar{P}_n\GFF_N$ converges as $N\to\infty$ to $\bar{P}_n\GFF$ uniformly in $x\in\Lambda$ almost surely for each $n$, we have
$$
\langle f,\bar{\GMC}_n^{(\alpha)}\rangle
=\lim_{N\to\infty}
\langle f,\bar{\GMC}_{N,n}^{(\alpha)}\rangle
=\lim_{N\to\infty}
\mathbb{E}[\langle f,\GMC_N^{(\alpha)}\rangle|\mathcal{F}_n]
=\mathbb{E}[\langle f,\GMC_\infty^{(\alpha)}\rangle|\mathcal{F}_n].
$$
Letting $n\to\infty$, we have $\langle f,\bar{\GMC}_\infty^{(\alpha)}\rangle=\langle f,\GMC_\infty^{(\alpha)}\rangle$ almost surely.
\end{proof}

\begin{cor}\label{cor:weak limit M2}
Regard $\GMC_N^{(\alpha)}$ as a measure as in Theorem {\rm{\ref{thm:expwick almostsure}}}.
Then, the sequence $\{ \GMC_N^{(\alpha)}\}_{N\in \mathbb N}$ converges in the weak topology, almost surely.
\end{cor}

\begin{proof}
Let $D$ be a countable dense set in $C(\Lambda )$ which includes the constant function $1$.
Then, by Corollary \ref{cor:weak limit M} we have
\[
{\mathbb P}\left( \lim _{N\rightarrow \infty} \langle f, \GMC _N^{(\alpha )} \rangle \ \mbox{exists for all}\ f\in D \right) =1 .
\]
From now, in order to clarify the dependence of the randomness $\omega \in \Omega$, we denote $\GMC _N^{(\alpha )}$ with a sample $\omega \in \Omega$ by $\GMC _N^{(\alpha )}(\omega )$.
Let ${\mathcal N}\in {\mathcal F}$ be the event that $\lim _{N\rightarrow \infty} \langle f, \GMC _N^{(\alpha )} (\omega ) \rangle$ does not exists for some $f\in D$.
For each $\omega \in \Omega \setminus {\mathcal N}$, define an operator $\mathring \GMC_\infty ^{(\alpha )} (\omega )$ on $C(\Lambda )$ with domain $D$ by 
$$\mathring \GMC_\infty ^{(\alpha )} (\omega ) (f) := \lim _{N\rightarrow \infty} \langle f, \GMC _N^{(\alpha )} (\omega )\rangle, \quad f\in D.
$$
Then, for $\omega \in \Omega \setminus {\mathcal N}$, it is easy to see that $\mathring \GMC_\infty ^{(\alpha )} (\omega )$ can be extended to a linear operator on the space linearly spanned by $D$.
Moreover, since $D$ includes the constant function $1$,
\begin{equation}\label{eq:thmasconv05}
\sup _{N\in {\mathbb N}} \int _\Lambda \GMC _N^{(\alpha )}  (\omega ) \hspace{0.5mm} {\dd}x < \infty ,
\end{equation}
and hence, for $f\in D$
\[
\left| \mathring \GMC_\infty ^{(\alpha )} (\omega ) (f) \right| = \lim _{N\rightarrow \infty} \left| \langle f, \GMC _N^{(\alpha )} (\omega ) \rangle \right| \leq \| f\|_{C(\Lambda )} \sup _{N\in {\mathbb N}} \int _\Lambda \GMC _N^{(\alpha )}  (\omega ) \hspace{0.5mm} {\dd}x \lesssim \| f\|_{C(\Lambda )} .
\]
In view of these facts, for $\omega \in \Omega \setminus {\mathcal N}$, $\mathring \GMC_\infty ^{(\alpha )} (\omega )$ is extended to a bounded linear operator $\GMC_\infty ^{(\alpha )} (\omega )$ on $C(\Lambda )$.
By the denseness of $D$ in $C(\Lambda )$ and (\ref{eq:thmasconv05}), we have for $\omega \in \Omega \setminus {\mathcal N}$, $\GMC_\infty ^{(\alpha )} (\omega ) (f) = \lim _{N\rightarrow \infty} \langle f, \GMC _N^{(\alpha )} (\omega ) \rangle$ for $f\in C(\Lambda )$.
Nonnegativity of $\GMC_\infty ^{(\alpha )}$ follows from that of $\{ \GMC _N^{(\alpha )}\} _{N\in {\mathbb N}}$.
\end{proof}

\begin{proof}[{\bfseries Proof of Theorem \ref{thm:expwick almostsure}}]
Since convergence of the corresponding measures follows from Corollary \ref{cor:weak limit M2}, we prove convergence in the Besov space and independence of the limit in $\psi$.

First we show the convergence of $\GMC_N^{(\alpha)}$ in $B_{p,p}^{-\beta}$.
By Theorem \ref{thm:asconv}, for small $\delta>0$ and any $N\ge N_\delta$,
\begin{align*}
\mathbb{E}\left [ \big \|\Delta_j\GMC_{N+1}^{(\alpha)}-\Delta_j\GMC_N^{(\alpha)} \big \|_{L^1} \right ]
&=\int_{\Lambda} \mathbb{E}
\big [
\big | \big \langle \check\rho_j(x-\cdot),\GMC_{N+1}^{(\alpha)}-\GMC_N^{(\alpha)} \big \rangle \big |
\big]
\hspace{0.5mm} {\dd}x\\
&\lesssim C_\delta 2^{2j}2^{-c_\delta N},
\end{align*}
where $\check\rho_j=\sum_{k\in\mathbb{Z}^2}(\mathcal{F}^{-1}\rho_j)(\cdot+2\pi k)$.
This means
$$
\mathbb{E}\left[\big\|\GMC_{N+1}^{(\alpha)}-\GMC_N^{(\alpha)}\big\|_{B_{1,1}^{-\gamma}}\right]\lesssim C_\delta2^{-c_\delta N}
$$
for any $\gamma>2$.
On the other hand, by Corollary \ref{cor:unif int M}, for any parameters $p',\beta'$ satisfying \eqref{def p beta},
$$
\sup_{N\in\mathbb{N}}\mathbb{E}\left[\big\|\GMC_{N+1}^{(\alpha)}-\GMC_N^{(\alpha)}\big\|_{B_{p',p'}^{-\beta'}}^{p'}\right]<\infty.
$$
Fix parameters $p$ and $\beta$ satisfying \eqref{def p beta}. For any $\varepsilon\in(0,1)$, let $p_\varepsilon$ and $\beta_\varepsilon$ be parameters defined by $1/p=\varepsilon+(1-\varepsilon)/p_\varepsilon$ and $\beta=\varepsilon\gamma+(1-\varepsilon)\beta_\varepsilon$.
Since $p_\varepsilon\to p$ and $\beta_\varepsilon\to\beta$ as $\varepsilon\to0$, $p_\varepsilon$ and $\beta_\varepsilon$ satisfy \eqref{def p beta} for small $\varepsilon>0$. For such $\varepsilon$, by the interpolation (Proposition \ref{prop:interpolation besov}), we have
\begin{align*}
\mathbb{E}\left[\big\|\GMC_{N+1}^{(\alpha)}-\GMC_N^{(\alpha)}\big\|_{B_{p,p}^{-\beta}}^{p}\right]
&\le \mathbb{E}\left[\big\|\GMC_{N+1}^{(\alpha)}-\GMC_N^{(\alpha)}\big\|_{B_{1,1}^{-\gamma}}\right]^{\varepsilon p}
\mathbb{E}\left[\big\|\GMC_{N+1}^{(\alpha)}-\GMC_N^{(\alpha)}\big\|_{B_{p_\varepsilon,p_\varepsilon}^{-\beta_\varepsilon}}^{p_\varepsilon}\right]^{(1-\varepsilon)p/p_\varepsilon}\\
&\le C_\delta'2^{-c_\delta'N}
\end{align*}
for some constants $C_\delta',c_\delta'>0$ depending on $p,\beta$, and $\varepsilon$.
This implies the $L^p(\mathbb{P})$-convergence of $\{ \GMC_N^{(\alpha)} \}$ in $B_{p,p}^{-\beta}(\Lambda)$. 
Moreover, since
$$
\sum_{N=N_\delta}^{\infty} \mathbb{E}\big [ \big \| \GMC_{N+1}^{(\alpha)} -\GMC_N^{(\alpha)} \big \|_{B_{p,p}^{-\beta}} \big ]<\infty,
$$
by the Borel-Cantelli lemma we obtain the almost-sure convergence of $\{ \GMC_N^{(\alpha)} \}_{N\in\mathbb{N}}$.

Finally we show the uniqueness of the limit. Consider two multipliers $\psi$ and $\bar\psi$ satisfying Hypothesis \ref{hypo on psi} and define the limits $\GMC_\infty^{(\alpha)}$ and $\bar{\GMC}_\infty^{(\alpha)}$, respectively. By Corollary \ref{cor:weak limit M}, $\langle\GMC_\infty^{(\alpha)},{\bf e}_k\rangle=\langle\bar{\GMC}_\infty^{(\alpha)},{\bf e}_k\rangle$ for any $k\in\mathbb{Z}^2$ almost surely, so $\Delta_j\GMC_\infty^{(\alpha)}=\Delta_j\bar{\GMC}_\infty^{(\alpha)}$ for any $j\ge-1$ almost surely. Hence $\GMC_\infty^{(\alpha)}=\bar{\GMC}_\infty^{(\alpha)}$ in $B_{p,p}^{-\beta}$ almost surely.
\end{proof}

\section{Wick exponentials of Ornstein--Uhlenbeck processes}\label{sec:OU}

For $\phi\in\mathcal{D}'(\Lambda)$ and an $L^2(\Lambda)$-cylindrical Brownian motion $W$,
let $X=X(\phi)$ be the unique solution of the initial value problem
\begin{align}\label{eq:OU}
\left\{
\begin{aligned}
\partial_tX_t&=\frac12(\triangle-1)X_t+\dot{W}_t,\qquad t>0,\\
X_0&=\phi.
\end{aligned}
\right.
\end{align}
In this section, we consider the Wick exponential of the infinite-dimensional 
Ornstein--Uhlenbeck (OU in short) process $X$.
First we recall the basic estimate of $X$ in \cite{HKK19}.

\begin{prop}\label{prop:aprioriOU}
For $\varepsilon>0$, $\delta\in(0,1)$, $m\in\mathbb{N}$, and $T>0$, there exists a constant $C>0$ such that one has the a priori estimate
\begin{align}\label{ineq:aprioriOU}
\mathbb{E}{\Big [}  \|X(\phi)\|_{C([0,T];H^{-\varepsilon})\cap C^{\delta/2}([0,T];H^{-\varepsilon-\delta})}^m {\Big ]}
\le C(1+\|\phi\|_{H^{-\varepsilon}}^m).
\end{align}
Moreover, for any $\varepsilon>0$ and $\phi_1,\phi_2\in H^{-\varepsilon}$,
\begin{align}\label{stabilityOU}
\|X(\phi_1)-X(\phi_2)\|_{C([0,T];H^{-\varepsilon})}
\le\|\phi_1-\phi_2\|_{H^{-\varepsilon}}.
\end{align}
\end{prop}

\begin{proof}
See \cite[Proposition 2.1]{HKK19} for the proof of \eqref{ineq:aprioriOU}.
The estimate \eqref{stabilityOU} is obtained by writing down the mild form of \eqref{eq:OU}.
\end{proof}

It is known that the GFF measure $\mu_0$ is the invariant measure of the process $X$
(see e.g., \cite[Theorem 6.2.1]{DZ96}).
Therefore, the random variable
$$
\Omega\times \mathcal{D}'(\Lambda)\ni(\omega,\phi)\mapsto X_t(\phi)(\omega)\in\mathcal{D}'(\Lambda)
$$
is also a GFF under the probability measure $\mathbb{P}\otimes\mu_0$ for any $t>0$.
Thus the existence of the Wick exponential of $X$ is an immediate consequence of Theorem \ref{thm:expwick almostsure}.

\begin{thm}\label{thm:expwick OU}
Assume that $\psi$ satisfies Hypothesis {\rm{\ref{hypo on psi}}}.
Let $|\alpha|<\sqrt{8\pi}$ and choose parameters $p$ and $\beta$ as in \eqref{def p beta}.
Then the functions
\begin{align*}
\mathcal{X}_t^N(\phi)(x):=\exp\left(\alpha \big(P_NX_t(\phi)\big)(x)-\frac{\alpha^2}2C_N\right),\qquad
N\in\mathbb{N}
\end{align*}
are uniformly bounded in the space $L^p(\mathbb{P}\otimes\mu_0;L^p([0,T];B_{p,p}^{-\beta}))$ for any $T>0$.
Moreover, the function $\mathcal{X}^N$ converges as $N\to\infty$ in the space $L^p([0,T];B_{p,p}^{-\beta})$, $\mathbb{P}\otimes\mu_0$-almost surely and in $L^p(\mathbb{P}\otimes\mu_0)$.
The limits obtained by different $\psi$'s coincide with each other, $\mathbb{P}\otimes\mu_0$-almost surely.
\end{thm}

\begin{proof}
Using the invariance of $\mu_0$ with respect to $X_t$ and using Theorems \ref{thm:expwick almostsure} and \ref{thm:asconv}, we have the exponential decay
\begin{align*}
\mathbb{E}^{\mathbb{P}\otimes\mu_0}
\left[\left\|\mathcal{X}^{N+1}-\mathcal{X}^N\right\|_{L^p([0,T];B_{p,p}^{-\beta})}^p \right]
&=\int_{\mathcal{D}'(\Lambda)}\int_0^T
\mathbb{E}\left[\left\|\mathcal{X}_t^{N+1}(\phi)-\mathcal{X}_t^N(\phi)\right\|_{B_{p,p}^{-\beta}}^p \right]
{\dd}t\,\mu_0({\dd}\phi)\\
&=T\mathbb{E}\left[\left\|\exp_{N+1}^\diamond(\alpha\GFF)-\exp_N^\diamond(\alpha\GFF)\right\|_{B_{p,p}^{-\beta}}^p\right]\\
&\le TC 2^{-c N}
\end{align*}
for some positive constants $c$ and $C$,
where $\GFF$ is a GFF under the probability $\mathbb{P}$.
Then the assertion is obtained by a similar way to the proof of Theorem \ref{thm:expwick almostsure}.
\end{proof}

Denote by $\mathcal{X}^\infty:=\lim_{N\to\infty}\mathcal{X}^N$ the $\mathbb{P}\otimes\mu_0$-almost-sure limit.
The following result is an immediate consequence of the $\mathbb{P}\otimes\mu_0$-almost-sure convergence in Theorem \ref{thm:expwick OU}.

\begin{cor}\label{cor:expwick OU}
For $\mu_0$-almost every $\phi\in \mathcal{D}'(\Lambda)$, the random function $\mathcal{X}^N(\phi)$ converges to $\mathcal{X}^\infty(\phi)$ in the space $L^p([0,T];B_{p,p}^{-\beta})$ almost surely.
\end{cor}

In Section \ref{sec:stationary}, the following ``stability" result of $\mathcal{X}^{\infty}(\phi)$ with respect to $\phi$ makes an important role.

\begin{lem}\label{lem:contiexpOU}
Let $\varepsilon>0$, and let $\{\xi_N\}_{N\in\mathbb{N}\cup\{\infty\}}$ be $H^{-\varepsilon}$-valued random variables independent of $W$.
Assume that the law $\nu_N$ of $\xi_N$ is absolutely continuous with respect to $\mu_0$ for any $N\in\mathbb{N}\cup\{\infty\}$, and Radon-Nikodym derivatives 
$\left\{\frac{{\dd}\nu_N}{{\dd}\mu_0}\right\}_{N\in\mathbb{N}\cup\{\infty\}}$ are uniformly bounded.
If $\lim_{N\to\infty}\xi_N=\xi_\infty$ in $H^{-\varepsilon}$ almost surely, then for any $T>0$,
$$
\lim_{N\to\infty}\mathcal{X}^{\infty}(\xi_N)=
\mathcal{X}^{\infty}(\xi_\infty)
$$
in $L^p([0,T];B_{p,p}^{-\beta})$ in probability.
\end{lem}

\begin{proof}
The proof is very similar to \cite[Lemma 2.5]{HKK19} and done by a slight modification.
For any fixed $M\in\mathbb{N}$, by the estimate \eqref{stabilityOU},
\begin{align*}
\|P_MX(\xi_N)-P_MX(\xi_\infty)\|_{C([0,T];C(\Lambda))}
&\lesssim_M\|X(\xi_N)-X(\xi_\infty)\|_{C([0,T];H^{-\varepsilon})}\\
&\lesssim\|\xi_N-\xi_\infty\|_{H^{-\varepsilon}}
\xrightarrow{N\to\infty}0,
\end{align*}
almost surely. In the first inequality, we use the fact that $P_M$ sends $H^{-\varepsilon}$ to $C(\Lambda)$, as mentioned after Hypothesis \ref{hypo on psi}.
Hence for any fixed $M\in\mathbb{N}$,
\begin{align*}
\lim_{N\to\infty}\|\mathcal{X}^M(\xi_N)-\mathcal{X}^M(\xi_\infty)\|_{C([0,T];C(\Lambda))}=0
\end{align*}
almost surely, from the definition of the Wick exponential $\mathcal{X}^M$.
On the other hand, since Radon-Nikodym derivatives 
$\frac{{\dd}\nu_N}{{\dd}\mu_0}$
are uniformly bounded, by using Theorem \ref{thm:expwick OU},
\begin{align*}
&\sup_{N\in\mathbb{N}\cup\{\infty\}}\mathbb{E}
\Big [
\|\mathcal{X}^M(\xi_N)-\mathcal{X}^\infty(\xi_N)\|_{L^p([0,T];B_{p,p}^{-\beta})}^p
\Big ]
\\
&\lesssim \sup_{N\in\mathbb{N}\cup\{\infty\}}
\mathbb{E} \left[ 
\int_{\mathcal{D}'(\Lambda)}\|\mathcal{X}^M(\phi)-\mathcal{X}^\infty(\phi)\|_{L^p([0,T];B_{p,p}^{-\beta})}^p \nu_N({\dd}\phi) \right]
\xrightarrow{M\to\infty}0.
\end{align*}
Hence, by using the inequality $(a+b)\wedge1\le a+(b\wedge1)$ for $a,b\ge0$, we have
\begin{align*}
&\mathbb{E}\left[\|\mathcal{X}^\infty(\xi_N)-
\mathcal{X}^\infty(\xi_\infty)\|_{L^p([0,T];B_{p,p}^{-\beta})}\wedge1\right]\\
&\le 2\sup_{N\in\mathbb{N}\cup\{\infty\}}\mathbb{E}
\Big [
\|\mathcal{X}^M(\xi_N)-\mathcal{X}^\infty(\xi_N)\|_{L^p([0,T];B_{p,p}^{-\beta})}
\Big ]
\\
&\quad+\mathbb{E}\left[\|\mathcal{X}^M(\xi_N)-
\mathcal{X}^M(\xi_\infty)\|_{L^p([0,T];B_{p,p}^{-\beta})}\wedge1\right].
\end{align*}
Letting $N\to\infty$ first and then $M\to\infty$, we have
$$
\lim_{N\to\infty}
\mathbb{E}\left[\|\mathcal{X}^\infty(\xi_N)-
\mathcal{X}^\infty(\xi_\infty)\|_{L^p([0,T];B_{p,p}^{-\beta})}\wedge1\right]=0.
$$
Thus we have the assertion.
\end{proof}


\section{Global well-posedness of the strong solution}\label{sec:wellposed}

In this section, we consider the approximating equation \eqref{expsqe1}:
\begin{equation*}
\left\{
\begin{aligned}
\partial _t \Phi_t^N &= \frac 12 (\triangle-1) \Phi_t^N 
- \frac\alpha2 \exp\left(\alpha \Phi_t^N-\frac{\alpha^2}2 C_N\right) + P_N\dot W_t ,\qquad t>0,\\
\Phi_0^N&=P_N\phi,
\end{aligned}
\right.
\end{equation*}
and prove Theorem \ref{mainthm1}. 
The proof goes in a similar way to \cite[Section 3]{HKK19} with a slight modification.
Similarly to the previous paper, we use the Da Prato--Debussche trick, 
that is, we decompose the solution of \eqref{expsqe1} by $\Phi^N=X^N+Y^N$, where $X^N$ and $Y^N$ solve
\begin{align}
\label{eq:X}
&\left\{
\begin{aligned}
\partial _t X_t^N &= \frac 12 (\triangle-1) X_t^N + P_N\dot W_t ,\qquad t>0,\\
X_0^N&=P_N\phi,
\end{aligned}
\right.\\[5pt]
\label{eq:Y}
&\left\{
\begin{aligned}
\partial _t Y_t^N &= \frac 12 (\triangle-1) Y_t^N
- \frac\alpha2 \exp(\alpha Y_t^N)\exp\left(\alpha X_t^N -\frac{\alpha^2}2C_N\right) ,\qquad t>0,\\
Y_0^N&=0.
\end{aligned}
\right.
\end{align}
Note that $X^N=P_NX(\phi)$, where $X(\phi)$ is the solution of \eqref{eq:OU} with the initial value $\phi$. Hence the renormalized exponential of $X^N$ in the latter equation \eqref{eq:Y} is equal to
$$
\exp\left(\alpha X_t^N -\frac{\alpha^2}2C_N\right)=\mathcal{X}_t^N(\phi).
$$
Since $\mathcal{X}^N$ converges to $\mathcal{X}^\infty$ in $L^p([0,T];B_{p,p}^{-\beta})$ as stated in Corollary \ref{cor:expwick OU}, we consider the solution map of the \emph{deterministic} equation
\begin{align*}
\partial_t\Upsilon_t&=\frac12(\triangle-1)\Upsilon_t-\frac\alpha2 e^{\alpha \Upsilon_t}\mathcal{X}_t,\qquad t\in[0,T]
\end{align*}
for any generic nonnegative $\mathcal{X}\in L^p([0,T];B_{p,p}^{-\beta})$.
\subsection{Products of continuous functions and nonnegative distributions}\label{sec:nnegdist}
Since any nonnegative distribution is regarded as a nonnegative Borel measure by Theorem \ref{thm:LL}, the product of a function $f\in C(\Lambda)$ and a nonnegative distribution $\xi\in\mathcal{D}'(\Lambda)$ is well-defined as a Borel measure.
\begin{defi}
For any $f\in C(\Lambda)$ and any nonnegative $\xi\in\mathcal{D}'(\Lambda)$, we define the signed Borel measure
$$
\mathcal{M}(f,\xi)({\dd}x):=f(x)\mu_\xi({\dd}x),
$$
where $\mu _\xi ({\dd}x)$ is the Borel measure associated with $\xi$, as in Theorem {\rm{\ref{thm:LL}}}.
\end{defi}
We recall some properties of the product map $\mathcal{M}$ from \cite[Section 3.1]{HKK19}. 
Recall that $B^{s,+}_{p,q}(\Lambda)$ denotes the set of nonnegative elements in $B^{s}_{p,q}(\Lambda)$.

\begin{thm}[{\cite[Theorems 3.4 and 3.5]{HKK19}}]\label{thm:stableM}
Let $s>0$ and $p,q\in[1,\infty]$. The map
$$
\mathcal{M}:C(\Lambda)\times B_{p,q}^{-s,+}\to B_{p,q}^{-s}
$$
is continuous, and bounded in the sense that
$$
\|\mathcal{M}(f,\xi)\|_{B_{p,q}^{-s}}\lesssim \|f\|_{C(\Lambda)}\|\xi\|_{B_{p,q}^{-s}}
$$
for any $f\in C(\Lambda)$ and $\xi\in B_{p,q}^{-s,+}$.
\end{thm}


\begin{thm}[{\cite[Theorem 3.6]{HKK19}}]\label{replaced keythm}
Let $s>0$, $p,q\in[1,\infty]$, and $r\in(1,\infty]$.
For any space-time functions $(Y,\mathcal{X}) \in L^1([0,T];C(\Lambda))\times L^r([0,T];B_{p,q}^{-s,+})$ and any function $f\in C_b^1(\mathbb{R})$, consider the time-dependent distribution
$$
\mathcal{M}(f(Y),\mathcal{X})(t)
:=\mathcal{M}(f(Y_t),\mathcal{X}_t).
$$
Then the correspondence $(Y,\mathcal{X})\mapsto\mathcal{M}(f(Y),\mathcal{X})$
is well-defined as a map
$$
L^1([0,T];C(\Lambda))\times L^r([0,T];B_{p,q}^{-s,+}) \to L^{r'}([0,T];B_{p,q}^{-s}).
$$
for any $r'\in[1,r]$. Moreover, if $r'<r$, this map is continuous.
\end{thm}

\subsection{Global well-posedness of $\Upsilon$}

In this part, we can consider more general parameters
\begin{align}\label{weaker p beta}
p\in(1,\infty),\qquad \beta\in\left(0,\frac2p(p-1)\right),
\end{align}
than those in \eqref{def p beta}. We fix such parameters $p,\beta$ and the time interval $[0,T]$.
We consider the initial value problem
\begin{align}
\label{eq:Ups}
\left\{
\begin{aligned}
\partial_t\Upsilon_t&=\frac12(\triangle-1)\Upsilon_t-\frac\alpha2 \mathcal{M}(e^{\alpha \Upsilon_t},\mathcal{X}_t),\qquad t\in (0,T],\\
\Upsilon_0&=\upsilon,
\end{aligned}
\right.
\end{align}
for any given $\mathcal{X}\in L^p([0,T];B_{p,p}^{-\beta,+})$ and $\upsilon\in B_{p,p}^{2-\beta}$.
To solve the equation \eqref{eq:Ups}, 
we introduce the space
\begin{align*}
\mathscr{Y}_T&=\left\{\Upsilon\in L^p([0,T];C(\Lambda))\cap C([0,T];L^p)\ ;
e^{\alpha\Upsilon}\in L^\infty([0,T];C(\Lambda))
\right\}
\end{align*}
as a solution space.
The purpose of this section is showing the following theorem:

\begin{thm}\label{thm:solmapUps}
Assume that $p$ and $\beta$ satisfy \eqref{weaker p beta}.
Let $\mathcal{X}\in L^p([0,T];B_{p,p}^{-\beta,+})$ and $\upsilon\in B_{p,p}^{2-\beta}$.
Then there exists the unique mild solution $\Upsilon\in\mathscr{Y}_T$ of \eqref{eq:Ups}, that is, $\Upsilon$ satisfies the equation
\begin{align}\label{eq:mildUps}
\Upsilon_t=e^{t(\triangle-1)/2}\upsilon-\frac\alpha2\int_0^te^{(t-s)(\triangle-1)/2}\mathcal{M}(e^{\alpha\Upsilon_s},\mathcal{X}_s) 
\hspace{0.5mm}
{\dd}s
\end{align}
for any $t\in(0,T]$.
Moreover, this solution belongs to the space
$$
L^p([0,T];B_{p,p}^{2/p+\delta})\cap C([0,T];B_{p,p}^{\delta})
$$
for any $\delta\in(0,\frac{2}p(p-1)-\beta)$, and the mapping
$$
\mathcal{S}:B_{p,p}^{2-\beta}\times L^p([0,T];B_{p,p}^{-\beta,+})\ni(\upsilon,\mathcal{X})
\mapsto \Upsilon\in L^p([0,T];B_{p,p}^{2/p+\delta})\cap C([0,T];B_{p,p}^{\delta})
$$
is continuous.
\end{thm}

Recall the following Schauder estimates for the heat semigroup.

\begin{prop}[{\cite[Lemma 2.2]{KOT03} and \cite[Proposition 6]{MW17}}]\label{prop:heatsemigr}
Let $s\in\mathbb{R}$ and $p,q\in[1,\infty]$.
\begin{enumerate}
\item \label{prop:heatsemigr1} For every $\delta\ge0$, $\|e^{t(\triangle-1)/2}u\|_{B_{p,q}^{s+2\delta}}\lesssim t^{-\delta}\|u\|_{B_{p,q}^s}$ uniformly over $t>0$.
\item \label{prop:heatsemigr2} For every $\delta\in[0,1]$, $\|(e^{t(\triangle-1)/2}-1)u\|_{B_{p,q}^{s-2\delta}}\lesssim t^\delta\|u\|_{B_{p,q}^s}$ uniformly over $t>0$.
\end{enumerate}
\end{prop}

\begin{rem}
We remark that, if $\triangle-1$ is replaced by $\triangle$, then
$$
\|e^{t\triangle/2}u\|_{B_{p,q}^{s+2\delta}}\lesssim (1+t^{-\delta})\|u\|_{B_{p,q}^s}
$$
is the right $t$-uniform estimate ({\rm{\cite[Lemma 2.2]{KOT03}}}). 
The constant $1$ comes from the bound of $e^{t\triangle/2}\Delta_{-1}u$.
In the above proposition, we can omit this constant by using the factor $e^{-t}$.
\end{rem}

\begin{prop}[{\cite[Proposition A.3]{HKK19}}]\label{prop:Schauder}
Let $\theta\in\mathbb{R}$, $p,q\in[1,\infty]$, and $r\in(1,\infty]$.
Let $U$ be an element of $L^r([0,T];B_{p,q}^\theta)$, and let $u$ be the mild solution of the equation
\begin{align*}
\partial_tu&=\frac12(\triangle-1)u+U,\qquad t>0,
\end{align*}
with initial value $u_0\in B_{p,q}^{\theta+2}$.
Then for any $\varepsilon>0$ and $\delta\in(0,2/r')$, one has
$$
\|u\|_{L^r([0,T];B_{p,q}^{\theta+2-\varepsilon})\cap C([0,T];B_{p,q}^{\theta+2/r'-\varepsilon})\cap C^{\delta/2}([0,T];B_{p,q}^{\theta+2/r'-\varepsilon-\delta})}
\lesssim
\|u_0\|_{B_{p,q}^{\theta+2}}+\|U\|_{L^r([0,T];B_{p,q}^{\theta})},
$$
where $r'\in[1,\infty)$ is such that $1/r+1/r'=1$.
\end{prop}

We first show the uniqueness of the solution, by following \cite[Lemma 3.8]{HKK19}.
Since the function $x\mapsto|x|^p$ is not twice differentiable if $p<2$, we need to modify the previous argument.

\begin{lem}\label{lem:unique}
For any $\mathcal{X}\in L^p([0,T];B_{p,p}^{-\beta,+})$ and $\upsilon\in B_{p,p}^{2-\beta}$, there is at most one mild solution $\Upsilon\in\mathscr{Y}_T$ of the equation \eqref{eq:Ups}.
\end{lem}

\begin{proof}
Let $\Upsilon,\Upsilon'\in\mathscr{Y}_T$ be two solutions of \eqref{eq:Ups} with the same $\mathcal{X}$ and $\upsilon$. Then $Z=\Upsilon-\Upsilon'$ solves the equation
\begin{align*}
\left\{\partial_t-\frac12(\triangle-1)\right\}Z_t
&=-\frac\alpha2\mathcal{M}(e^{\alpha \Upsilon_t}-e^{\alpha \Upsilon_t'},\mathcal{X}_t)
=:D_t,
\end{align*}
where $D\in L^p([0,T];B_{p,p}^{-\beta})$, because of definition of $\mathscr{Y}_T$ and Theorem \ref{thm:stableM}.
Let $\varepsilon>0$ and define $Z^\varepsilon=e^{\varepsilon\triangle}Z$. Then $Z^\varepsilon$ solves the equation
\begin{align*}
\left\{\partial_t-\frac12(\triangle-1)\right\}Z^\varepsilon
&=e^{\varepsilon\triangle}D.
\end{align*}
By the regularizing effect of the heat semigroup (Proposition \ref{prop:heatsemigr}), $e^{\varepsilon\triangle}D$ belongs to $L^p([0,T];C^\infty(\Lambda))$. 
Then by the Schauder estimate (Proposition \ref{prop:Schauder}), we have that $Z^\varepsilon$ belongs to $C([0,T];C^\infty(\Lambda))$.
Hence for any $f\in C^2(\mathbb{R})$, we have
\begin{align*}
&\int_\Lambda f(Z_t^\varepsilon(x)) 
\, {\dd}x
=f(0)|\Lambda|+\int_0^t\int_\Lambda f'(Z_s^\varepsilon(x))\partial_sZ_s^\varepsilon(x)
\, {\dd}x \hspace{0.2mm} {\dd}s
\\
&=f(0)|\Lambda|+\frac12\int_0^t\int_\Lambda f'(Z_s^\varepsilon(x))(\triangle-1)Z_s^\varepsilon(x)
\, {\dd}x \hspace{0.2mm} {\dd}s
+\int_0^t\int_\Lambda f'(Z_s^\varepsilon(x))e^{\varepsilon\triangle}D_s(x)
\, {\dd}x \hspace{0.2mm} {\dd}s
\\
&=f(0)|\Lambda|-\frac12\int_0^t\int_\Lambda f''(Z_s^\varepsilon(x))|\nabla Z_s^\varepsilon(x)|^2
\, {\dd}x \hspace{0.2mm} {\dd}s
-\frac12\int_0^t\int_\Lambda f'(Z_s^\varepsilon(x))Z_s^\varepsilon(x)
\, {\dd}x \hspace{0.2mm} {\dd}s
\\
&\quad+\int_0^t\int_\Lambda f'(Z_s^\varepsilon(x))e^{\varepsilon\triangle}D_s(x)
\, {\dd}x \hspace{0.2mm} {\dd}s,
\end{align*}
where the first equality is justified as a Riemann--Stieltjes integral, because 
$$\partial_sZ^\varepsilon=\frac12(\triangle-1)Z^\varepsilon+e^{\varepsilon\triangle}D\in L^p([0,T];C^\infty(\Lambda)).$$
For $\lambda >0$, let
$$
f_\lambda(x)=(\lambda^2+x^2)^{p/2},\qquad x\in\mathbb{R},
$$
and for $R>0$, let $\varphi_R\in C^\infty(\mathbb{R})$ be a nonnegative even smooth function such that
$$
\varphi_R(x)
=
\begin{cases}
1,&|x|<R,\\
0,&|x|>R+1.
\end{cases}
$$
Then we define $f_{\lambda,R} \in C^2(\mathbb{R})$ by the function determined by
$$
\left\{
\begin{aligned}
f_{\lambda,R}''(x)&=f_\lambda''(x)\varphi_R(x),\qquad x\in\mathbb{R},\\
f_{\lambda,R}'(0)&=0,\\
f_{\lambda,R}(0)&=\lambda^p.
\end{aligned}
\right.
$$
Since we easily have the properties
\begin{itemize}
\item $f_{\lambda,R}''\ge0$,
\item $f_{\lambda,R}'$ is bounded and $xf_{\lambda,R}'(x)\ge0$,
\item $f_{\lambda,R}(x)\uparrow f_\lambda(x)$ as $R\to\infty$,
\end{itemize}
we have the inequality
\begin{align*}
\int_\Lambda f_{\lambda,R}(Z_t^\varepsilon(x))
\hspace{0.5mm} {\dd}x
\le\lambda^p|\Lambda|
+\int_0^t\int_\Lambda f_{\lambda,R}'(Z_s^\varepsilon(x))e^{\varepsilon\triangle}D_s(x)
\hspace{0.5mm} {\dd}x \hspace{0.2mm} {\dd}s.
\end{align*}
Once we let $\varepsilon\to0$,
$e^{\varepsilon\triangle}D\to D$ in $L^p([0,T];B_{p,p}^{-\beta-\kappa})$ for any $\kappa>0$
by Proposition \ref{prop:heatsemigr}, 
and hence $Z^\varepsilon\to Z$ in $L^p([0,T];B_{p,p}^{2-\beta-2\kappa})$ by Proposition \ref{prop:Schauder}.
Since $B_{p,p}^{2-\beta-2\kappa}\subset C(\Lambda)$ for small $\kappa>0$,
by using Theorem \ref{replaced keythm} we have
\begin{align}\label{ineq:energy ineq}
\int_\Lambda f_{\lambda,R}(Z_t(x))
\hspace{0.5mm} {\dd}x
\le\lambda^p|\Lambda|
+\int_0^t\int_\Lambda f_{\lambda,R}'(Z_s(x))D_s(x)
\hspace{0.5mm} {\dd}x \hspace{0.2mm} {\dd}s
\end{align}
for almost every $t$.
Here, we used the boundedness of $f_{\lambda ,R}'$
and that $e^{\varepsilon \triangle}D$ is a difference of two nonnegative functions,
for the convergence of the second term of the right-hand side.
We can deduce the term as
\begin{align*}
\int_\Lambda f_{\lambda,R}'(Z_s(x))D_s(x)
\hspace{0.5mm} {\dd}x
&=-\frac\alpha2\int_\Lambda(e^{\alpha\Upsilon_s(x)}-e^{\alpha\Upsilon_s'(x)})
f_{\lambda,R}'(Z_s(x)) \mathcal{X}_s(x) 
\hspace{0.5mm} {\dd}x
\\
&=-\frac{\alpha^2}2\int_\Lambda e^{A(\alpha\Upsilon_s(x),\alpha\Upsilon_s(x))}
Z_s(x)f_{\lambda,R}'(Z_s(x)) \mathcal{X}_s(x) 
\hspace{0.5mm} {\dd}x
\le0,
\end{align*}
where $A(x,y)$ is a continuous function on $\mathbb{R}^2$ defined by
$$
A(x,y)=
\begin{cases}
\log\frac{e^x-e^y}{x-y},&x\neq y,\\
x&x=y.
\end{cases}
$$
Hence letting $R\to\infty$ in \eqref{ineq:energy ineq}, we have
\begin{align*}
\int_\Lambda f_\lambda(Z_t(x))
\hspace{0.5mm} {\dd}x
\le \lambda^p|\Lambda|.
\end{align*}
for almost every $t$.
Letting $\lambda\to0$, we have $\|Z_t\|_{L^p(\Lambda)}=0$, which implies $\Upsilon=\Upsilon'$ for almost every $(t,x)$,
thus $\Upsilon=\Upsilon'$ in $\mathscr{Y}_T$.
\end{proof}

Next we show the existence of the solution, by following \cite[Lemma 3.10]{HKK19}.
Since the only difference is that we use Besov spaces instead of Sobolev spaces, we omit some details in this part.
The following embedding theorem is frequently used below.

\begin{lem}[{\cite[Corollary~5]{JSimon}}]\label{Simon}
Let $\mathcal{A}\subset \mathcal{B}\subset \mathcal{C}$ be Banach spaces such that the inclusion $\mathcal{A}\hookrightarrow \mathcal{B}$ is compact. Then for any $p\in[1,\infty]$ and $s>0$, the embeddings
\begin{align*}
L^p([0,T];\mathcal{A})\cap C^s([0,T];\mathcal{C}) &\hookrightarrow L^p([0,T];\mathcal{B}),\\
C([0,T];\mathcal{A})\cap C^s([0,T];\mathcal{C}) &\hookrightarrow C([0,T];\mathcal{B})
\end{align*}
are compact.
\end{lem}

\begin{lem}\label{lem:existence}
For any $\mathcal{X}\in L^p([0,T];B_{p,p}^{-\beta,+})$ and $\upsilon\in B_{p,p}^{2-\beta}$, there is at least one mild solution $\Upsilon\in\mathscr{Y}_T$. Moreover, for any $\delta\in(0,\frac2p(p-1)-\beta)$, there exists a constant $C>0$ independent of $\mathcal{X}$ and $\upsilon$ such that one has the a priori estimate
\begin{align}\label{eq:aprioriUps}
\begin{aligned}
&\|\Upsilon\|_{L^p([0,T];B_{p,p}^{2/p+\delta})
\cap C([0,T];B_{p,p}^{\delta})
\cap C^{\delta/2}([0,T];L^p)}\\
&\qquad\le C\left\{\|\upsilon\|_{B_{p,p}^{2-\beta}}+e^{|\alpha|\|\upsilon\|_{C(\Lambda)}}
\|\mathcal{X}\|_{L^p([0,T];B_{p,p}^{-\beta})}\right\}.
\end{aligned}
\end{align}
\end{lem}

\begin{proof}
As discussed in \cite[Lemma 3.10]{HKK19}, for any $\mathcal{X}\in L^p([0,T];B_{p,p}^{-\beta,+})$, there exists a family $\{\mathcal{X}^N\}_{N\in\mathbb{N}}$ of nonnegative continuous functions on $[0,T]\times\Lambda$ such that $\mathcal{X}^N\to\mathcal{X}$ in $L^p([0,T];B_{p,p}^{-\beta,+})$ as $N\to\infty$.
For such $\mathcal{X}^N$, we consider the classical global solutions of the approximating equations
$$
\left\{
\begin{aligned}
\partial_t\Upsilon_t^N&=\frac12(\triangle-1)\Upsilon_t^N-\frac\alpha2 e^{\alpha \Upsilon_t^N}\mathcal{X}_t^N,\\
\Upsilon_0^N&=\upsilon.
\end{aligned}
\right.
$$
Note that $\alpha \Upsilon_t^N\le|\alpha| \|\upsilon\|_{C(\Lambda)}$ follows from the comparison principle.
By applying the Schauder estimate (Proposition \ref{prop:Schauder}) and Theorem \ref{thm:stableM}, for any $\delta'\in(\delta,\frac2p (p-1)-\beta)$ we have
\begin{align*}
\|\Upsilon^N\|_{L^p([0,T];B_{p,p}^{2/p+\delta'})\cap C^{\delta'/2}([0,T];L^p)}
&\lesssim \|\upsilon\|_{B_{p,p}^{2-\beta}}+\|\mathcal{M}(e^{\alpha \Upsilon^N},\mathcal{X}^N)\|_{L^p([0,T];B_{p,p}^{-\beta})}\\
&\lesssim \|\upsilon\|_{B_{p,p}^{2-\beta}}+\|e^{\alpha \Upsilon^N}\|_{L^\infty([0,T];C(\Lambda))}\|\mathcal{X}^N\|_{L^p([0,T];B_{p,p}^{-\beta})}\\
&\lesssim \|\upsilon\|_{B_{p,p}^{2-\beta}}+e^{|\alpha|\|\upsilon\|_{C(\Lambda)}}\|\mathcal{X}^N\|_{L^p([0,T];B_{p,p}^{-\beta})}.
\end{align*}
By Lemma \ref{Simon}, the embeddings
\begin{align*}
L^p([0,T];B_{p,p}^{2/p+\delta'})\cap C^{\delta'/2}([0,T];L^p)&\hookrightarrow
L^p([0,T];B_{p,p}^{2/p+\delta}),\\
C([0,T];B_{p,p}^{\delta'})\cap C^{\delta'/2}([0,T];L^p)&\hookrightarrow
C([0,T];B_{p,p}^{\delta})
\end{align*}
are compact. Here, recall that the embedding $B_{p,p}^s\hookrightarrow B_{p,p}^{s'}$ is compact for any $s'<s$ (see \cite[Corollary 2.96]{BCD11}). Hence there exists a subsequence $\{N_k\}$ such that
$$
\Upsilon^{N_k}\to \Upsilon\qquad\text{in $L^p([0,T];B_{p,p}^{2/p+\delta})\cap C([0,T];B_{p,p}^{\delta})$}.
$$
This yields the bound \eqref{eq:aprioriUps} for $\Upsilon$,
thus in particular $\Upsilon\in\mathscr{Y}_T$.

We have that $\Upsilon$ solves the mild equation \eqref{eq:mildUps} by a similar argument to \cite[Lemma 3.10]{HKK19}. Since $\alpha\Upsilon^N$ is uniformly bounded from above, we can apply Theorem \ref{replaced keythm} to the function $f\in C_b^1(\mathbb{R})$ such that $f(x)=e^{x}$ on some half line $x\in (-\infty,a]$ and obtain
$$
\mathcal{M}(e^{\alpha \Upsilon^{N_k}},\mathcal{X}^{N_k})\to
\mathcal{M}(e^{\alpha \Upsilon},\mathcal{X})
\qquad\text{in $L^q([0,T];B_{p,p}^{-\beta})$}
$$
for any $q<p$.
Then letting $N_k\to\infty$ on both sides of \eqref{eq:mildUps} and applying the Schauder estimate (Proposition \ref{prop:Schauder}), 
we have that $(\Upsilon,\mathcal{X})$ solves the same equation in the space $C([0,T];B_{p,p}^\delta)$.
\end{proof}

By Lemmas \ref{lem:unique} and \ref{lem:existence}, the solution map $\mathcal{S}:(\upsilon,\mathcal{X})\mapsto\Upsilon$ is well-defined. 
The continuity of the map
$$
\mathcal{S}:B_{p,p}^{2-\beta}\times L^p([0,T];B_{p,p}^{-\beta,+})\ni(\upsilon,\mathcal{X})
\mapsto \Upsilon\in L^p([0,T];B_{p,p}^{2/p+\delta})\cap C([0,T];B_{p,p}^{\delta})
$$
follows from a similar compactness argument as above, and from uniqueness of the solution.
Indeed, by the a priori estimate \eqref{eq:aprioriUps}, any convergent sequence of $B_{p,p}^{2-\beta}\times L^p([0,T];B_{p,p}^{-\beta,+})$ is sent to a bounded sequence of $L^p([0,T];B_{p,p}^{2/p+\delta''})\cap C([0,T];B_{p,p}^{\delta''})$ by the map $\mathcal{S}$, for any $\delta''\in(\delta,\frac2p(p-1)-\beta)$.
This sequence is precompact in $L^p([0,T];B_{p,p}^{2/p+\delta})\cap C([0,T];B_{p,p}^{\delta})$.
By the same argument as before, we see that any accumulation point solves the equation \eqref{eq:mildUps}, which is unique. Hence this precompact sequence converges.
This completes the proof of Theorem \ref{thm:solmapUps}.

\subsection{Proof of Theorem \ref{mainthm1}}

From Theorem \ref{thm:solmapUps}, the first main result of this paper (Theorem \ref{mainthm1}) immediately follows.

\begin{proof}[{\bfseries Proof of Theorem \ref{mainthm1}}]
By the Da Prato--Debussche decomposition \eqref{eq:X}-\eqref{eq:Y}, the solution $\Phi^N(\phi)$ of the equation \eqref{expsqe1} satisfies
$$
\Phi^N(\phi)=P_NX(\phi)+\mathcal{S}(0,\mathcal{X}^N(\phi)).
$$
For $\mu_0$-almost every $\phi$, $X(\phi)\in C([0,T];H^{-\varepsilon})$ for any $\varepsilon>0$, in view of Proposition \ref{prop:aprioriOU}.
Hence the first term $P_NX(\phi)$ of the right-hand side converges almost surely to $X(\phi)$ in $C([0,T];H^{-\varepsilon})$ for any $\varepsilon>0$, under Hypothesis \ref{hypo on psi}. 
The second term $\mathcal{S}(0,\mathcal{X}^N(\phi))$ converges almost surely to $\mathcal{S}(0,\mathcal{X}^\infty(\phi))$ in $C([0,T];B_{p,p}^{\delta})$ (see Theorems \ref{thm:expwick OU} and \ref{thm:solmapUps}). 
Hence $\Phi^N(\phi)$ converges to
$$
\Phi(\phi)=X(\phi)+\mathcal{S}(0,\mathcal{X}^\infty(\phi))
$$
in the space $C([0,T];B_{p,p}^{-\varepsilon})$ for any $\varepsilon>0$ almost surely, for $\mu_0$-almost every $\phi$.
\end{proof}

\section{Stationary solution}\label{sec:stationary}

In this section, we prove Theorem \ref{mainthm2} and Corollary \ref{maincor} by assuming that $\psi$ satisfies Hypotheses \ref{hypo on psi} and \ref{hypo on P}.
Recall that $\widetilde{\Phi}^N=\widetilde{\Phi}^N(\phi)$ is a unique solution of the SPDE \eqref{expsqe2}:
\begin{equation*}
\left\{
\begin{aligned}
\partial_t\widetilde{\Phi}_t^N&=\frac12(\triangle-1)\widetilde{\Phi}_t^N
-\frac\alpha2 P_N\exp\left(\alpha P_N\widetilde{\Phi}_t^N-\frac{\alpha^2}2C_N\right)+\dot{W}_t,\qquad t>0,\\
\widetilde{\Phi}_0^N&=\phi\in\mathcal{D}'(\Lambda),
\end{aligned}
\right.
\end{equation*}
and $\Phi=\Phi(\phi)$ is the strong solution obtained by Theorem \ref{mainthm1}.
Since the nonlinear term of \eqref{expsqe2} is given by the log-derivative of the approximating measure $\mu_N^{(\alpha)}$ defined by \eqref{expNmeas}, it is easy to show that $\mu_N^{(\alpha)}$ is an invariant measure of the process $\widetilde{\Phi}^N$ (see \cite[Section 4]{HKK19} for details).
Therefore, if $\xi_N$ is a random variable with the law $\mu_N^{(\alpha)}$ and independent of $W$, then 
$$
\widetilde{\Phi}^{N,{\rm stat}}:=\widetilde{\Phi}^N(\xi_N)
$$
is a stationary process.
Let $\xi$ be a $\mathcal{D}'(\Lambda)$-valued random variable with the law $\mu^{(\alpha)}$ and independent of $W$, and define
$$
{\Phi}^{\rm stat}:=\Phi(\xi).
$$
The proof of Theorem \ref{mainthm2} consists of showing the following two facts:
\begin{enumerate}
\item $\{\widetilde{\Phi}^{N,{\rm stat}}\}_{N\in\mathbb{N}}$ is tight in the space $C([0,T];H^{-\varepsilon})$ for any $\varepsilon>0$.
\item $\widetilde{\Phi}^{N,{\rm stat}}$ converges in law to ${\Phi}^{\rm stat}$ in the space $C([0,T];B_{p,p}^{-\varepsilon})$ for any $\varepsilon>0$.
\end{enumerate}
Once they are proved, Theorem \ref{mainthm2} is obtained as follows:
(i) implies that there exists a subsequence $\{\widetilde{\Phi}^{N_k,{\rm stat}}\}_{k\in\mathbb{N}}$ converging in law to a stochastic process $\Psi$ in the space $C([0,T];H^{-\varepsilon})$. 
On the other hand, $\{\widetilde{\Phi}^{N_k,{\rm stat}}\}_{k\in\mathbb{N}}$ converges to ${\Phi}^{\rm stat}$ in $C([0,T];B_{p,p}^{-\varepsilon})$ by (ii).
Since $C([0,T];H^{-\varepsilon})$ is continuously embedded into $C([0,T];B_{p,p}^{-\varepsilon'})$ for any $\varepsilon'>\varepsilon$, the laws of $\Psi$ and ${\Phi}^{\rm stat}$ in $C([0,T];B_{p,p}^{-\varepsilon'})$ coincide. 
Since $H^{-\varepsilon}$ and $B_{p,p}^{-\varepsilon'}$ are separable, by Lusin-Souslin's theorem (cf. \cite[Theorem 15.1]{Kec95}), $C([0,T];H^{-\varepsilon})$ is a measurable subset of $C([0,T];B_{p,p}^{-\varepsilon'})$.
Therefore,
\begin{align*}
\mathbb{P}\left({\Phi}^{\rm stat}\in C([0,T];H^{-\varepsilon})\right)
=\mathbb{P}\left(\Psi\in C([0,T];H^{-\varepsilon})\right)=1,
\end{align*}
and hence $\Psi \overset{d}{=} {\Phi}^{\rm stat}$ in $C([0,T];H^{-\varepsilon})$.
This implies that the accumulation point of the laws of $\{\widetilde{\Phi}^{N,{\rm stat}}\}_{N\in\mathbb{N}}$ in $C([0,T];H^{-\varepsilon})$ is unique, therefore $\widetilde{\Phi}^{N,{\rm stat}}$ converges in law to ${\Phi}^{\rm stat}$ in the space $C([0,T];H^{-\varepsilon})$.
For any bounded continuous function $f$ on $H^{-\varepsilon}$, by Corollary \ref{cor:expmeas},
\begin{align*}
\mathbb{E} \big[f({\Phi}_t^{\rm stat}) \big]
=\lim_{k\to\infty}\mathbb{E} \big [f(\widetilde{\Phi}_t^{N_k, {\rm stat}}) \big]
=\lim_{k\to\infty}\int _{H^{-\varepsilon}} f(\phi) \mu_{N_k}^{(\alpha)}({\dd}\phi)
=\int  _{H^{-\varepsilon}} f(\phi) \mu^{(\alpha)}({\dd}\phi)
\end{align*}
for any $t\ge0$. This means that ${\Phi}_t^{\rm stat}$ has a law $\mu^{(\alpha)}$ for any $t>0$.

Corollary \ref{maincor} is obtained as follows.
Since
$$
\int_{\mathcal{D}'(\Lambda)} \mathbb{P}\left(\Phi(\phi)\in C([0,T];H^{-\varepsilon})\right) \mu^{(\alpha)}({\dd}\phi)
=\mathbb{P}\left({\Phi}^{\rm stat}\in C([0,T];H^{-\varepsilon})\right)=1,
$$
we have
$$
\mathbb{P}\left(\Phi(\phi)\in C([0,T];H^{-\varepsilon})\right)=1
$$
for $\mu^{(\alpha)}$-almost every $\phi\in\mathcal{D}'(\Lambda)$. 
Since $\mu^{(\alpha)}$ and $\mu_0$ are absolutely continuous with respect to each other (Corollary \ref{cor:expmeas}), ``$\mu^{(\alpha)}$-almost every $\phi$" can be replaced by ``$\mu_0$-almost every $\phi$".

We now turn to proofs of (i) and (ii). The proofs go in very similar ways to \cite[Section 4]{HKK19}.

\begin{proof}[{\bfseries Proof of (i)}]
By the definition \eqref{eq:OU} of the OU process $X$, we can decompose $\widetilde{\Phi}^{N,{\rm stat}}=X(\xi_N)+{\bf Y}^N$, where ${\bf Y}^N$ solves
\begin{align*}
\label{eq:statY}
&
\left\{
\begin{aligned}
\partial _t {\bf Y}_t^N&= \frac 12 (\triangle-1) {\bf Y}_t^N - \frac\alpha2 P_N\exp_N^\diamond(\alpha \widetilde{\Phi}_t^{N,{\rm stat}}),\\
{\bf Y}_0^N&=0.
\end{aligned}
\right.
\end{align*}
For $X(\xi_N)$, 
it is easy to check that
$$
\sup_{N\in\mathbb{N}}\mathbb{E}\left[ \left\| X_0(\xi_N) \right\| _{H^{-\varepsilon}} \right]
+\sup_{N\in\mathbb{N}} \mathbb{E}\left[ \sup _{s,t \in [0,T]} \frac{\left\| X_t(\xi_N)  - X_s(\xi_N) \right\| _{H^{-\varepsilon}}}{|t-s|^\delta} \right] <\infty
$$
for sufficiently small $\delta ,\varepsilon>0$,
by the a priori estimate of the OU process (Proposition \ref{prop:aprioriOU}) and the uniform bound of 
Radon-Nikodym derivatives $\frac{{\dd}\mu_N^{(\alpha)}}{{\dd}\mu_0}$ (Corollary \ref{cor:expmeas}).
For ${\bf Y}^N$, by the Schauder estimate (Proposition \ref{prop:Schauder}), the invariance of $\mu_N^{(\alpha)}$ under $\widetilde{\Phi}^N$, and Corollary \ref{log-derivative}, for any small $\delta>0$ we have
\begin{align*}
\mathbb{E}\big [  \|{\bf Y}^N\|_{C^{\delta}([0,T];L^2)}^2 \big ]
&\lesssim
\mathbb{E}\Big [  \left\|P_N\left\{\exp_N^\diamond(\alpha \widetilde{\Phi}^{N,{\rm stat}})\right\}\right\|^2_{L^2([0,T];H^{-s})}
\Big ]
\\
&\le
\mathbb{E}\Big [  \left\|\exp_N^\diamond(\alpha \widetilde{\Phi}^{N,{\rm stat}})\right\|_{L^2([0,T];H^{-s})}^2
\Big ]
\\
&=\int_{\mathcal{D}'(\Lambda)}\mu_N^{(\alpha)}({\dd}\phi) \int_0^T
\mathbb{E}\Big [  \left\|\exp_N^\diamond(\alpha \widetilde{\Phi}_t^N(\phi))\right\|_{H^{-s}}^2\Big ]
{\hspace{0.5mm}}{\dd}t\\
&=T\int_{\mathcal{D}'(\Lambda)} \|\exp_N^\diamond(\alpha\phi)\|_{H^{-s}}^2 \mu_N^{(\alpha)}({\dd}\phi) 
\lesssim 1.
\end{align*}
Then by a similar argument to \cite[Theorem 4.2]{HKK19}, we have the tightness of $\{\widetilde{\Phi}^{N,{\rm stat}}\}_{N\in\mathbb{N}}$ in $C([0,T];H^{-\varepsilon})$.
\end{proof}

\begin{proof}[{\bfseries Proof of (ii)}]
By a similar argument to the proof of \cite[Theorem 1.3]{HKK19},
we can assume that $\xi_N$ converges to $\xi$ in $H^{-\varepsilon}$ almost surely.
Then we can complete the proof of (ii) by showing that
$$
\widetilde{\Phi}^{N,{\rm stat}} \to {\Phi}^{\rm stat}
$$
in $C([0,T];B_{p,p}^{-\varepsilon})$, in probability.
To do this, we decompose 
$\widetilde{\Phi}^{N,{\rm stat}} =X(\xi_N)+{\bf Y}^N$, as in the proof of (i),
and decompose ${\Phi}^{\rm stat} =X(\xi)+{\bf Y}$, where ${\bf Y}=\mathcal{S}(0,\mathcal{X}^\infty(\xi))$.
Since
\begin{align*}
X(\xi_N)\to X(\xi),&\qquad\text{in } C([0,T];H^{-\varepsilon})\quad\text{almsot surely},\\
\mathcal{X}^N(\xi_N)\to\mathcal{X}^\infty(\xi),&\qquad\text{in } L^p([0,T];B_{p,p}^{-\beta})\quad\text{in probability},
\end{align*}
by
\eqref{stabilityOU} of Proposition \ref{prop:aprioriOU}
and Lemma \ref{lem:contiexpOU}, we consider the solution $\Upsilon^N=\mathcal{S}_N(0,\mathcal{X}^N)$ of the deterministic initial value problem
\begin{align*}
\left\{
\begin{aligned}
\partial_t\Upsilon_t^N&=\frac12(\triangle-1)\Upsilon_t^N-\frac\alpha2 P_N\left(e^{\alpha P_N\Upsilon_t^N}\mathcal{X}_t^N\right),\\
\Upsilon_0^N&=0
\end{aligned}
\right.
\end{align*}
 for any nonnegative functions $\{\mathcal{X}^N\}_{N\in\mathbb{N}}\subset C([0,T]\times\Lambda)$.
Then, the proof completes, once we show that; if
$$
\mathcal{X}^N\to\mathcal{X}\qquad\text{in $L^p([0,T];B_{p,p}^{-\beta})$},
$$
then
$$
\mathcal{S}_N(0,\mathcal{X}^N)
\to\mathcal{S}(0,\mathcal{X})
\qquad\text{in $C([0,T];B_{p,p}^{\delta})$}.
$$
This is obtained by a similar way to Lemma \ref{lem:existence}.
Indeed, the a priori estimate \eqref{eq:aprioriUps} holds for $\Upsilon^N$ uniformly over $N$, since $\{ P_N\}$ are nonnegative and uniformly bounded as operators on $B_{p,p}^{-\beta}$, in view of  Hypothesis \ref{hypo on P}.
If $\{\Upsilon^{N_k}\}_{k\in\mathbb{N}}$ is a convergent subsequence, then the limit $\Upsilon$ solves the equation \eqref{eq:mildUps} as a consequence of the continuity of $P_N$ as $N\to\infty$, which is assumed by Hypothesis \ref{hypo on P}.
\end{proof}
%
\section{Relation with Dirichlet form theory} \label{sec:DF}
In this section, we prove Theorem \ref{mainthm3}. Although the proof goes 
in a very similar way to one in \cite{HKK19}, we provide a
sketch of the proof for readers' convenience.

We fix the parameter $s\in(0,1)$ appearing in Corollary \ref{log-derivative} and set $D={\rm Span}\{e_{k};k\in {\mathbb Z}^{2} \}$, $H=L^{2}$ and $E=H^{-s}$.
In what follows, $\langle \cdot, \cdot \rangle$ stands for the pairing of $E$ and its dual space
$E^{*}=H^{s}$.
By Corollary \ref{log-derivative}, the
map $\phi\mapsto\exp^\diamond(\alpha\phi)$ can be regarded as a 
${\mathcal B}(E)/{\mathcal B}(E)$-measurable map.
Let 
$({\mathcal E}, {\mathfrak F}C^{\infty}_{b})$ be the 
pre-Dirichlet form defined by (\ref{DF-intro}), that is,
\begin{equation*}
{\cal E}(F,G)=
\frac{1}{2} 
\int_{E} \big( D_{H}F(\phi ), D_{H}G(\phi ) \big)_{H}
 \mu ^{(\alpha )}({\dd}\phi),\qquad F,G\in {\mathfrak F}C_{b}^{\infty}.
\end{equation*}
Then we obtain the following:
\begin{prop} \label{Prop-IbP}
It holds that
\begin{equation}
{\cal E}(F,G)=-\int_{E} {\cal L}F(\phi) G(\phi) \hspace{0.2mm} \mu^{(\alpha)}({\dd}\phi),
\qquad F,G\in {\mathfrak F}C_{b}^{\infty},
\label{eq:Prop-IbP}
\end{equation}
where ${\mathcal L}F\in L^{2}(\mu^{(\alpha)})$ is given by
\begin{align*}
\mathcal{L}F (\phi )
&=\frac12\sum_{i, j=1}^n \partial_i \partial_jf(\langle \phi,{l_1}\rangle,\dots,\langle \phi,{l_n}\rangle) \langle l_{i}, l_{j} \rangle
\\
&\quad-\frac12\sum_{j=1}^n \partial_jf(\langle \phi,{l_1}\rangle,\dots,\langle \phi,{l_n}\rangle)
\cdot \big \{
\big\langle (1-\triangle)\phi, l_{j}
\big\rangle
+\alpha
\big\langle
\exp^\diamond(\alpha \phi), {l_j} \big\rangle
\big \}
\end{align*}
for $F(\phi)=f(\langle \phi,{l_1}\rangle,\dots,\langle \phi,{l_n}\rangle)$
with $f\in C^{\infty}_{b}(\mathbb R^{n}), l_{1}, \ldots l_{n} \in 
D$.
\end{prop}
\begin{proof}
Let $\psi = {\mathbf 1}_{[-1,1]^2}$, which satisfies Hypothesis \ref{hypo on psi}.
Applying the Gaussian integration by parts formula with respect to 
$\mu_{0}$ (see \cite[page 207]{GJ86}), we have
\begin{align*}
\int_{E}
& 
\big(D_{H}F(\phi),h \big)_{H} 
\exp \Big(-
\int_{\Lambda} \exp^{\diamond}_{N}(\alpha \phi)(x) {\hspace{0.5mm}}{\dd}x 
\Big)
\hspace{0.5mm}
\mu_{0}({\dd}\phi)
\\
&=
\int_{E}F(\phi)\Big ( \langle \phi, (1-\triangle)h \rangle
-\alpha \langle \exp_{N}^{\diamond}(\alpha \phi), P_N h \rangle
\Big )
\exp \Big(-
\int_{\Lambda} \exp_{N}^\diamond(\alpha \phi)(x) \hspace{0.5mm} {\dd}x 
\Big)
\hspace{0.5mm}
\mu_{0}({\dd}\phi)
\nonumber
\end{align*}
for all $F\in {\mathfrak F}C_{b}^{\infty}, h \in D$ and $N\in \mathbb N$.

Now we recall Theorem \ref{thm:expwick almostsure},
Corollary \ref{log-derivative} and 
$\lim_{N \to \infty}Z^{(\alpha)}_{N}=Z^{(\alpha)} >0$. Taking the limit $N\to \infty$ on 
both sides of the above equality, we obtain
\begin{equation*}
\int_{E}
\big(D_{H}F(\phi),h \big)_{H} 
\mu^{(\alpha)}({\dd}\phi)=
\int_{E}F(\phi)\Big ( \langle \phi, (1-\triangle)h \rangle
-\alpha \langle \exp^{\diamond}(\alpha \phi), h \rangle
\Big )
\mu^{(\alpha)}({\dd}\phi)
\end{equation*}
and this leads us to the desired integration by parts formula (\ref{eq:Prop-IbP}).
Besides, applying Corollary \ref{log-derivative} again, we obtain 
${\mathcal L}F\in L^{2}(\mu^{(\alpha)})$.
This completes the proof.
\end{proof}

Proposition \ref{Prop-IbP} implies that $({\mathcal E}, {\mathfrak F}C^{\infty}_{b})$ is closable on $L^{2}(\mu^{(\alpha)})$.
We denote the closure
by
$({\mathcal E}, {\mathcal D}({\mathcal E}))$.
As mentioned in Section 1.2, $({\mathcal E}, {\mathcal D}({\mathcal E}))$
is a quasi-regular Dirichlet form 
on $L^{2}(\mu^{(\alpha)})$, and thus we obtain an $E$-valued diffusion
process 
${\mathbb M}=
(\Theta, {\mathcal G}, ( {\mathcal G}_{t})_{t\geq 0},  (\Psi_{t})_{t\geq 0}, ({\mathbb Q}_{\phi} )_{\phi \in E})$
properly associated with $({\cal E}, {\cal D(E)})$. By recalling 
Corollary \ref{log-derivative}
and applying
\cite[Lemma 4.2]{AR91},
we have 
\begin{equation*}
{\mathbb E}^{{\mathbb Q}_{\phi}} \Big[
\int_{0}^{T} \big \Vert \exp^{\diamond}(\alpha \Psi_{t}) \big \Vert_{E} \, {\dd}t \Big ]<\infty, \qquad T>0, ~\mu^{(\alpha)}
\mbox{-a.e.}~\phi.
\label{Q-1}
\end{equation*}
In particular,
\begin{equation*}
{\mathbb Q}_{\phi} \Big( \int_{0}^{T} 
\big \Vert \exp^{\diamond}(\alpha \Psi_{t}) \big \Vert_{E} \, {\dd}t<\infty
\Big)=1, \qquad T>0,~ \mu^{(\alpha)}
\mbox{-a.e.}~\phi.
\label{Q-2}
\end{equation*}
Thus we are able to apply \cite[Lemma 6.1 and Theorem 6.2]{AR91}
and \cite[Theorem 13]{Ond04} as in \cite{HKK19}. As a result, 
there exists 
an $H$-cylindrical 
$( {\mathcal G}_{t})$-Brownian motion ${\mathcal W}=({\mathcal W}_{t})_{t\geq 0}$
defined on $(\Theta, {\mathcal G}, {\mathbb Q}_{\phi})$ such that
\begin{align}
{\Psi}_t&=
e^{t(\triangle-1)/2}\phi
-\frac\alpha2\int_0^te^{(t-s)(\triangle-1)/2}\exp^\diamond(\alpha{\Psi}_s) \, {\dd}s
\nonumber 
\\
&\quad
+\int_0^te^{(t-s)(\triangle-1)/2} 
\, {\dd} {\mathcal W}_s, \qquad t\geq 0, 
~{\mathbb Q}_{\phi}
 \mbox{-a.s.},~\mu^{(\alpha)}\mbox{-a.e. } \phi .
\nonumber
\end{align}


Now we are going to prove Theorem \ref{mainthm3}.
Precisely, we prove that the process $\Psi$ coincides with the strong solution $\Phi$ driven by the cylindrical Brownian motion ${\mathcal W}$.
We decompose $\Psi=\mathfrak{X}(\phi)+\mathfrak{Y}$, where
\begin{align}\label{eq:X+Y DF}
\begin{aligned}
\mathfrak{X}(\phi)_{t}&:=e^{t(\triangle-1)/2}\phi+\int_0^te^{(t-s)(\triangle-1)/2}
\, {\dd} \mathcal{W}_s,\\
\mathfrak{Y}_{t}&:=-\frac\alpha2\int_0^te^{(t-s)(\triangle-1)/2}\exp^\diamond(\alpha{\Psi}_s)
\, {\dd} s, \qquad t\geq 0.
\end{aligned}
\end{align}
From the Da Prato--Debussche trick as used in Section \ref{sec:wellposed}, it is sufficient to show that
$$
{\mathbb Q}_\phi\Bigl({\mathfrak Y}=\mathcal{S}(0,\exp^\diamond(\alpha{\mathfrak X}(\phi)))\Bigr)=1,
\qquad\text{$\mu^{(\alpha)}$-a.e. $\phi$.}
$$
We prepare the following lemma.

\begin{lem}\label{lem:wexplaw}
Assume that the mollifier $\psi$ satisfies Hypothesis {\rm{\ref{hypo on psi}}}.
Let $E_0$ be  the set of all $\phi\in E$ such that the convergence
$$
\exp^\diamond(\alpha\phi)
=\lim_{N\to\infty}\exp_N^\diamond(\alpha\phi)
$$
holds in $B_{p,p}^{-\beta}$. 
Then, for any $f\in H^{1+\varepsilon}$ and $\phi\in E_0$ such that $f+\phi\in E_0$, one has
$$
\exp^\diamond(\alpha(f+\phi))=\exp(\alpha f)\exp^\diamond(\alpha\phi).
$$
\end{lem}

\begin{proof}
Let $f\in H^{1+\varepsilon}$ and $\phi\in E_0$.
$P_Nf$ converges to $f$ in $H^{1+\varepsilon}$ by Hypothesis \ref{hypo on psi}.
Since $H^{1+\varepsilon}\subset C(\Lambda)$, $\exp(\alpha P_Nf)$ converges to $\exp(\alpha f)$ in $C(\Lambda)$.
Therefore, by Theorem \ref{thm:stableM},
\begin{align*}
\exp_N^\diamond(\alpha(f+\phi))=\exp(\alpha P_Nf)\exp_N^\diamond(\alpha\phi)
\xrightarrow{N\to\infty}\exp(\alpha f)\exp^\diamond(\alpha\phi)
\end{align*}
in $B_{p,p}^{-\beta}$.
If $f+\phi\in E_0$, $\exp_N^\diamond(\alpha(f+\phi))$ converges to $\exp^\diamond(\alpha(f+\phi))$.
From these convergences the assertion follows.
\end{proof}

\begin{proof}[{\bfseries Proof of Theorem \ref{mainthm3}}]
It is sufficient to check that $\mathfrak{Y}$ belongs to the space $\mathscr{Y}_T$ and solves the mild equation \eqref{eq:mildUps}.
By the invariance of $\mu^{(\alpha)}$ under ${\Psi}$ and Corollary \ref{log-derivative},
\begin{align*}
\int_E
{\mathbb E}^{{\mathbb Q}_{\phi}}
{\Big [}
\big \|\exp^\diamond(\alpha{\Psi}) \big \|_{L^2([0,T];H^{-s})}^2 
{\Big ]}
\mu^{(\alpha)}({\dd}\phi)
=\int_0^T
\, {\dd}t\int_E \big \|\exp^\diamond(\alpha\phi) \big \|_{H^{-s}}^2\mu^{(\alpha)}({\dd}\phi)
<\infty.
\end{align*}
In particular,
$$
{\mathbb Q}_\phi\Bigl (
\exp^\diamond(\alpha{\Psi})\in L^2([0,T];H^{-s})\Bigr )=1,
\qquad\text{$\mu^{(\alpha)}$-a.e. $\phi$.}
$$
By the Schauder estimate (Proposition \ref{prop:Schauder}),
\begin{align}\label{eq:Y in Y_T as}
{\mathbb Q}_\phi\Bigl (
\mathfrak{Y}\in L^2([0,T];H^{1+\kappa})\cap C([0,T];H^\kappa)
\Bigr )=1,
\qquad \mu^{(\alpha)}\mbox{-a.e.}\ \phi
\end{align}
for small $\kappa>0$.
Since $\alpha\mathfrak{Y}$ is nonpositive, 
we have
\begin{align*}
{\mathbb Q}_\phi
\big (
{\mathfrak Y}\in\mathscr{Y}_T\big )=1,
\qquad\text{$\mu^{(\alpha)}$-a.e. $\phi$.}
\end{align*}

Finally we show that $\mathfrak{Y}$ solves the mild equation \eqref{eq:mildUps} with $(\upsilon,\mathcal{X})=(0,\mathfrak{X})$. By the definition \eqref{eq:X+Y DF} of $\mathfrak{Y}$, it is sufficient to show that
\begin{align}\label{exp rule holds as}
{\mathbb Q}_\phi
\Big (
\exp^\diamond(\alpha{\Psi}_t)=e^{\alpha{\mathfrak Y}_t}\cdot\exp^\diamond(\alpha{\mathfrak X}_t),
\ \mbox{a.e.}\ t
\Big )=1,
\qquad \mu_0\mbox{-a.e.}\ \phi.
\end{align}
Recall the definition of the subset $E_0$ in Lemma \ref{lem:wexplaw}.
Then $\mu_0(E_0)=1$, so $\mu^{(\alpha)}(E_0)=1$ 
by the absolute continuity (see Corollary \ref{cor:expmeas}).
By using the invariance of $\mu^{(\alpha)}$ under ${\Psi}$,
\begin{align*}
\int_E{\mathbb E}^{{\mathbb Q}_{\phi}}
\left[\int_0^T\mathbf{1}_{E_0^c}({\Psi}_t)
\, {\dd}t \right]\mu^{(\alpha)}({\dd}\phi)
=\int_0^T 
\, {\dd}t \int_E\mathbf{1}_{E_0^c}(\phi)\mu^{(\alpha)}({\dd}\phi)
=T\mu^{(\alpha)}(E_0^c)=0.
\end{align*}
Similarly, by the invariance of $\mu_0$ under ${\mathfrak X}$,
\begin{align*}
\int_E
{\mathbb E}^{{\mathbb Q}_\phi}
\left[\int_0^T\mathbf{1}_{E_0^c}({\mathfrak X}_t)
\, {\dd}t
\right]\mu_0({\dd}\phi)
&=\int_0^T 
\, {\dd}t \int_E\mathbf{1}_{E_0^c}(\phi)\mu_0({\dd}\phi)
=T\mu_0(E_0^c)=0.
\end{align*}
From these equalities and \eqref{eq:Y in Y_T as}, we have
\begin{align*}
{\mathbb Q}_\phi
\Big(
{\Psi}_t\in E_0,\ {\mathfrak X}_t\in E_0,\ {\mathfrak Y}_t\in H^{1+\kappa},\ \text{a.e. $t$}
\Big )
=1
\end{align*}
for $\mu_0$-almost every $\phi$.
Therefore, Lemma \ref{lem:wexplaw} implies \eqref{exp rule holds as}.
\end{proof}
%
%

\renewcommand{\thesection}{\Alph{section}}

\setcounter{section}{0}

\section{Green functions and their approximation}\label{hypo on G app}
In this appendix, we provide some properties of Green functions and their approximation on the whole space and the torus.
In the end, we prove a proposition, which yields Proposition \ref{hypo on G}.

\subsection{Green function on the whole plane}
Recall that $\psi$ is a function satisfying Hypothesis \ref{hypo on psi}, $\psi_N=\psi(2^{-N}\cdot)$, and
\begin{align*}
G_{M,N}(x,y)
=\frac1{2\pi}\sum_{k\in\mathbb{Z}^2}\frac{\psi_M(k)\psi_N(k)}{1+|k|^2}{\bf e}_k(x-y),
\qquad M,N\in\mathbb{N}.
\end{align*}
We regard $G_{M,N}$ as a periodic function on $\mathbb{R}^2\times\mathbb{R}^2$, rather than a function on $\Lambda\times\Lambda$.
Then by the Poisson summation formula, we can write it as an infinite sum of decreasing functions
$$
G_{M,N}(x,y)=\sum_{k\in\mathbb{Z}^2}{K}_{M,N}(x-y+2\pi k),
\qquad
{K}_{M,N}
:=\frac1{2\pi}\mathcal{F}^{-1}\left(\frac{\psi_M \psi_N}{1+| \cdot |^2}\right).
$$
Hence we need to observe the behavior of ${K}_{M,N}$ for our purpose.
Setting $\rho_{M,N}=\frac1{2\pi}\mathcal{F}^{-1}\left(\psi_M\psi_N\right)$, we can write $K_{M,N}$ as a convolution 
$$
{K}_{M,N}(x)
=(1-\triangle_{\mathbb{R}^2})^{-1}\rho_{M,N}(x)=\int_{\mathbb{R}^2}{K}(x-y)\rho_{M,N}(y)
\, {\dd}y,
$$
where $\triangle_{\mathbb{R}^2}$ is the Laplacian on $\mathbb{R}^2$, and $K$ is the Green function of $1-\triangle _{{\mathbb R}^2}$.
\begin{prop}\label{prop:AGreenH1}
The function $K:\mathbb{R}^2\setminus\{0\}\to\mathbb{R}$ is positive and has the estimates
\begin{align}\label{prop:AGreenH11}
{K}(x)
\left\{
\begin{aligned}
&=-\frac1{2\pi}\log|x|+O(1), &|x|<1,\\
&\lesssim e^{-|x|/2}, &|x|\ge1.
\end{aligned}
\right.
\end{align}
\end{prop}
\begin{proof}
By the relation between the heat semigroup and the resolvent kernel, we have
\begin{equation}\label{eq:AGreenH99}
{K}(x) = \frac{1}{4\pi} \int _0^\infty \exp\left({-t-\frac{|x|^2}{4t}}\right) \frac{{\dd}t}t 
\nonumber
\end{equation}
for $x\neq 0$. Since the integral over $(0,|x|/2)$ and $(|x|/2,\infty)$ are equal in view of the change of variables by $s=|x|^2/4t$, we have
\begin{equation*}
{K}(x) = \frac{1}{2\pi} \int _{|x|/2}^\infty \exp\left({-t-\frac{|x|^2}{4t}}\right) \frac{{\dd}t}t 
= \frac{1}{2\pi} \int _1^\infty \exp\left({-\frac{|x|}2\left(t+\frac1t\right)}\right) \frac{{\dd}t}t.
\end{equation*}
Hence we observe the behavior of the function
\begin{align*}
g(r)=\int_1^\infty \exp\left({-\frac{r}2\left(s+\frac1s\right)}\right) \frac{{\dd}s}s,\qquad r\in(0,\infty).
\end{align*}
Since the integrand is bounded by $e^{-rs/2}$ on $s\ge1$, we have $g(r)\lesssim e^{-r/2}$ for $r\ge1$, so the latter part of \eqref{prop:AGreenH11} follows. 
To consider the estimate on $r<1$, we decompose
\begin{align*}
g(r)=\int_1^{1/r}\frac{{\dd}s}s
+\int_1^{1/r}\left\{\exp\left({-\frac{r}2\left(s+\frac1s\right)}\right)-1\right\} \frac{{\dd}s}s
+\int_{1/r}^\infty \exp\left({-\frac{r}2\left(s+\frac1s\right)}\right) \frac{{\dd}s}s.
\end{align*}
The first term is equal to $-\log r$. The other terms are $O(1)$, since
\begin{align*}
\int_1^{1/r}\left|\exp\left({-\frac{r}2\left(s+\frac1s\right)}\right)-1\right|\frac{{\dd}s}s
\le\int_1^{1/r}\frac{r}2\left(s+\frac1{s}\right)\frac{{\dd}s}s
\le\frac12,
\end{align*}
and
\begin{align*}
&\int_{1/r}^\infty \exp\left({-\frac{r}2\left(s+\frac1s\right)}\right)\frac{{\dd}s}s
\le\int_{1/r}^\infty r\exp\left({-\frac{r}2s}\right)
{\dd}s\le2e^{-1/2}.
\end{align*}
Thus we have the former part of \eqref{prop:AGreenH11}.
\end{proof}

Next we consider the convolution of ${K}$ and a function with sufficient decay.
\begin{lem}[{\cite[Lemma 4.1]{RV10}}]\label{lemma RV}
For any function $\rho$ on $\mathbb{R}^2$ such that 
\begin{align*}
|\rho(x)|\le C(1+|x|)^{-2-\gamma}
\end{align*}
for some $C>0$ and $\gamma>0$, one has
\begin{align*}
\sup_{|x|>1}\left|\int_{\mathbb{R}^2}|\rho(y)|\log\frac{|x|}{|x-y|}
\, {\dd}y\right|<\infty.
\end{align*}
\end{lem}

\begin{lem}\label{lem:K*rho}
Let $\rho$ be a function on $\mathbb{R}^2$ such that $\int_{\mathbb R^{2}}\rho(x)
\, {\dd}x=1$ and
\begin{align}\label{ineq:decay of rho}
|\rho(x)|\le C(1+|x|)^{-4-2\gamma}
\end{align}
for some $C>0$ and $\gamma>0$. Set $\rho_N=2^{2N}\rho(2^N\cdot)$ for $N\in\mathbb{N}$.
Then for any $|x|<1$ and $N\in\mathbb{N}$,
\begin{align}\label{prop:Kmn1}
K*\rho_N(x)&=-\frac1{2\pi}\log\left(|x|\vee2^{-N}\right)+O(1).
\end{align}
Moreover, for any $x\in\mathbb{R}^2$ and $N\in\mathbb{N}$,
\begin{align}
\label{prop:Kmn2}
|K*\rho_N(x)|&\lesssim |x|^{-2-\gamma}.
\end{align}
\end{lem}

\begin{proof}
First we prove \eqref{prop:Kmn1}. By Proposition \ref{prop:AGreenH1}, we can decompose
$$
K(x)=-\frac1{2\pi}\log(|x|\wedge1)+R(x),
$$
where $R$ is a bounded function with rapid decay as $|x|\to\infty$.
Since $R*\rho_N$ is bounded, it is sufficient to show that
\begin{align*}
\left(\rho_N*\log(|\cdot|\wedge1)\right)(x)=\log\left(|x|\vee2^{-N}\right)+O(1).
\end{align*}
We decompose
\begin{align}\label{proof:K*rho}
\left(\rho_N*\log(|\cdot|\wedge1)\right)(x)=\int_{\mathbb{R}^2}\rho_N(y)\log|x-y|
\, {\dd}y
-\int_{|x-y|>1}\rho_N(y)\log|x-y|
\, {\dd}y.
\end{align}
Since $|x|<1$, the second term of the right-hand side is bounded. Indeed, since $1<|x-y|\le1+|y|$,
\begin{align*}
\int_{|x-y|>1}\rho_N(y)\log|x-y|
\, {\dd}y
&\le \int_{\mathbb{R}^2} |\rho_N(y)|\log(1+|y|)
\, {\dd}y\\
&= \int_{\mathbb{R}^2} |\rho(z)|\log(1+2^{-N}|z|)
\, {\dd}z\\
&\le \int_{\mathbb{R}^2} |\rho(z)|\log(1+|z|)
\, {\dd}z<\infty.
\end{align*}
Consider the first term of the right-hand side of \eqref{proof:K*rho}. When $|x|>2^{-N}$, by Lemma \ref{lemma RV},
\begin{align*}
\int_{\mathbb{R}^2}\rho_N(y)\log|x-y|
\, {\dd}y
&=\log2^{-N}+\int_{\mathbb{R}^2}\rho(y)\log|2^Nx-y|
\, {\dd}y\\
&=\log2^{-N}+\log|2^Nx|+O(1)=\log|x|+O(1).
\end{align*}
When $|x|\le2^{-N}$, by the calculation
\begin{align*}
&\left|\int_{\mathbb{R}^2}\rho(y)\log|2^Nx-y|
\, {\dd}y\right|\\
&\le\left|\int_{|2^Nx-y|<1}\rho(y)\log|2^Nx-y|
\, {\dd}y\right|
+\left|\int_{|2^Nx-y|\ge1}\rho(y)\log|2^Nx-y|
\, {\dd}y\right|\\
&\le \left(\int_{\mathbb{R}^2}\rho^2(y)
\, {\dd}y\right)^{1/2}\left(\int_{|y|<1}\left(\log|y|\right)^2
\, {\dd}y\right)^{1/2}
+\int_{\mathbb{R}^2}\rho(y)\log(1+|y|)
\, {\dd}y<\infty,
\end{align*}
we have
\begin{align*}
\int_{\mathbb{R}^2}\rho_N(y)\log|x-y|
\, {\dd}y
=\log2^{-N}+O(1).
\end{align*}
Thus, we have \eqref{prop:Kmn1}.

Next we prove \eqref{prop:Kmn2}.
By Proposition \ref{prop:AGreenH1}, $K\in L^p(\mathbb{R}^2)$ for any $p\in[1,\infty)$ and
$$
\sup_{x\in\mathbb{R}^2}|x|^{2+\gamma}K(x)<\infty.
$$
Hence we have
\begin{align*}
|x|^{2+\gamma}|K*\rho_N(x)|
&\lesssim\int_{\mathbb{R}^2}|y|^{2+\gamma}K(y)|\rho_N(x-y)|
\, {\dd}y
+\int_{\mathbb{R}^2}|x-y|^{2+\gamma}K(y)|\rho_N(x-y)|
\, {\dd}y\\
&= \int_{\mathbb{R}^2}|y|^{2+\gamma}K(y)|\rho_N(x-y)|
\, {\dd}y
+\int_{\mathbb{R}^2}K(x-y)|y|^{2+\gamma}|\rho_N(y)|
\, {\dd}y\\
&\lesssim \int_{\mathbb{R}^2}|\rho_N(y)|
\, {\dd}y
+\left(\int_{\mathbb{R}^2} \left(|y|^{2+\gamma}|\rho_N(y)|\right)^q
\, {\dd}y\right)^{1/q}
\end{align*}
for any $q\in(1,\infty)$.
By the condition \eqref{ineq:decay of rho},
\begin{align*}
\int_{\mathbb{R}^2} \left(|y|^{2+\gamma}|\rho_N(y)|\right)^q
\, {\dd}y
=2^{-N(2+\gamma q)} \int_{\mathbb{R}^2} \left(|y|^{2+\gamma}|\rho(y)|\right)^q
\, {\dd}y
\lesssim 
\int_{\mathbb R^{2}}
(1+|y|)^{-(2+\gamma)q}
\, {\dd}y<\infty.
\end{align*}
Thus we have \eqref{prop:Kmn2}.
\end{proof}

\subsection{Green function on the torus}

We return to the proof of Proposition \ref{hypo on G}.

\begin{lem}\label{lem:psi barpsi}
Let $\psi$ be a function satisfying Hypothesis {\rm{\ref{hypo on psi}}}. 
Then there exists a smooth function $\bar\psi$ with the following properties:
\begin{itemize}
\item $\bar\psi$ satisfies Hypothesis {\rm{\ref{hypo on psi}}}.
\item $0\le\psi\le\bar\psi$ on $\mathbb{R}^2$.
\item For any $k\in\mathbb{N}^2$ there exists a constant $C_k$ such that
\begin{align}\label{barpsi decay}
|\partial^k\bar\psi(x)|\le C_k(1+|x|)^{-2-\kappa-|k|_1}
\end{align}
for any $x\in\mathbb{R}^2$, where $\kappa$ is a constant as in Hypothesis {\rm{\ref{hypo on psi}(ii)}} 
and $|k|_1:=|k_1|+|k_2|$ for each $k=(k_1,k_2)\in\mathbb{N}^2$.
\end{itemize}
\end{lem}

\begin{proof}
By Hypothesis \ref{hypo on psi}(ii),
$$
|\psi(x)|\le C(1+|x|)^{-2-\kappa}
$$
for some constants $C,\kappa>0$.
Then, there exists a radial smooth function $\bar\psi$ such that $0\le\psi\le\bar\psi$ on $\mathbb{R}^2$, $\bar\psi(x)=1$ on $x\in B(0,r)$ for some $r>0$, and
$$
\bar\psi(x)=C(1+|x|)^{-2-\kappa}
$$
on $x\in B(0,R)^c$ for some $R>r$.
Obviously, $\bar\psi$ satisfies all the required properties.
\end{proof}

Now we prove the following proposition, which yields Proposition \ref{hypo on G}.
The estimate \eqref{estimate:propA.G2} in the following proposition is better than \eqref{estimate:GN2}, because \eqref{estimate:propA.G2} is $L^p$-estimate for all $p\in [1,\infty )$.

\begin{prop}\label{propA.G}
Assume that $\psi$ satisfies Hypothesis {\rm{\ref{hypo on psi}}}.
Then for any $x,y\in\mathbb{R}^2$ with $|x-y|<1$ and any $M,N\in\mathbb{N}$,
\begin{align}\label{estimate:propA.G1}
G_{M,N}(x,y)&=-\frac1{2\pi}\log\left(|x-y|\vee2^{-M}\vee2^{-N}\right)+R_{M,N}(x,y),
\end{align}
where the remainder term $R_{M,N}(x,y)$ is uniformly bounded over $x,y,M,N$.
Moreover, for any $p\in [1,\infty )$, there exist constants $C>0$ and $\theta>0$ such that, for any $M,N\in\mathbb{N}$,
\begin{align}\label{estimate:propA.G2}
\iint_{\Lambda\times\Lambda}|G_{M,N+1}(x,y)-G_{M,N}(x,y)|^p 
\, {\dd}x
{\hspace{0.2mm}}{\dd}y\le C2^{-\theta N}.
\end{align}
\end{prop}

\begin{proof}
First, we prove \eqref{estimate:propA.G1}.
Let $M\le N$ without loss of generality. 
First we assume that $\psi$ satisfies \eqref{barpsi decay} in addition to Hypothesis \ref{hypo on psi}.
By \eqref{barpsi decay}, the function $\rho_0=\frac1{2\pi}\mathcal{F}^{-1}\psi$ satisfies that for all 
$n\in {\mathbb N}$
\begin{align*}
(1+|x|^2)^n\vert\rho_0(x)\vert
&=\frac1{2\pi}\left\vert\mathcal{F}^{-1}\left\{(1-\triangle)^n\psi\right\}(x)\right\vert\\
&\le\frac1{2\pi}\int_{\mathbb{R}^2}\left\vert (1-\triangle)^n\psi(\xi) \right\vert 
\, {\dd}
\xi<\infty.
\end{align*}
Recall that $\rho_{M,N}=\frac1{2\pi}\mathcal{F}^{-1}(\psi_M\psi_N)$.
Let $\rho_N=\frac1{2\pi}\mathcal{F}^{-1}\psi_N$. Since $\rho_N=2^{2N}\rho_0(2^N\cdot)$, we have
$$
\rho_{M,N}=\rho_M*\rho_N=2^{2M}(\rho_0*\rho_{N-M})(2^M\cdot).
$$
Let $\widetilde{\rho}_{M,N}=\rho_0*\rho_{N-M}$. 
Then from the above estimate of $\rho_0$, we have the uniform estimates
$$
|\widetilde{\rho}_{M,N}(x)|\lesssim(1+|x|)^{-6}.
$$
Indeed, for $|x|<1$, since $\widetilde{\rho}_{M,N}$ is uniformly bounded, this estimate is obvious. For $|x|\ge1$,
\begin{align*}
|x|^6
&
|\widetilde{\rho}_{M,N}(x)|\\
&\lesssim\int_{\mathbb{R}^2}|y|^6|\rho_0(y)||\rho_{N-M}(x-y)|
\, {\dd}y
+\int_{\mathbb{R}^2}|x-y|^6|\rho _0(y)||\rho_{N-M}(x-y)|
\, {\dd}y\\
&\lesssim\int_{\mathbb{R}^2}|\rho_{N-M}(y)|
\, {\dd}y+\int_{\mathbb{R}^2}|y|^6|\rho_{N-M}(y)|
\, {\dd}y\\
&\lesssim \int_{\mathbb{R}^2}\left(|\rho_0(y)|+|y|^6|\rho _0(y)|\right)
{\dd}y<\infty.
\end{align*}
Since $\int_{\mathbb R^{2}} \widetilde{\rho}_{M,N}(x)
\, {\dd}x=1$, $\widetilde{\rho}_{M,N}$ satisfies the conditions of Lemma \ref{lem:K*rho} with $\gamma=1$.
Therefore, the estimates \eqref{prop:Kmn1} and \eqref{prop:Kmn2} yield
\begin{align*}
G_{M,N}(x,y)=\sum_{k \in\mathbb{Z}^2}(K*\rho_{M,N})(x-y+2\pi k)
=-\frac1{2\pi}\log\left( |x-y| \vee2^{-M}\right)+O(1)
\end{align*}
for $|x-y|<1$.

Next let $\psi$ be an arbitrary function satisfying Hypothesis \ref{hypo on psi}.
Let $\bar\psi$ be a smooth function in Lemma \ref{lem:psi barpsi}, and define the 
function $\bar G_{M,N}$ similarly to $G_{M,N}$ with replacement $\psi$ by $\bar \psi$.
As shown above, $\bar{G}_{M,N}$ satisfies the estimate
$$
\bar G_{M,N}(x,y)=-\frac1{2\pi}\log\left(|x-y|\vee2^{-M}\right)+O(1).
$$
Since $0\le\psi\le\bar\psi$,
\begin{equation}\label{eq:note2020.07.02a}
\begin{aligned}
|G_{M,N}(x,y) - \bar G_{M,N}(x,y)| 
&\leq \frac{1}{4\pi ^2} \sum _{k\in {\mathbb Z}^2; \,|k|\leq 2^M} \frac{\bar \psi _M(k) \bar \psi _N(k) - \psi _M(k) \psi _N(k)}{1+|k|^2} \\
&\quad+ \frac{1}{4\pi ^2} \sum _{k\in {\mathbb Z}^2; \,|k|>2^M} \frac{\bar \psi _M(k) \bar \psi _N(k) - \psi _M(k) \psi _N(k)}{1+|k|^2} .
\end{aligned}
\end{equation}
Hypothesis \ref{hypo on psi}(iii) and the property of $\bar \psi$ imply that for sufficiently small $\zeta >0$,
\[
|\psi (x)-1| + |\bar \psi (x) -1| \leq C |x|^\zeta , \qquad x\in {\mathbb R}^2,
\]
with a positive constant $C$.
Hence, by the boundedness of $\psi$ and $\bar\psi$,
\begin{align*}
\sum _{k\in {\mathbb Z}^2;\, |k|\leq 2^M} &
\frac{\bar \psi _M(k) \bar \psi _N(k) - \psi _M(k) \psi _N(k)}{1+|k|^2} \\
&\lesssim \sum _{k\in {\mathbb Z}^2;\, |k|\leq 2^M} \frac{|\bar \psi (2^{-M} k) \bar \psi (2^{-N}k) -1| + |\psi (2^{-M}k) \psi (2^{-N}k) -1|}{1+|k|^2} \\
&\lesssim \sum _{k\in {\mathbb Z}^2;\, |k|\leq 2^M} \frac{2^{-M\zeta}|k|^\zeta}{1+|k|^2} 
\\
&
\lesssim1.
\end{align*}
Besides, since $0\le\psi\le\bar\psi$ and $\bar\psi(x)\lesssim(1+|x|)^{-2-\kappa}$,
\begin{align*}
\sum _{k\in {\mathbb Z}^2;\, |k|>2^M}
\frac{\bar \psi _M(k) \bar \psi _N(k) - \psi _M(k) \psi _N(k)}{1+|k|^2}
%
&\leq \sum _{k\in {\mathbb Z}^2;\, |k|>2^M} \frac{\bar \psi _M(k)}{1+|k|^2} \\
&\lesssim \sum _{k\in {\mathbb Z}^2;\, |k|>2^M} \frac{1}{(1+|k|^2)(1+|2^{-M}k|)^{2+\kappa}} \\
&\lesssim \sum _{k\in {\mathbb Z}^2;\, |k|>2^M} \frac{2^{(2+\kappa)M}}{(1+|k|)^{4+\kappa}}\\
&\lesssim 2^{(2+\kappa)M}\int_{|x|>2^M}\frac1{|x|^{4+\kappa}}
\, {\dd}x\\
&\lesssim1.
\end{align*}
These inequalities and \eqref{eq:note2020.07.02a} yield the estimate \eqref{estimate:propA.G1} for $G_{M,N}$.
%

Finally,
we prove \eqref{estimate:propA.G2}.
Let $p\in[1,\infty)$.
In view of the shift invariance of $G_{M,N}(\cdot,\cdot)$ and the compactness of $\Lambda$, it is sufficient to show
\begin{align}\label{estimate:GN2a}
\int _{\Lambda} \left| G_{M,N+1}(x,0) - G_{M,N}(x,0) \right| ^p 
\, {\dd}x \lesssim 2^{-\theta N}
\end{align}
for some $\theta>0$.
Hypothesis \ref{hypo on psi}(iii) implies that for sufficiently small $\zeta >0$,
\[
|\psi (x)-1| \leq C |x|^\zeta , \qquad x\in {\mathbb R}^2,
\]
with another positive constant $C$.
Then, by Plancherel's formula we have
\begin{align*}
\left\| G_{M,N+1} (\cdot,0) - G_{M,N} (\cdot,0) \right\| ^2_{H^{1-\zeta}}
&\lesssim \sum _{k\in {\mathbb Z}^2} \frac{|\psi _{N+1} (k) - \psi _N(k)|^2}{(1+|k|^2)^{2(1-\zeta)}} \\
&\lesssim \sum _{k\in {\mathbb Z}^2} \frac{|\psi (2^{-N-1}k) - 1|^2 + |\psi (2^{-N}k) - 1|^2}{(1+|k|^2)^{2(1-\zeta )}} \\
&\lesssim 2^{-2N\zeta} \sum _{k\in {\mathbb Z}^2} \frac{|k|^{2\zeta}}{(1+|k|^2)^{2(1-\zeta )}} \\
&\lesssim 2^{-2N\zeta} \sum _{k\in {\mathbb Z}^2} \frac{1}{(1+|k|^2)^{2-3\zeta}} \\
&\lesssim 2^{-2N\zeta}
\end{align*}
for sufficiently small $\zeta$.
Since the Sobolev embedding theorem implies $H^{1-\zeta} \subset L^p$ for $\zeta \leq 2/p$, by talking $\zeta$ sufficiently small we have the estimate \eqref{estimate:GN2a}.
\end{proof}

\subsection{Approximations by averaging}\label{sec:average}

We introduce a class of approximations of the Gaussian free field, which contains the circle average (see e.g. \cite{Ber17, Ber16, DS11}), and show that the associated kernels also satisfy \eqref{estimate:GN1} and \eqref{estimate:GN2} in Proposition \ref{hypo on G}.
This implies that our construction of Wick exponentials of the Gaussian free field in Section \ref{sec:Wick} includes the circle averaging approximation.

Let ${\mathbb X}$ be the Gaussian free field on $\Lambda = {\mathbb T}^2$ as defined in Section \ref{sec:Wickmain}, and extend ${\mathbb X}$ on ${\mathbb R}^2$ periodically.
Let $\nu$ be a probability measure on ${\mathbb R}^2$ supported in the unit ball $B(0,1)$ such that
\begin{equation}\label{ass:circaverage}
\sup _{|x|\leq 2}\int _{{\mathbb R}^2} |\log (x-y) |\nu ({\dd}y) < \infty .
\end{equation}
For $N\in {\mathbb N}$ denote by $\nu _N$ the measure given by $\nu _N(A)= \nu (2^N A)$ for a Borel set $A$.
Define the approximation ${\mathbb X}_N$ of ${\mathbb X}$ by
\[
{\mathbb X}_N (x) := {\mathbb X}*\nu _N (x) := \int _{{\mathbb R}^2} {\mathbb X}(x-y) \nu _N ({\dd}y) .
\]
Then the random field ${\mathbb X}_N$ has the covariance function
\[
G_{M,N}(x,y) := {\mathbb E}[{\mathbb X}_M(x){\mathbb X}_N(y)] = \frac{1}{2\pi} \sum _{k\in {\mathbb Z}^2} \frac{({\mathcal F}\nu _M)(k) ({\mathcal F}\nu _N)(k)}{1+|k|^2} {\bf{e}}_k(x-y)
\]
for $M,N\in {\mathbb N}$, where 
\[
({\mathcal F}\mu )(\xi ):= \frac{1}{2\pi} \int _{{\mathbb R}^2} e^{-\sqrt{-1}\xi \cdot x} \mu ({\dd}x), \quad \xi \in {\mathbb R}^2
\]
for a probability measure $\mu$.

\begin{prop}\label{prop:circaverage}
The sequence $G_{M,N}(x,y)$ defined as above satisfies \eqref{estimate:GN1} and \eqref{estimate:GN2} in Proposition \ref{hypo on G}.
\end{prop}

\begin{proof}
The estimate \eqref{estimate:GN1} is obtained in \cite[Lemma 3.5]{Ber17}.
We show the estimate \eqref{estimate:GN2}.
It is easy to see 
\begin{align*}
({\mathcal F}\nu _N)(k)&= ({\mathcal F}\nu )(2^{-N}k) , \quad k\in {\mathbb Z}^2,\\
|({\mathcal F}\nu )(\xi )| &\leq \frac{1}{2\pi} \int _{{\mathbb R}^2} \nu ({\dd}y), \quad \xi \in {\mathbb R}^2,\\
|({\mathcal F}\nu )(\xi _1) - ({\mathcal F}\nu )(\xi _2)| &\leq 
\frac{1}{2\pi} \int |e^{-\sqrt{-1} \xi _1\cdot y} - e^{-\sqrt{-1} \xi _2\cdot y}| \nu ({\dd}y) \\
&\lesssim |\xi _1 -\xi _2| \int _{|y|\leq 1} |y|\nu ({\dd}y), \quad \xi _1, \xi _2 \in {\mathbb R}^2.
\end{align*}
From these inequalities it follows that $\mathcal{F}\nu$ is bounded and $\zeta$-H\"older continuous for any $\zeta \in (0,1]$. Hence, the estimate \eqref{estimate:GN2} is obtained in the same way as the proof of \eqref{estimate:propA.G2} in Proposition \ref{propA.G}.
\end{proof}

\vspace{4mm}
\noindent
{\bf Acknowledgements.}
The authors are grateful to Professor Makoto Nakashima for helpful discussions on Gaussian 
multiplicative chaos,
and to Professor Ryo Takada for helpful comments on the references on Besov spaces.
They also thank anonymous two referees for helpful suggestions that improved the quality of the present paper.
This work was partially supported by JSPS KAKENHI Grant Numbers 
17K05300, 17K14204, 19K14556, 20K03639 and 21H00988.


\end{document}